\documentclass[12pt,english,reqno]{amsart}
\usepackage{amscd,amssymb,amsmath,amstext,amsthm,amsfonts,bbold,wasysym}
\usepackage[dvips]{graphicx}
\usepackage{lipsum}
\usepackage{mathrsfs}
\usepackage{appendix}
\usepackage[normalem]{ulem}
\usepackage{float}
\usepackage{hyperref}
\usepackage[dvipsnames]{xcolor}
\usepackage[ansinew]{inputenc}

\usepackage{a4wide}
\newcommand{\cv}{\vec c}

\newcommand{\im}{\operatorname{Im}}

\newtheorem*{theorem}{Theorem}

\newtheorem{maintheorem}{Theorem}

\newtheorem{T}{Theorem}[section]
\newtheorem{Corollary}[T]{Corollary}
\newtheorem{Proposition}[T]{Proposition}
\newtheorem{Lemma}[T]{Lemma}

\newtheorem{Remark}[T]{Remark}

\newtheorem{Definition}[T]{Definition}
\newtheorem{Example}[T]{Example}

\newtheorem{Claim}{Claim}
\newtheorem{Conjecture}{Conjecture}

\newcommand{\mo}{\operatorname{mod}}

\def \AA {{\mathbb A}}
\def \BB {{\mathbb B}}

\def \dd {{\mathbb{d}}}
\def \RR {{\mathbb R}}
\def \ZZ {{\mathbb Z}}
\def \NN {{\mathbb N}}

\def \TT {{\mathbb T}}

\def \EE {{\mathbb E}}

\def \XX {{\mathbb X}}

\def \YY {{\mathbb Y}}

\def \cl {\mathcal{L}}

\def \cm {\mathcal{M}}

\def \cq {\mathcal{Q}}
\def \cp {\mathcal{P}}
\def \cc {\mathcal{C}}
\def \ch {\mathcal{H}}

\def \cu {\mathcal{U}}
\def \co {\mathcal{O}}
\def \cn {\mathcal{N}}

\def \ca {\mathcal{A}}

\def \cv {\mathcal{V}}

\def \cw {\mathcal{W}}

\def \ii {{\mathbb i}}

\newcommand{\length}{\operatorname{length}}
\newcommand{\leb}{\operatorname{Leb}}
\newcommand{\dist}{\operatorname{dist}}

\newcommand{\supp}{\operatorname{supp}}
\newcommand{\per}{\operatorname{Per}}

\newcommand{\diam}{\operatorname{diameter}}

\newcommand{\cy}{\operatorname{Cyl}}

\newcommand{\Res}{\operatorname{Res}}

\begin{document}

\thanks{Work carried out at National Institute for Pure and Applied Mathematics (IMPA) and Federal University of
Bahia (UFBA). Partially supported by CNPq-Brazil (PDS 137389/2016-7 and BPP 304897/2015-9)}

\author{V Pinheiro}
\address{Instituto de Matematica - UFBA, Av. Ademar de Barros, s/n,
40170-110 Salvador Bahia}
\email{viltonj@ufba.br}

\date{\today}

\title{Lift and Synchronization}

\begin{abstract}
We study the problem of lifting a measure to an induced map $F(x)=f^{R(x)}(x)$. In particular, we give a necessary and sufficient condition for an ergodic $f$ invariant probability $\mu$ to be $F$-liftable as well as a condition for the lift to be an ergodic measure. Moreover, we show that every lift of $\mu$ is a weighted average of the restriction of $\mu$ to a countable number of $F$-ergodic components. We introduce the concept of a coherent schedule of events and relate it to the lift problem.
As a consequence, we prove that we can always synchronize coherent schedules at almost every point with respect to a given invariant probability $\mu$, showing that we can synchronize ``Pliss times'' $\mu$ almost everywhere. We also provide a version of this synchronization to non-invariant measures and, from that, we obtain some results related to Viana's conjecture \cite{Vi98} on the existence of Sinai-Ruelle-Bowen (SRB) measures for maps with non-zero Lyapunov exponents for Lebesgue almost every point.
\end{abstract}

\maketitle

\tableofcontents



\section{Introduction and statement of main results}\label{IntroductionAndStatementsOfMainResults}

Let $f:\XX\to\XX$ be a measurable map defined on a measurable space $(\XX,\mathfrak{A})$.
In the study of the forward evolution of the orbit of a point $p\in\XX$, it is a common strategy to analyze its orbit at special moments $\cu(p)\subset\NN$.
These moments can be selected in order to condense  information (for instance, when one is using a {\em Poincaré map} on a cross section of a flow) or to emphasize some particular property that we are interested in.
A concrete example of the last situation are the {\em hyperbolic times} in the study of non-uniformly hyperbolic dynamics.
Hyperbolic times are a particular case of Pliss times defined as follows. Consider an additive $f$-cocycle $\varphi:\NN\times\XX\to\RR$, i.e., $\varphi$ is measurable and $\varphi(n+m,p)=\varphi(n,f^m(p))+\varphi(m,p)$. 
Given $\gamma\in\RR$, we say that $n\ge1$ is a {\bf\em $(\gamma,\varphi)$-Pliss time} for $p\in\XX$, with respect to $f$, if
$$
  \frac{1}{n-k}\varphi(n-k,f^k(p))\ge\gamma\text{ for every }0\le k<n.
$$
According to Pliss Lemma (see Lemma~\ref{PlissLemma}), if $\limsup_{n\to\infty}\frac{1}{n}\varphi(n,p)>\lambda$ then $f$ has positive frequency of $(\lambda,\varphi)$-Pliss times. That is, if $\limsup_{n\to\infty}\frac{1}{n}\varphi(n,p)>\lambda$ then the upper natural density of moments with Pliss time is positive, where the {\bf\em upper natural density} of $U\subset\NN$ is $$\dd_{\NN}^+(U):=\limsup_{n\to\infty}\frac{1}{n}\#\big(\{1,2,3,\cdots,n\}\cap U\big).$$

To formalize the idea of selected times, define a {\bf\em schedule of events} as a measurable map $\cu:\XX\to2^{\NN}$, where $2^{\NN}$ is the power set of the natural numbers $\NN=\{1,2,3,\cdots\}$ with the metric $\dist(A,B)=2^{-\min(A\triangle B)}$ (see Section~\ref{SecPlissESy} for more details).
Suppose that one has two collections of selected moments, that is, two schedule of events $\cu$ and $\cv$, each one representing good moments of different properties that we need to analyze simultaneously.
Given a point $p\in\XX$, the central problem of the present paper is to know if the intersection of the two schedules at $p$ is a statistical significant subset of $\NN$, that is, if the upper natural density of $\cu(p)\cap\cv(p)$ is nonzero.

In many problems, it is possible to have good information about the state of a point $p$ at time $t$ by knowing the state of $p$ at a time $t+\ell$ for a finite fixed $\ell\ge0$.
That is, a fixed displacement $\ell$ on one of the schedules may be acceptable.
Thus, we can consider the problem of finding $\ell\ge0$ such that
\begin{equation}\label{Equationgbnj567}
  \dd_{\NN}^+(\{j\in\NN\,;\,(j,j+\ell)\in\cu(p)\times\cv(p)\})>0.
\end{equation}

We say that $\cu(p)$ and $\cv(p)$ are {\bf\em synchronizable} if (\ref{Equationgbnj567}) is true for some $\ell\ge0$.
Clearly, each schedule to have positive upper natural density is a necessary condition to (\ref{Equationgbnj567}), but it is not a sufficient one, as can it be seen in Example~\ref{ExampleNonSyn}. 
Because of that, we introduce the idea of {\em coherence}.
A {\bf\em coherent schedule of events} is a measurable map $\cu:\XX\to 2^{\NN}$ such that
\begin{enumerate}
	\item $n\in\cu(x)\implies n-j\in\cu(f^j(x))$ for every $x\in\XX$ and $n>j\ge0$;
	\item $n\in\cu(x)$ and $m\in\cu(f^n(x))$ $\implies$  $n+m\in\cu(x)$ for every $x\in\XX$ and $n,m\ge1$.
\end{enumerate}

Note that if we define $\cu(x)$ as the set of  $(\lambda,\varphi)$-Pliss times for $x$, then $\cu(x)$ satisfies both conditions above. Furthermore, by Pliss Lemma, $\dd_{\NN}^+(\cu(x))>0$ whenever $\limsup\frac{1}{n}\varphi(n,x)>\lambda$.
That is, Pliss times generates coherent schedule of events with positive upper natural density. In Section~\ref{SecPlissESy},
we show that it is easy to produce coherent schedule of events, even without using Pliss times.

A coherent schedule of events $\cu$ yields  to an induced map $F(x)=f^{\min\cu(x)}(x)$ with special properties. In particular, $F$ is {\em orbit-coherent} (see Definition~\ref{DefinitionORBCOR}).  
Hence, in order to obtain the synchronization, in the presence of an invariant probability, we study induced maps (Section~\ref{IndCoheInvSets} and \ref{MeIndMaEr}) and obtain a suitable characterization of liftable measures (Section~\ref{LiftRes}).

A $f$ invariant probability $\mu$ is liftable by a induced map $F(x)=f^{R(x)}(x)$ if there is a $F$ invariant probability $\nu\ll\mu$ such that $\mu$ is the normalized push-forward of $\nu$, i.e., $\mu=\frac{1}{\int R d\nu}\pi_*\nu$, where $\pi_*\nu=\sum_{j=0}^{+\infty}f^j_*(\nu|_{\{R>j\}})$ (see Section~\ref{LiftRes} for details).
If $\nu$ is a $F$-lift of $\mu$, then we can use $\nu$ and $F$ to study the statistical behavior of the orbits of $\mu$-almost every point with respect to the original map $f$.
In Theorem~\ref{TheoremLift} below we obtain three equivalent conditions for an invariant probability to be liftable.
Furthermore, the condition (2) in Theorem~\ref{TheoremLift} is automatically satisfied by most hyperbolic measures and induced maps generated in the context of non-uniformly hyperbolic dynamics.

\begin{maintheorem}
\label{TheoremLift}
Let $(\XX,\mathfrak{A})$ be a measurable space, $f:\XX\to\XX$ a bimeasurable map and $F:A\to\XX$ a measurable $f$-induced map with induced time $R:A\to\NN$.
Let $A_0:=\bigcap_{j\ge0}F^{-j}(\XX)$.
If $\mu$ is an ergodic $f$-invariant probability, then the four condition below are equivalent.
\begin{enumerate}
	\item $\mu$ is $F$-liftable.
	\item $\mu\big(\big\{x\in A_0\,;\,\dd_{\NN}^+(\{j\ge0\,;\,f^j(x)\in\co_F^+(x)\})>0\big\}\big)>0$.
	\item $\mu\big(\big\{x\in A_0\,;\, \liminf_{n}\frac{1}{n}\sum_{j=0}^{n-1}R\circ F^{j}(x)<+\infty\big\}\big)>0$.
	\item There is a $F$-invariant probability $\nu$ and $1\le C<+\infty$ such that $\nu\le C\mu$ and $\int R d\nu<+\infty$.
\end{enumerate}
Furthermore, if $F$ is orbit-coherent then $\mu$ admits one and only one $F$-lift (and this $F$-lift is $F$-ergodic).
\end{maintheorem}

We point out that one of the applications of induced maps is on the study of Thermodynamical Formalism of systems with some hyperbolicity (see, for instance, \cite{CP,CLP,PSZ,PS}) and it has been common to use the integrability of the induced time by the $f$-invariant probability $\mu$ to assure that $\mu$ is liftable (see Theorem 4.10 at \cite{Zw}).
Nevertheless, the integrability  criterium may leave out a relevant set of liftable invariant measures. 
Indeed, in Example~\ref{Examplejg6f7tv} below we show a very simple induced map $F(x)=\sigma^{R(x)}(x)$ of the shift $\sigma:\Sigma_2^+\to\Sigma_2^+$ having an uncountable set $M$ of liftable measures $\mu$ with full support, big entropy and non integrable induced time. 

\begin{Example}\label{Examplejg6f7tv}
	Consider the lateral shift of two symbols $\sigma:\Sigma_2^+\to\Sigma_2^+$,
where $\Sigma_2^+=\{0,1\}^\NN$ with the product topology and the usual metric $d(x,y)=2^{-\min\{j\ge1\,;\,x(j)\ne y(j)\}}$,
 $x=(x(1),$ $x(2),$ $x(3),$ $\cdots)$, $y=(y(1),$ $y(2),$ $y(3),$ $\cdots)$, $x(j)$ and $y(j)\in\{0,1\}$ $\forall\,j\ge0$.
 Given $(a_1,\cdots,a_{n})\in\{0,1\}^n$, define the cylinder  
 $$C_+(a_1,\cdots,a_{n})=\big\{(x(1),x(2),x(3),\cdots)\in\Sigma_2^+\,;\,x(1)=a_1,\cdots,x(n)=a_{n}\big\}.$$
Let $F:\Sigma_2^+\setminus\{\overline{0}\}\to\Sigma_2^+$ be the induced map $F(x)=\sigma^{R(x)}(x)$ with induced time  
\begin{equation}\label{EquatioInduchgf}
   R(x)= n \text{ for any  }x\in C_+(\underbrace{0,\cdots,0}_{\;n-1\text{\, times}},1)\text{ and }n\ge1.
\end{equation}

\begin{Lemma}[See proof in Appendix~\ref{AppendixAuxResAndPro}]\label{Lemma222jhf76h3} Let $F$ be the induced map above.
Then, there is an uncountable set $M$ of ergodic $\sigma$ invariant probabilities $\mu$ such that $\mu $ is $F$-liftable, $h_{\mu}(\sigma)>0$, $\supp\mu=\Sigma_2^+$ and $\int R d\mu=+\infty$. Furthermore, $\sup\{h_{\mu}(\sigma)\,;\,\mu\in M\}=\log 2=h_{top}(\sigma)$.
\end{Lemma}
\end{Example}

Notice that the condition of $f$ to be bimeasurable (not only measurable) that appears in most of the results of this paper is not a big restriction. Indeed, by Purves's result\footnote{\begin{theorem}[Purves \cite{Pu}, see also \cite{M81}]\label{TheoremPurves}
	Let $\XX$ and $\YY$ be complete separable metric spaces. Let $A\subset\XX$ be a Borel set of $\XX$ and $T:A\to\YY$ be a (Borel) measurable map. Then, $T$ is bimeasurable if and only if $\{y\in\YY\,;\,T^{-1}(y)$ is uncountable$\}$ is a countable set.
\end{theorem}},
if $\XX$ is a complete separable metric space and $f:\XX\to\XX$ is a countable-to-one measurable map, then $f$ is bimeasurable.

Theorem~\ref{TheoremFiftErgDec} below shows that any $F$-lift of $\mu$ is a weighted average of the restriction of $\mu$ to (at most) a countable number of $F$-ergodic components.

\begin{maintheorem}[Lift ergodic decomposition]\label{TheoremFiftErgDec}
Let $(\XX,\mathfrak{A})$ be a measurable space, $f:\XX\to\XX$ a bimeasurable map and $F:A\to\XX$ a measurable $f$-induced map with induced time $R$.
If $\mu$ is an ergodic $f$-invariant probability that is $F$-liftable, then there exist a countable collection $\{\nu_n\}_n$ of ergodic $F$-invariant probabilities and constants  $0<c_n<+\infty$ such that $\nu_n\le c_n\mu$. Furthermore, each $F$-lift $\nu$ of $\mu$ is uniquely written as $\nu=\sum_{n}a_j\nu_n$, with $\sum_{n} a_n=1$ and $0\le a_n\le1$ $\forall n$.	
In addition,   
if $\int R d\mu<+\infty$ then  $\{\nu_n\}_n$ is a finite collection.
\end{maintheorem}

Using induced maps, we were able to show that  synchronization of coherence schedules of events is always possible in the presence of an invariable measure.
Indeed, Theorem~\ref{TheoremSynINV} shows that the synchronization (up to a fixed displacement) of coherent schedules of events with positive upper natural density is always possible at $\mu$ almost every point and every ergodic invariant probability $\mu$.

\begin{maintheorem}[Synchronization for invariant measures]\label{TheoremSynINV}
Let $(\XX,\mathfrak{A})$ be a measurable space and $f:\XX\to\XX$ a bimeasurable map.
If $\mu$ is an ergodic $f$-invariant probability and $\cu_0,\cdots,\cu_m:\XX\to2^{\NN}$ is a finite collection of coherent schedules of events for $f$ such that  $\dd_{\NN}^+(\cu_j(x))>0$ for $\mu$ almost every $x\in\XX$ and every $0\le j\le m$, then there are $\ell_1,\cdots, \ell_{m}\ge0$ and $\theta>0$ such that 
\begin{equation}\label{Equatiojhu76r4}
  \lim_{n\to+\infty}\frac{1}{n}\#\big\{1\le j\le n\,;\,(j,j+\ell_1,\cdots,j+\ell_{m})\in\cu_0(x)\times\cdots\times\cu_m(x)\big\}=\theta,
\end{equation}
for $\mu$ almost every $x\in\XX$.
\end{maintheorem}

In Appendix~\ref{SecHYpTIm} we use Theorem~\ref{TheoremSynINV} to prove the existence of SRB measures for a class of examples of skew-products with critical region.

This paper has two goals. The first one is to synchronize coherent schedules for typical points with respect to an invariant probability.
The second goal is to obtain some advance on the synchronization of Pliss times associated to Lyapunov exponents for set of points with positive Lebesgue measure, without assuming the existence of a SRB measure.
The first goal was motivated by problems in ergodic theory and  thermodynamical formalism associated to non-uniformly hyperbolic dynamics.
In Appendix~\ref{SectionExamplesOfTheoApplications} we give some theoretical applications of the results above.
The second goal mentioned above was motivated by Viana's conjecture about non-zero Lyapunov exponents (see \cite{Vi98} and Conjecture 12.37 of \cite{BDV}).
\begin{Conjecture}[Viana]
	If a smooth map $f$ has only non-zero Lyapunov exponents at Lebesgue almost every point, then it admits some SRB measure.
\end{Conjecture}

Here we are focused on the case when $f$ has only positive Lyapunov exponents at Lebesgue almost every point.
Assume also that $f:M\to M$ is a $C^1$ local diffeomorphism on a compact Riemannian manifold $M$. 

According to Oseledets Theorem, there is a set $U\subset M$ with total probability (i.e., $\mu(U)=1$ for every invariant probability $\mu$) such that $\lim_n\frac{1}{n}\log|Df^n(x)v|$ exists for every $v\in T_xM\setminus\{0\}$ and $x\in U$.
The limit $\lim_n\frac{1}{n}\log|Df^n(x)v|$ is called the {\bf\em Lyapunov exponent of $x$ on the direction $v$}. 

If $x\notin U$ and $\lim_n\frac{1}{n}\log|Df^n(x)v|$ does not exist, then one can use $\limsup$ or $\liminf$ to define the Lyapunov exponents.
As we are interested in positive Lyapunov exponents, the  convenient assumption is $\liminf_n\frac{1}{n}\log|Df^n(x)v|>0$.
It follows from the continuity of the map $T_x^1M:=\{v\in T_xM\,;\,|v|=1\}\ni v\mapsto\frac{1}{n}\log|Df^n(x)v|$  and the compactness of $T_x^1M$ that $$\liminf_{n\to+\infty}\frac{1}{n}\log|Df^n(x)v|>0\;\;\forall\,v\in T_x^1M\implies\limsup_{n\to+\infty}\frac{1}{n}\log\|(Df^n(x))^{-1}\|^{-1}>0.$$
Therefore, we say that $x$ {\bf\em has only positive Lyapunov exponents} when
\begin{equation}\label{EquationLyaExpDef}
  \limsup_{n\to+\infty}\frac{1}{n}\log\|(Df^n(x))^{-1}\|^{-1}>0.
\end{equation}

Hence, a version of Viana's conjecture for a local diffeomorphism on a compact manifold having only positive Lyapunov exponents is the following.

\begin{Conjecture}[Viana]\label{ConjectureVianExpanding}
Let $M$ be a compact Riemannian manifold and $f:M\to M$ a $C^{1+}$ local diffeomorphism. If (\ref{EquationLyaExpDef}) holds for Lebesgue almost every $x\in M$ then 
$f$ admits some absolutely continuous invariant measure (in particular, a SRB measure).
\end{Conjecture}

One year after publishing his conjecture, Viana, in a joined work with J. Alves and C. Bonatti \cite{ABV}, instead of using the Lyapunov exponents condition (\ref{EquationLyaExpDef}),  assumed that
\begin{equation}\label{EquationSynNUEDEForte}
\liminf_{n\to+\infty}\frac{1}{n}\sum_{j=0}^{n-1}\log\|(Df(f^j(x)))^{-1}\|^{-1}>0
\end{equation}
for Lebesgue almost every point and they proved that Conjecture~\ref{ConjectureVianExpanding} is true when (\ref{EquationLyaExpDef}) is replaced by (\ref{EquationSynNUEDEForte}).  
In \cite{Pi06}, the author was able to weaken condition (\ref{EquationSynNUEDEForte}) to 
\begin{equation}\label{EquationSynNUEDEF}
\limsup_{n\to+\infty}\frac{1}{n}\sum_{j=0}^{n-1}\log\|(Df(f^j(x)))^{-1}\|^{-1}>0
\end{equation} and still obtain the existence of an absolutely continuous invariant measure.
Notice that (\ref{EquationSynNUEDEForte}) $\implies$ (\ref{EquationSynNUEDEF}) $\implies$ (\ref{EquationLyaExpDef}), as $\log\|(Df^n(x))^{-1}\|^{-1}\ge\sum_{j=0}^{n-1}\log\|(Df(f^j(x)))^{-1}\|^{-1}$.
The condition (\ref{EquationSynNUEDEForte}), and a posteriori (\ref{EquationSynNUEDEF}), came to be known as {\bf NUE} (non-uniformly expanding) condition and the dynamical systems satisfying (\ref{EquationSynNUEDEForte}) or (\ref{EquationSynNUEDEF}) on a set of points with positive Lebesgue measure came to be known as Non-uniformly Expanding Dynamics.
Nevertheless, there are many authors that refers to non-uniformly expanding dynamics when there is a set of points with positive Lebesgue measure and only positive Lyapunov exponents.

We observe that {\bf\em synchronized} {\bf NUE} would be a more appropriate name to conditions (\ref{EquationSynNUEDEForte}) or  (\ref{EquationSynNUEDEF}), letting  ``NUE'' for the condition (\ref{EquationLyaExpDef}).
To see that, let
 $T^1M=\bigcup_{x\in M}T_x^1M\subset TM$ be the {\em unit tangent bundle}. Note that $\log|Df^n(x)v|$ is an additive cocycle for the auxiliary skew-product 
 $F:T^1M\to T^1M$ given by  $$F(x,v)=\bigg(f(x),\frac{Df(x)v}{|Df(x)v|}\bigg).$$
Indeed, as $F^n(x,v)=\big(f^n(x),\frac{Df^n(x)v}{|Df^n(x)v|}\big)$, taking $h(x,v)=\log|Df(x)v|$, we get that  
$$h(F^j(x,v))=h\bigg(f^j(x),\frac{Df^j(x)v}{|Df^j(x)v|}\bigg)=\log\bigg|Df(f^j(x))\frac{Df^j(x)v}{|Df^j(x)v|}\bigg|=$$
$$=\log|Df^{j+1}(x)v|-\log|Df^j(x)v|$$
and so, we conclude that $\psi:\NN\times T^1M\to\RR$, defined by
\begin{equation}\label{Equationug488uh}
  \psi(n,(x,v)):=\log|Df^n(x)v|=\sum_{j=0}^{n-1}h\circ F^j(x,v),
\end{equation}
is an additive $F$-cocycle.
Thus, if $\gamma= \limsup_{n}\frac{1}{n}\log\|(Df^n(x))^{-1}\|^{-1}>0$, define $\cl(x,v)$ as the set of all $(\gamma/2,\psi)$-Pliss times for $(x,v)\in T^1M$.
Note that, for $v\ne u\in T_x^1M$, we don't know if there exists some $\ell\ge0$ such that $$\dd_{\NN}^+(\{j\in\NN\,;\,(j,j+\ell)\in\cl(x,v)\times\cl(x,u)\})>0.$$ Nevertheless, assuming that 
$\gamma=\limsup_n\frac{1}{n}\sum_{j=0}^{n-1}\log\|(Df(f^j(x)))^{-1}\|^{-1}>0$, it follows from Pliss Lemma that there exists $\theta>0$ such that $\dd_{\NN}^+(\cq(x))\ge\theta$, where $\cq(x)$ is the set of all $(\gamma/2,\Psi)$-Pliss times with $\Psi:\NN\times\XX\to\RR$ being the $f$ additive cocycle \begin{equation}\label{Equationfgyuygh}
  \Psi(n,x)=\sum_{j=0}^{n-1}\log\|(Df(f^j(x)))^{-1}\|^{-1}.
\end{equation}
 Thus, it follows from the fact that $\psi(n,(x,v))\ge\Psi(n,x)$ for every $(x,v)\in T^1M$ that $\cq(x)\subset\bigcap_{v\in T_x^1M}\cl(x,v)$. In particular,  
$$\dd_{\NN}^+(\cl(x,v)\cap\cl(x,u))\ge\theta>0$$
for every $v,u\in T_x^1M$.
This means that condition (\ref{EquationSynNUEDEForte}) or (\ref{EquationSynNUEDEF}) implies the synchronization (with $\ell=0$) of the Pliss times associated to the Lyapunov exponents of a point $x\in M$ on any given pair of directions $v,u\in T_x^1M$, as we had claimed.

Although specific, the examples in Appendix~\ref{SecHYpTIm} give a flavor of how to use synchronization to produce SRB measures (see Theorem~\ref{TheoremExemploSyEnd}).
In those examples, the existence of an invariant measure at the base of the skew-product allows us to use Theorem~\ref{TheoremSynINV} to assure the syntonization of the Pliss times associated to the Lyapunov exponents and, as a consequence, obtain a SRB measure.

In Section~\ref{SectionPointwiseSynchronization}, we study conditions to obtain a pointwise synchronization of  sup-additive cocycles.
In particular, letting $\psi, \cl, \Psi$ and $\cq$ be as above, we study conditions on the orbit of a point $p$, $\co_f^+(p)$, to assure that we can synchronize all the $\cl(p,v)$ with $v\in T_x^1M$.
That is, a condition on $\co_f^+(p)$ to obtain $\dd_{\NN}^+(\bigcap_{v\in T_x^1 M}\cl(p,v))>0$, without assuming $\dd_{\NN}^+(\cq(p))>0$.
With the results of Section~\ref{SectionPointwiseSynchronization}, for a map having only positive Lyapunov exponents almost everywhere, we can give a necessary and sufficient condition to the existence of SRB measures (see Theorem~\ref{MainTHEOVianaConjDiffeo} below).

As in Conjecture~\ref{ConjectureVianExpanding}, suppose that $f:M\to M$ is a $C^{1+}$ local diffeomorphism, $M$ is a compact Riemannian manifold and that  (\ref{EquationLyaExpDef}) holds for every point $x$ in a set $U$ with full Lebesgue measure, i.e., $\leb(M\setminus U)=0$.
Define {\bf\em Lyapunov residue} of $x\in U$ as $$\Res(x)=
	\lim_{n\to\infty}\dd_{\NN}^+(\{j\ge 0\,;\, R\circ f^j(x)\ge n\}),
$$ where
\begin{equation}\label{Equationnyhiygf}
  R(x)=
\min\bigg\{j\ge1\,;\,\frac{1}{j}\log\|(Df^j(x))^{-1}\|^{-1}>\frac{1}{2}\limsup_{n\to+\infty}\frac{1}{n}\log\|(Df^n(x))^{-1}\|^{-1}\bigg\}
\end{equation}

is the first time that $x$ ``reaches half of its limit (\ref{EquationLyaExpDef})''.

Note that $\Res(x)=0$ for every ergodic invariant probability $\mu$ with $\mu(U)>0$.
Indeed, as $\mu$ is ergodic, $\mu(U)>0$ implies that (\ref{EquationLyaExpDef}) holds for almost every $x$.
Hence, $\sum_{j=0}^{+\infty}\mu(\{R=j\})=1$. By Birkhoff,
$$\dd_{\NN}^+(\{j\ge 0\,;\, R\circ f^j(x)\ge n\})=\mu(\{R\ge n\})$$ for almost every $x$ and so,
$$\Res(x)=\lim_n\mu(\{R\ge n\})=\lim_n\sum_{j=n}^{+\infty}\mu(\{R=j\})=0.$$

Therefore, the Lyapunov residue being zero on a set of positive Lebesgue measure is a necessary condition to the existence of a SRB measure for $f$. In Theorem~\ref{MainTHEOVianaConjDiffeo}, we show that Lyapunov residue being zero on a set of positive Lebesgue measure is also a sufficient condition to the existence of a SRB measure.

\begin{maintheorem}
	\label{MainTHEOVianaConjDiffeo}
	Let $M$ be a compact Riemannian manifold and $f:M\to M$ be a $C^{1+}$ local diffeomorphism such that Lebesgue almost every point of $M$ has only positive Lyapunov exponents.
	Then, there exists an ergodic absolute continuous invariant measure if and only if the Lyapunov residue is zero on a set of positive measure.
\end{maintheorem}

\begin{Corollary}
		Let $M$ be a compact Riemannian manifold and $f:M\to M$ be a $C^{1+}$ local diffeomorphism such that Lebesgue almost every point of $M$ has only positive Lyapunov exponents.
	If $\limsup\frac{1}{n}\sum_{j=0}^{n-1}R\circ f^j(x)<+\infty$ for a positive measures set of points $x\in M$, where $R(x)$ is given by (\ref{Equationnyhiygf}), then $f$ admits an ergodic absolute continuous invariant probability.
\end{Corollary}
\begin{proof}
	Note that, if $\Res(x)>0$ then $\limsup_n\frac{1}{n}\sum_{j=0}^{n-1}R\circ f^j(x)=+\infty$. Indeed, if $\delta=\Res(x)>0$ then  for any given $m\in\NN$ there is a sequence $n_j\to+\infty$ such that $$\frac{1}{n_j}\sum_{i=0}^{n_j-1}R\circ f^i(x)\ge\frac{m}{n_j}\#\{i<n_j\,;\,R\circ f^i(x)>m\}\ge m\frac{\delta}{2}.$$
Thus, $\limsup\frac{1}{n}\sum_{j=0}^{n-1}R\circ f^j(x)\ge m\delta/2$ for every $m\in\NN$. That is, $\limsup\frac{1}{n}\sum_{j=0}^{n-1}R\circ f^j(x)=+\infty$.
\end{proof}

In Theorem~\ref{MainTheoremPartial} below, we have a version of the result of Theorem~\ref{MainTHEOVianaConjDiffeo} for partially hyperbolic systems. Nevertheless, we note that in Theorem~\ref{MainTHEOVianaConjDiffeo} we ask a stronger hypothesis over the Lyapunov exponents along the unstable direction (as well as on the stable direction). That is, we ask  $\liminf_{n}\frac{1}{n}\log\|(Df^n|_{\EE^{cu}(x)})^{-1}\|^{-1}$, instead of $\limsup_{n}\frac{1}{n}\log\|(Df^n|_{\EE^{cu}(x)})^{-1}\|^{-1}$, to be bigger than zero. 

Given a $C^1$ diffeomorphism $f:M\to M$,
we say that a forward invariant set $U$ (i.e.,$f(U)\subset U$) is  {\bf\em partially hyperbolic} if there exist a $Df$-invariant splitting $T_U M=\EE^{cu}\oplus\EE^{cs}$ and a constant $\sigma\in(0,1)$ such that the following three conditions holds for every $x\in U$:
\begin{enumerate}
	\item $\|Df|_{\EE^{cs}(x)}\|\|Df^{-1}|_{\EE^{cu}(x)}\|\le\sigma$ (dominated splitting);
	\item $\liminf_{n\to+\infty}\frac{1}{n}\log\|(Df^n|_{\EE^{cu}(x)})^{-1}\|^{-1}>0$ (positive Lyapunov exponents along the unstable direction);
	\item $\limsup_{n\to+\infty}\frac{1}{n}\log\|Df^n|_{\EE^{cs}(x)}\|<0$ (negative Lyapunov exponents along the stable direction).
\end{enumerate}

Given a point $x\in U$, we define the {\bf\em unstable Lyapunov residue} and the {\bf\em stable Lyapunov residue} as, respectively, \begin{equation}\label{EquationUnstLyRe}
  \Res^{u}(x)=\lim_{n\to\infty}\dd_{\NN}^+(\{j\ge 0\,;\, R^u\circ f^j(x)\ge n\})
  \end{equation}
and
\begin{equation}\label{EquationStLyRe}
  \Res^{s}(x)=\lim_{n\to\infty}\dd_{\NN}^+(\{j\ge 0\,;\, R^s\circ f^j(x)\ge n\}),
\end{equation}
where $$R^{u}(x)=\min\bigg\{n\ge1\,;\,\frac{1}{j}\log\|(Df^j|_{\EE^{cu}(x)})^{-1}\|^{-1}>\frac{1}{2}\liminf_{k\to+\infty}\frac{1}{k}\log\|(Df^k|_{\EE^{cu}(x)})^{-1}\|^{-1}\;\forall\,j\ge n\bigg\}$$
 and
 $$R^{s}(x)=\min\bigg\{n\ge1\,;\,\frac{1}{j}\log\|Df^j|_{\EE^{cs}(x)}\|<\frac{1}{2}\limsup_{k\to+\infty}\frac{1}{k}\log\|Df^k|_{\EE^{cs}(x)}\|\;\forall\,j\ge n\bigg\}.$$

Recall that the {\bf\em basin of attraction of a measure $\mu$} is the set
\begin{equation}\label{DefBasinOfAttraction}
  \beta_f(\mu)=\bigg\{x\,;\,\lim_n\frac{1}{n}\sum_{j=0}^{n-1}\delta_{f^j(x)}\to\mu\text{ on the weak topology}\bigg\}
\end{equation}

\begin{maintheorem}\label{MainTheoremPartial}
	Let $f:M\to M$ be a $C^2$ diffeomorphism having a partially hyperbolic set $U$. If $\Res^{u}(x)=\Res^s(x)=0$ for almost every $x\in U$, then almost every point in $U$ belongs to the basin of attraction of some SRB supported on $\bigcap_{j=0}^{+\infty}f^j\big(\overline{U}\big)$.
\end{maintheorem}

\section{Induced maps and coherence}\label{IndCoheInvSets}

Consider a map $f:X\to X$ defined in a set $X$.
The $f$-induced map defined on $A\subset X$ with an induced time $R:A\to\NN:=\{1,2,3,\cdots\}$ is the map $F:A\to X$ defined by $F(x)=f^{R(x)}(x)$. 

The {\bf\em forward orbit} of a point $x$ (for instance, with respect to $f$) is $\co_f^+(x)=\{f^j(x)\,;\,j\ge0\}$, the {\bf\em backward orbit (or pre-orbit)} of $x$ is $\co_f^-(x)=\{y\in\XX\,;\,x\in\co_f^+(y)\}$ and the orbit is $\co_f(x)=\co_f^+(x)\cup\co_f^-(x)$.
The {\bf\em omega limit set} of $x$, $\omega_f(x)$, is the set of accumulating points of $\co_f^+(x)$, that is, the set of $y\in \XX$ such that $y=\lim_{j\to\infty} f^{n_j}(x)$ for some sequence $n_j\to+\infty$.
The {\bf\em alpha limit set} of $x$, $\alpha_f(x)$, is the set of accumulating points of $\co_f^-(x)$, i.e., the set of $y\in \XX$ such that $y=\lim_{j\to\infty} y_j$ for some sequence $y_j$ such that $f^{n_j}(y_j)=x$ and $n_j\to+\infty$.

	Given a set $U\subset A$, we define the {\bf\em $(f,R)$-spreading of $U$} (for short, the {\bf\em  spreading of $U$}) as $$\widetilde{U}=\bigcup_{x\in U}\bigcup_{j=0}^{R(x)-1}f^j(x)= \bigcup_{n\ge1}\bigcup_{j=0}^{n-1}f^j\big(U\cap \{R=n\}\big)=\bigcup_{j\ge0}f^j\big(U\cap \{R>j\}\big).$$

\begin{Lemma}\label{LemmaGenInd1}
	If $F(U)\subset U\subset A$ then  $f(\widetilde{U})\subset\widetilde{U}$. Also, if $F(U)= U\subset A$, then  $f(\widetilde{U})=\widetilde{U}$.
\end{Lemma}
\begin{proof}
Given a set $U\subset A$, it is easy to see that
\begin{equation}\label{Eqcannel76rtf}f\bigg(\bigcup_{n\ge1}\bigcup_{j=0}^{n-2}f^j\big(U\cap \{R=n\}\big)\bigg)=\bigcup_{n\ge1}\bigcup_{j=1}^{n-1}f^j\big(U\cap\{R=n\}\big)\subset\widetilde{U}.
\end{equation}

If $F(U)\subset U$ then $$f\bigg(\bigcup_{n\ge1}f^{n-1}\big(U\cap\{R=n\}\big)\bigg)=\bigcup_{n\ge1}f^{n}\big(U\cap\{R=n\}\big)=F(U)\subset U\subset\widetilde{U}.$$
And so, $f(\widetilde{U})\subset\widetilde{U}$.
Similarly, $f(\widetilde{U})=\widetilde{U}$ whenever $F(U)=U$.
\end{proof}

In the example below one can see that, in general, $F$-invariance (i.e., backward invariance) is not preserved by the spreading.

\begin{Example}[$F^{-1}(U)=U\subset A \nRightarrow(f|_{\widetilde{A}})^{-1}(\widetilde{U})=\widetilde{U}$]\label{Exampleuyhj34}
Consider $\XX=\{1,2,3,4\}$ with $f:\XX\to\XX$ given by $f(1)=f(4)=2$, $f(2)=3$ and $f(3)=1$.
Let $A=\XX$ and $R:A\to\NN$ given by $R(1)=2$, $R(2)=3$ and $R(3)=R(4)=1$.
\begin{figure}
  \begin{center}\includegraphics[scale=.19]{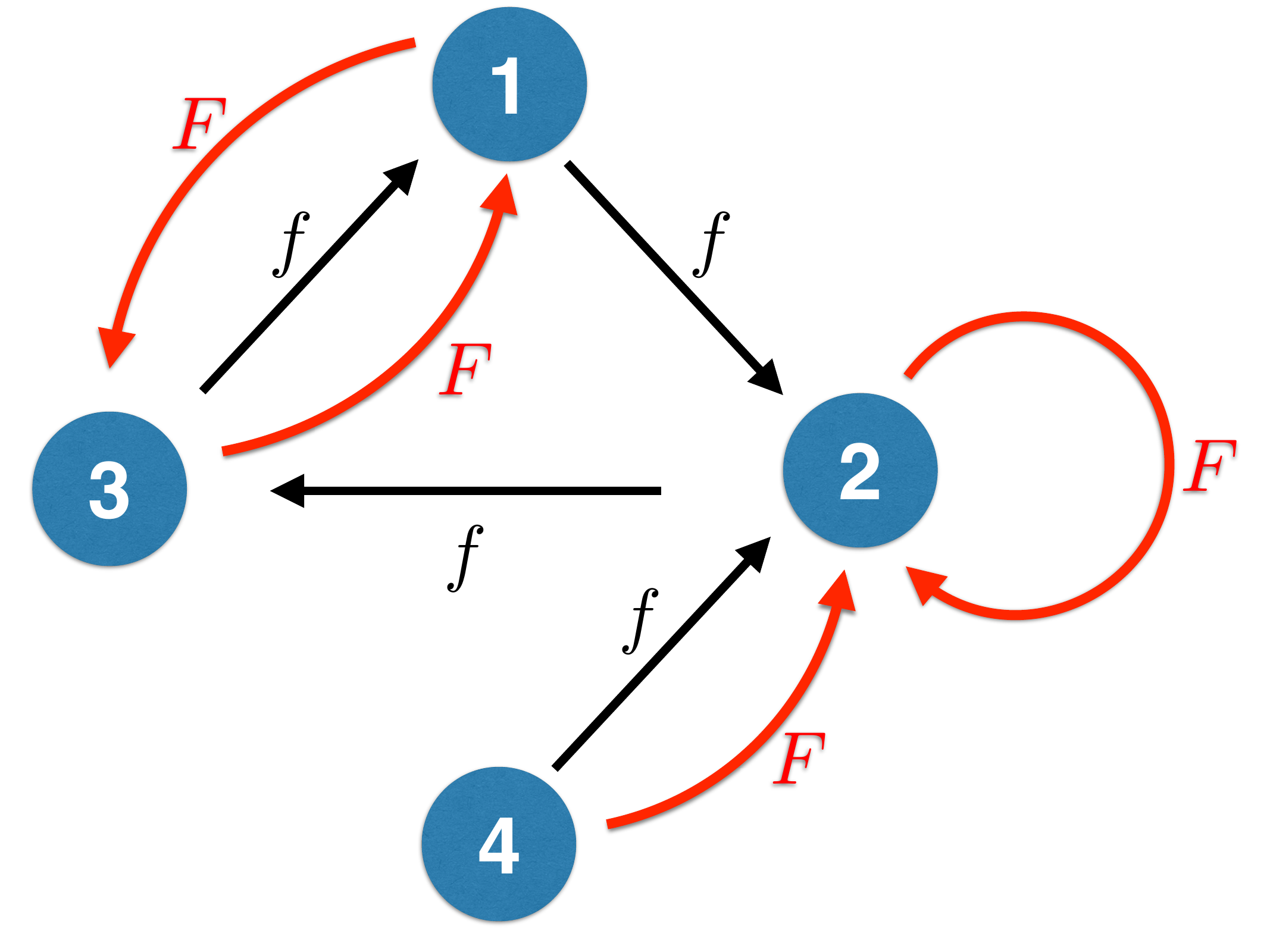}\\
 \caption{The picture shows the diagram of the maps $f$ and $F$ of Example~\ref{Exampleuyhj34}.}\label{FigurasLifSych1}
  \end{center}
\end{figure}
Thus, 
$F:A\to X$ is given by $F(1)=f^2(1)=3$, $F(2)=f^3(2)=2$, $F(3)=f(3)=1$ and $F(4)=f(4)=2$, see Figure~\ref{FigurasLifSych1}.

In this example, $\widetilde{A}=A=\XX$ and so, $(f|_{\widetilde{A}})^{-1}(x)=f^{-1}(x)$ $\forall\,x$.
As $F^{-1}(1)=\{3\}$ and $F^{-1}(3)=\{1\}$, $U=\{1,3\}$ is  a $F$-invariant.
As $\widetilde{U}=\{1,2,3\}$, we get that $(f|_{\widetilde{A}})^{-1}(\widetilde{U})=f^{-1}(\widetilde{U})=\{1,2,3,4\}\ne \widetilde{U}$.
\end{Example} 

\begin{Definition}[Coherent and exact induced times]\label{DefPlissTime}
	We say that an induced time $R$ is {\bf\em coherent} if $R(x)\ge R\circ f^{j}(x)+ j$ whenever 
$x\in A$, $0\le j<R(x)$ and $f^j(x)\in A$. The induced time $R$ is called {\bf\em exact} if $R(x)=R\circ f^j(x)+j$ for every  $0\le j<R(x)$ and every $x\in A$.
\end{Definition}
 Note that the class of the exact induced times contains all the {\em first entry times}.
 That is, given a set $U\subset X$ such that $\co_f^+(x)\cap U\ne\emptyset$ for every $x\in A$, the map $R(x)=\min\{j\ge1\,;\,f^j(x)\in U\}$ is called the {\bf\em first entry time} to $U$ and $F(x)=f^{R(x)}(x)$ is called the {\bf\em first entry map} to $U$.
 A particular case of a first entry map is when $A=U$, in this case $F$ and $R$ are called, respectively, the {\bf\em first return map} and the {\bf\em first return time} to $A$.
 
In Section~\ref{SecPlissESy} we show that coherent induced times appear quite naturally associated to the existence of $(\gamma,\varphi)$-Pliss times (Definition~\ref{DefinitionPlissTime}) and it has many applications in the theory of non-uniformly hyperbolic dynamics.

\begin{Lemma}\label{LemmaCanCoh1}
	Suppose that $R$ is a coherent induced time and let $x\in A_0:=\bigcap_{n\ge0}F^{-n}(\XX)$. If $0\le a<R(x)$ and $f^a(x)\in  A_0$
	then there exists $1\le b\le \#\{a\le  j< R(x)\,;\,f^j(x)\in A_0\}$ such that $$R(x)=a+\sum_{j=0}^b R(F^j(f^a(x))).$$
In particular, $F(x)=F^b(f^a(x))$.
\end{Lemma}
\begin{proof}
As $f^a(x)\in A_0$, it follows from the coherence that $a<a+R(p_0)\le R(x)$, where $p_0:=f^a(x)$.
If $a+R(p_0)= R(x)$ then the proof is done.
If not, as $p_1:=f^{a+R(p_0)}(x)=F(f^a(x))\in A_0$, by coherence we get that $a+R(p_0)+R(p_1)\le R(x)$.
Again, if $a+R(p_0)+R(p_1)=R(x)$, the proof is done.
If not, we take $p_2:=f^{a+R(p_0)+R(p_1)}(x)=F^2(f^a(x))$ and repeat the process. As $\sum_{j=0}^{n-1} R(p_j)\ge n$ and $R(x)<+\infty$ the process will stop.
That is, there exists $b\ge 0$ such that $R(x)=a+\sum_{j=0}^{b-1} R(p_j)=a+\sum_{j=0}^{b-1} R(F^j(f^a(x)))$.
As $f^{a+\sum_{j=0}^{n-1} R(p_j)}(x)\in  A_0$ $\forall\,0\le n\le b$, we get that $b\le\#\{a\le j<R(x)\,;\,f^j(x)\in  A_0\}$. 
\end{proof}

\begin{Definition}[Orbit-coherence]\label{DefinitionORBCOR}
	The induced map $F$ is called {\bf\em orbit-coherent}
	if 
	\begin{equation}
		\co_f^+(x)\cap\co_f^+(y)\ne\emptyset\Longleftrightarrow\co_F^+(x)\cap\co_F^+(y)\ne\emptyset
	\end{equation}
	for every $x,y\in\bigcap_{j\ge0}F^{-j}(\XX)$.
\end{Definition}

\begin{Lemma}\label{LemmaOrbCoh1}
	If $R$ is a coherent induced time then $F$ is orbit-coherent. \end{Lemma}
\begin{proof} Set $R_n(p)=\sum_{j=0}^{n-1}R\circ F^j(p)$ any $p\in\bigcap_{j\ge0}F^{-j}(\XX)$ and $n\ge0$. Note that $F^n(p)=f^{R_n(p)}(p)$.
Let $n_x,n_y\ge0$ be such that $\alpha:=f^{n_x}(x)=f^{n_y}(y)$ and $m_x,m_y\ge0$ be the integers satisfying $R_{m_x}(x)\le n_x<R_{m_x+1}(x)$ and $R_{m_y}(y)\le n_y<R_{m_y+1}(y)$.
Letting $p=F^{n_x}(x)=f^{R_{n_x}(x)}(x)$ and $a=R_{n_x}(x)-n_x$, it follows from Lemma~\ref{LemmaCanCoh1} that there exists $b\ge1$ such that $F^{n_x+1}(x)=F(p)=F^{b}(f^{a}(p))=F^{b}(\alpha)$.
Similarly, there exists $b'\ge1$ so that $F^{n_y+1}(y)=F(q)=F^{b'}(f^{a'}(q))=F^{b'}(\alpha)$, where $q=F^{n_y}(y)=f^{R_{n_y}(y)}(y)$ and $a'=R_{n_y}(y)-n_y$.
Thus, $\co_F^+(x)\cap\co_F^+(y)\supset\co_F^+(F^{\ell}(\alpha))$, where $\ell=\max\{a,a'\}$ .
\end{proof}

\begin{Example}[Orbit-coherence $\nRightarrow$ coherence]\label{Examplejgduyt7}
	Let $A=X=\{1,2,3\}$, $f:X\to X$ given by $f(1)=2$, $f(2)=3$ and $f(3)=1$. Let $R(x)=2$ $\forall\,x$, that is, $F=f^2$, see Figure~\ref{FigurasLifSych2}.
\begin{figure}
  \begin{center}\includegraphics[scale=.17]{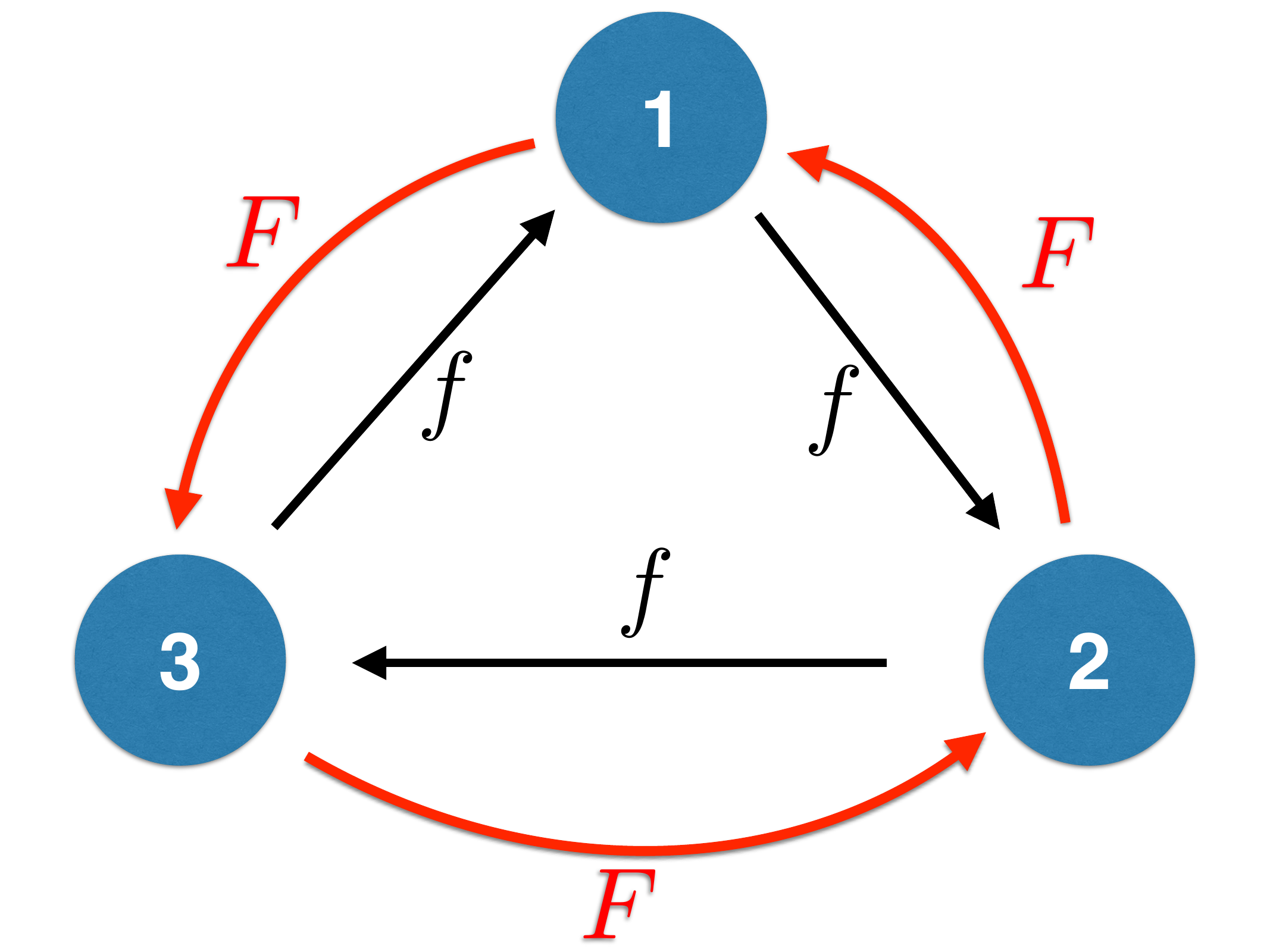}\\
 \caption{The picture shows the diagram of the maps $f$ and $F$ of Example~\ref{Examplejgduyt7}.}\label{FigurasLifSych2}
  \end{center}
\end{figure}
	Note that $\co_f^+(x)=\co_F^+(x)=X$ for every $x$ and so, $F$ is orbit-coherent. Nevertheless, as $R(x)=2<3=R(f^1(x))+1$ $\forall\,x$, $R$ is not coherent.
\end{Example}
 
\section{Measurable induced maps and ergodicity}\label{MeIndMaEr}

In this section $(X,\mathfrak{A},\mu)$ is a finite measure space, that is, $\mathfrak{A}$ is a $\sigma$-algebra of subsets of $X$ and $\mu$ is a measure on $\mathfrak{A}$ with $\mu(X)<+\infty$. Consider a measurable map $f:X\to X$,  $A\in\mathfrak{A}$ and a $f$-induced map $F:A\to\XX$ given by a measurable induced time $R:A\to\{1,2,3,\cdots\}$.

\begin{Definition}
	A map $g:X\to X$ is called {\bf\em $\mu$-ergodic} if $g$ is measurable and $\mu(V)$ or $\mu(\XX\setminus V)=0$ for every $g$-invariant measurable set $V\subset X$. Conversely, we say that $\mu$ is {\bf\em $f$-ergodic} whenever $f$ is $\mu$-ergodic.
\end{Definition}
Note that we are not assuming in the definition above that $f$ preserves $\mu$. That is, $\mu$ does not need to be $f$-invariant to be $f$-ergodic. 

A measurable map $g:U\to X$, $U\in\mathfrak{A}$, is called {\bf\em non-singular} (with respect to $\mu$) if $g_*\mu\ll\mu$, that is, if $\mu\circ g^{-1}\ll\mu$.

Is this section we study the connection between ergodicity and coherence. In particular, we show in Proposition~\ref{PropositionErGOS} that if $f$ is a non-singular ergodic map then $F$ is also non-singular and ergodic, whenever $F$ is orbit coherent.
Although orbit-coherent induced maps have a well behavior with respect to the ergodicity, this is not true for the transitivity, as one can see in Example~\ref{ExampleBadTras} below.

\begin{Example}[$f$ transitivity and coherence $\nRightarrow$ $F$ transitivity]\label{ExampleBadTras}
	Let $A=X=\{1,2,3\}$, $f:X\to X$ given by $f(1)=2$, $f(2)=3$ and $f(3)=1$. 
	Let $R(1)=2$ and $R(2)=R(3)=1$. Thus, $F(1)=F(2)=3$ and $F(3)=1$, see Figure~\ref{FigurasLifSych3} .
\begin{figure}
  \begin{center}\includegraphics[scale=.17]{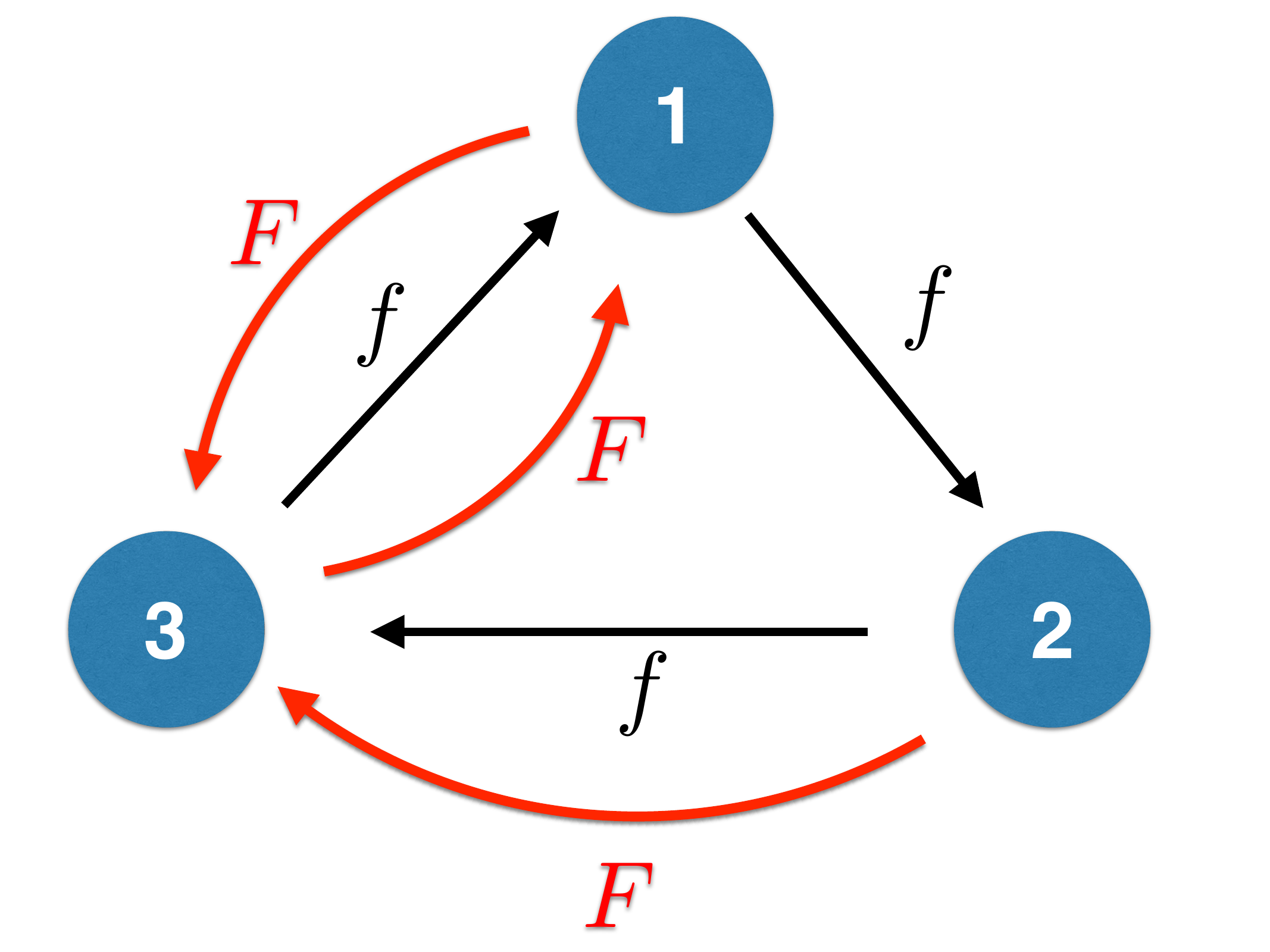}\\
 \caption{The picture shows the diagram of the maps $f$ and $F$ of Example~\ref{ExampleBadTras}.}\label{FigurasLifSych3}
  \end{center}
\end{figure}
Note that $f$ is transitive and $R$ is coherent (in particular $F$ is orbit-coherent), but $F$ is not transitive as $\bigcup_{j\ge0}F^j(\{1\})=\{1,3\}$ and so, $\bigcup_{j\ge0}F^j(\{1\})\cap\{2\}=\emptyset$.
\end{Example}

\begin{Lemma}\label{LemmaCoe}
	If $F$ is orbit-coherent and $F^{-1}(U)=U\subset A_0$, where $A_0=\bigcap_{n\ge0}F^{-n}(\XX)$, then $\widetilde{U}\cap A_0=U$ and $(f|_{\widetilde{A_0}})^{-1}(\widetilde{U})=\widetilde{U}$.
\end{Lemma}
\begin{proof}

If $p\in\widetilde{U}\cap A_0$ then $\exists\,x_p\in U$ and $0\le a<R(x_p)$ such that $p=f^a(x_p)\in A_0$. As $\co_f^+(x_p)\supset\co_f^+(p)$, it follows from the orbit-coherence that $\exists\,n,m\ge0$ such that $F^n(p)=F^m(x_p)$. 
As $U$ is $F$-invariant, we get that $F^m(x_p)\in U$ and so, $p\in F^{-n}(U)=U$, proving that $\widetilde{U}\cap A_0=U$.

Suppose that $F^{-1}(U)=U$. As $f(\widetilde{U})\subset\widetilde{U}$  (see Lemma~\ref{LemmaGenInd1}), we only need to show that $(f|_{\widetilde{A_0}})^{-1}(\widetilde{U})\subset\widetilde{U}$. 
So, consider $p\in\widetilde{U}$ and $x_p\in U$ and $0\le j<R(x_p)$ be such that $p=f^j(x_p)$.
Let $q$ be any pre-image of $p$ by $(f|_{\widetilde{A_0}})^{-1}$.
That is, $q\in (f|_{\widetilde{A_0}})^{-1}(p)=f^{-1}(p)\cap\widetilde{A_0}$.
Let $x_q\in A_0$ and $0\le\ell<R(x_q)$ such that $f^{\ell}(x_q)=q$.
As $p\in\co_f^+(x_p)\cap\co_f^+(x_q)$, it follows from Lemma~\ref{LemmaOrbCoh1} that there exist $n,m\ge1$ so that $F^n(x_p)=F^m(x_q)$.
As $x_p\in U$ and $U$ is $F$-invariant, we get that $x_q\in F^{-m}(U)=U$, proving the lemma. 
\end{proof}

\begin{Lemma}\label{LemmaSingSing}
	If $f$ is non-singular with respect to $\mu$, then 
	\begin{enumerate}
		\item $F$ is non-singular with respect to $\mu$;
		\item $f$ is non-singular with respect to $\mu|_U$, for every forward invariant set $U\in\mathfrak{A}$  with $\mu(U)>0$.
	\end{enumerate}
\end{Lemma}
\begin{proof}
 As $F^{-1}(V)=\bigcup_n (f^n|_{\{R=n\}})^{-1}(V)$, if $\mu(V)=0$, it follows from $\mu\circ f^{-1}\ll\mu$ that $\mu(f^{-n}(V))=\mu(V)=0$ $\forall\,n\ge1$ and so, $\mu(F^{-1}(V))=0$, proving that $F$ is non-singular with respect to $\mu$. 

As $\mu\circ f^{-1}\ll\mu$, if $\mu|_U(V)=\mu(V\cap U)=0$, then $\mu(f^{-1}(V\cap U))=0$.
By the forward invariance of $U$, we get that $U\subset f^{-1}(f(U))\subset f^{-1}(U)$ and so $\mu(f^{-1}(V)\cap U)\le\mu(f^{-1}(V)\cap f^{-1}(U))=\mu(f^{-1}(V\cap U))=0$, proving that $f$ is non-singular with respect to $\mu|_U$.
\end{proof}

\begin{Lemma}\label{LemmaEhErgodica}
	If  $f$ is $\mu$-ergodic and $U\subset X$ a measurable set with positive measure, then $\nu:=\frac{1}{\mu(U)}\mu|_U$ is a $f$-ergodic probability.
\end{Lemma}
\begin{proof}
		Consider a $f$ invariant measurable set $V$ with $\nu(V)>0$.
	Thus, $\mu(V)\ge\nu(V)>0$.
	As a consequence, it follows from the ergodicity of $\mu$ that $\mu(V)=1$.
	So, $\nu(V)=\frac{1}{\mu(U)}\mu(V\cap U)=\frac{1}{\mu(U)}\mu(U)=1$, proving that $\nu$ is ergodic.	
\end{proof}

\begin{Lemma}
	Suppose that  $f$ is non-singular with respect to $\mu$. If $U\subset\XX$ is a $\mu$-almost $f$-invariant measurable set, i.e., $\mu(f^{-1}(U)\triangle U)=0$, then there is a $f$-invariant measurable set $U'\subset\XX$ such that   $\mu(U\triangle U')=0$.
\end{Lemma}
\begin{proof}
Recall that  $\mu(A\triangle C)\le\mu(A\triangle B)+\mu(B\triangle C)$, for every measurable set $A,B$ and $C$.
Thus, as $\mu(f^{-1}(U)\triangle U)=0$ and $\mu\circ f^{-1}\ll\mu$, we get that $\mu(f^{-j}(U)\triangle f^{-(j-1)}(U))$ $=$ $\mu\big(f^{-(j-1)}\big(f^{-1}(U)\triangle U\big)\big)=0$ for every $j\ge1$.
So, $$\mu(f^{-j}(U)\triangle U)\le\mu(f^{-j}(U)\triangle f^{-(j-1)}(U))+\cdots+\mu(f^{-1}(U)\triangle U)=0.$$
As a consequence, $\mu(f^{-j}(U)\cap U)\le\mu(U)\le\mu(f^{-j}(U)\cap U)+\mu(f^{-j}(U)\triangle U)=\mu(f^{-j}(U)\cap U)$. That is,  $\mu(f^{-j}(U)\cap U)=\mu(U)$ $\forall\,j\ge0$. So,  $\mu(U_0)=\mu(U)$, where $U_0:=\bigcap_{j\ge0}f^{-j}(U)$.
As $f(U_0)\subset U_0$, we get that $f^{-1}(U')=U'$, where $U'=\bigcup_{j\ge0}f^{-j}(U_0)$.
Note that $$\mu(U'\triangle U)\le\mu(U'\triangle U_0)\le\sum_{j\ge0}\mu(f^{-j}(U_0)\triangle U_0)=0,$$ as it is easy to see that $\mu(f^{-j}(U_0)\triangle U_0)=0$.
\end{proof}

\begin{Corollary}\label{Corollaryoiuy89}
	If $f$ is non-singular with respect to $\mu$, then $f$ is $\mu$-ergodic if and only if $\mu(U)\mu(\XX\setminus U)=0$ for every measurable set $U$ such that $\mu(f^{-1}(U)\triangle U)=0$.
\end{Corollary}

\begin{Proposition}\label{PropositionErGOS}
Suppose that $f$ is non-singular with respect to $\mu$ and $\mu(A_0)>0$, where $A_0:=\bigcap_{n\ge0}F^{-n}(\XX)$. If $F$ is orbit-coherent and $\mu$ is $f$-ergodic then $\mu|_{A_0}$ is $F$-ergodic.
\end{Proposition}
\begin{proof}
As $\mu(\widetilde{A_0})>0$ and $f$ is $\mu$-ergodic and non-singular, it follows from Lemma~\ref{LemmaSingSing} and \ref{LemmaEhErgodica} that $f$ is ergodic and non-singular with respect to the probability $\nu:=\frac{1}{\mu(\widetilde{A_0})}\mu|_{\widetilde{A_0}}$.
Let $U\subset A_0$ be a measurable set that is $F$-invariant and such that $\mu(U)>0$.
Thus, $\nu(\widetilde{U})=\mu(\widetilde{U})/\mu(\widetilde{A_0})>0$.
It follows by Lemma~\ref{LemmaCoe} that $$\nu(f^{-1}(\widetilde{U})\triangle\widetilde{U})=\frac{1}{\mu(\widetilde{A_0})}\mu((f^{-1}(\widetilde{U})\triangle\widetilde{U})\cap \widetilde{A_0})=\frac{1}{\mu(\widetilde{A_0})}\mu((f^{-1}(\widetilde{U})\cap \widetilde{A_0})\triangle\widetilde{U})=$$
$$=\frac{1}{\mu(\widetilde{A_0})}\mu\big(f|_{\widetilde{A_0}}^{-1}(\widetilde{U})\triangle\widetilde{U}\big)=\frac{1}{\mu(\widetilde{A_0})}\mu(\widetilde{U}\triangle\widetilde{U})=0.$$
As $f$ is ergodic and non-singular with respect to $\nu$, it follows from Corollary~\ref{Corollaryoiuy89} that $\nu(\widetilde{U})=1$. That is, $\mu(\widetilde{U})=\mu(\widetilde{A_0})$.
Thus, using again Lemma~\ref{LemmaCoe}, we have that $\mu(U)=\mu(\widetilde{U}\cap A_0)=\mu(A_0)$, proving the $F$-ergodicity of $\mu|_{A_0}$.
\end{proof}

\section{Lift results}\label{LiftRes}

In this section, unless otherwise noted, $(\XX,\mathfrak{A})$ is a measurable space and $f:\XX\to\XX$ a measurable map. 
The Lemma~\ref{FolkloreResultA} below is a well-known result, and a proof of it can be found in the Appendix.

\begin{Lemma}[Folklore result I]\label{FolkloreResultA}
Let $\mu$ be an ergodic $f$ invariant probability and $F:A\to U$ the first return map to a set $U\subset\XX$ by $f$, where $A=\{x\in B\,;\,\co_f^+(f(x))\cap U\ne\emptyset\}$. If $\mu(U)>0$ then $\frac{1}{\mu(U)}\mu|_{U}$ is an ergodic $F$-invariant probability.
\end{Lemma}

The {\bf\em tower} associated to the induced map $F$ is the set 
$$\widehat{A}=A\sqcup\{(x,n)\,;\,x\in A\text{ and }1\le n<R(x)\}.$$
The {\bf\em tower map} associated to $F$ is the map 
$\widehat{F}:\widehat{A}\to\widehat{\XX}$, where $\widehat{\XX}:=\XX\sqcup \XX\times\{1,2,3,\cdots\}\supset\widehat{A}$  and $\widehat{F}$ is defined by $$\widehat{F}(x)=\begin{cases}
	F(x) & \text{ if }x\in A\text{ and }R(x)=1\\
	(x,1) & \text{ if }x\in A\text{ and }R(x)>1
\end{cases}$$
and, for $x\in A$ and $1\ne n<R(x)$,
$$\widehat{F}(x,n)=
\begin{cases}
	(x,n+1) & \text{ if }n<R(x)-1\\
	F(x) & \text{ if }n=R(x)-1
\end{cases}.
$$

Let $\pi:\widehat{\XX}\to\XX$ be the {\bf\em tower projection} given by $\pi(x)=x$, if  $x\in\XX$, and $\pi(x,n)=f^n(x)$ when $(x,n)\in\XX\times\{1,2,3,\cdots\}$.
Considering on $\widehat{\XX}$ the topology induced by $\XX$, the projection is a measurable map.
Furthermore, as $\pi\circ\widehat{F}=f\circ\pi$, if $\eta$ is a $\widehat{F}$-invariant probability then the push-forward $\pi_*\eta:=\eta\circ\pi^{-1}$ is a $f$ invariant probability.
Thus, a $f$-invariant probability $\mu$ is called {\bf\em liftable by $\widehat{F}$} if $\mu=\pi_*\eta$ for some $\widehat{F}$-invariant probability $\eta$ ($\eta$ is called the {\bf\em $\widehat{F}$-lift of $\mu$} and $\mu$ is the {\bf\em tower projection of $\mu$}).

Observing that $F$ is the first return map to $A$ by $\widehat{F}$ and as $\co_{\widehat{F}}^+(p)\cap A\ne\emptyset$ for every $p\in\bigcap_{j\ge0}\widehat{F}^{-j}(\widehat{\XX})$,
it follows from  Lemma~\ref{FolkloreResultA} that $\nu:=\frac{1}{\eta(A)}\eta|_{A}$ if a $F$-invariant probability, whenever $\eta$ is a $\widehat{F}$-invariant ergodic probability.
Note that, by the tower construction, $\eta(A)>0$ for every $\widehat{F}$ invariant probability.
Furthermore, we can conclude, using the Ergodic Decomposition Theorem, that $\frac{1}{\eta(A)}\eta|_A$ is $F$-invariant even when $\eta$ is not $\widehat{F}$-ergodic.

Conversely, given a probability $\eta$ on $\XX$, define the {\bf\em tower measure} associated to $\eta$ as $\widehat{\eta}(U)=\sum_{n\ge0}\eta(U_n)$, where $U_0=U\cap A$ and $U_n\subset A$ is defined by $U_n\times\{n\}=U\cap\big(A\times\{n\}\big)$ for $n\ge1$.
It is not difficult to check that, if $\nu$ is a $F$-invariant probability, then $\widehat{\nu}$ is a $\widehat{F}$-invariant measure.
Nevertheless $\widehat{\nu}$ is a finite measure only if $\widehat{\nu}(\widehat{\XX})=\int R d\nu<\infty$.
So, if $\nu$ is a $F$ invariant probability with $\int Rd \nu<+\infty$ then
\begin{equation}\label{Eq78987yh}\small
  \widetilde{\nu}:=\pi_*\bigg(\frac{1}{\widehat{\nu}(\widehat{\XX})}\widehat{\nu}\bigg)=\frac{1}{\int R d\nu}\pi_*\widehat{\nu}=\frac{1}{\int R d\nu}\sum_{n\ge1}\sum_{j=0}^{n-1}f^j_*(\nu|_{\{R=n\}})=\frac{1}{\int R d\nu}\sum_{j\ge0}f^j_*(\nu|_{\{R>j\}})
\end{equation}
 is a $f$-invariant probability. As a consequence, we get the following well-known result.

\begin{Lemma}[Folklore result II]\label{FolkloreResultB}
If $F:A\to\XX$ is a measurable $f$ induced map with induced time $R$ and $\nu$ is a $F$ invariant probability then $$\mu:=\sum_{n\ge1}\sum_{j=0}^{n-1}f^j_*(\nu|_{\{R=n\}})=\sum_{j\ge0}f^j_*(\nu|_{\{R>j\}})
$$ is a $f$ invariant measure with $\mu(\XX)=\int R d\nu$.
\end{Lemma}

Because of (\ref{Eq78987yh}), we say that a $f$-invariant probability $\mu$ is {\bf\em $F$-liftable} if there is a $F$-invariant probability $\nu$, with $\int Rd \nu<+\infty$, such that $\widetilde{\nu}=\mu$.
The probability $\nu$ is called a {\bf\em $F$-lift} of $\mu$.
Hence, $\mu$ is $F$-liftable if and only if $\mu$ is $\widehat{F}$-liftable. Note that, if $\mu$ is a $f$-invariant probability and $\nu\ll\mu$ is a $F$ invariant probability then $\widetilde{\nu}\ll\mu$. Therefore, if $\mu$ is also $f$ ergodic then it follows from the ergodicity that $\mu=\widetilde{\nu}$ which proves the Corollary~\ref{FolkloreResultC} below.

\begin{Corollary}[Folklore result III]\label{FolkloreResultC}
If $\mu$ is an ergodic $f$ invariant probability and $F:A\to\XX$ is a measurable $f$ induced map with induced time $R$ then, $\mu$ is $F$-liftble if and only if there exists a $F$ invariant probability $\nu\ll\mu$ such that $\int R d\nu<+\infty$.
\end{Corollary}

In Lemma~\ref{LemmaCalculoBOM} below, let $F:A\to\XX$ be a measurable $f$ induced map with induced time $R$. Furthermore, for any given $x\in A_0:=\bigcap_{j\ge0}F^{-j}(A)$ let $$\theta_F(x):=\limsup_{n\to\infty}\frac{1}{n}\#\{0\le j<n\,;\,f^j(x)\in\co_F^+(x)\}.$$

\begin{Lemma}\label{LemmaCalculoBOM} If $i_x(n)>0$ then $\frac{1}{i_x(n)}\sum_{j=0}^{i_x(n)-1}R\circ F^j(x)< \frac{n}{i_x(n)}$ for every $x\in A_0$, where $i_x(n)=\#\{0\le j<n\,;\,f^j(x)\in\co_F^+(x)\}$. In particular,  
$\liminf_{n}\frac{1}{n}\sum_{j=0}^{n-1}R\circ F^j(x)\le\frac{1}{\theta_F(x)}$ for every $x\in A_0$.
Moreover, if $\lim_{n}\frac{1}{n}\sum_{j=0}^{n-1}R\circ F^j(x)$ exists, then $$\lim_{n\to\infty}\frac{1}{n}\#\{0\le j<n\,;\,f^j(x)\in\co_F^+(x)\}=\frac{1}{\lim_{n}\frac{1}{n}\sum_{j=0}^{n-1}R\circ F^j(x)}.$$
\end{Lemma}
\begin{proof} Let $x\in A_0$ be such that $\theta_F(x)>0$ and set $$E_x(n):=\bigg\{\sum_{k=0}^{j-1}R\circ F^k(x)\,;\,j\ge0\text{ and }\sum_{k=0}^{j-1}R\circ F^k(x)<n\bigg\}.$$
	As $E_x(n)=\{0\le j<n\,;\,f^j(x)\in\co_F^+(x)\}$ and $\# E_x(n)$ $=$ $\#\{j\ge0\,;\,\sum_{k=0}^{j-1}R\circ F^k(x)<n\}$ $=$ $\max\{j\ge0\,;\,\sum_{k=0}^{j-1}R\circ F^k(x)<n\}$, we get that $$i_x(n)=\#\{0\le j<n\,;\,f^j(x)\in\co_F^+(x)\}=\max\bigg\{j\ge0\,;\,\sum_{k=0}^{j-1} R\circ F^k(x)<n\bigg\}.$$
Hence, 
$\sum_{j=0}^{i_x(n)-1}R\circ F^j(x)<n\le\sum_{j=0}^{i_x(n)}R\circ F^j(x)$ for every $x\in A_0$ and $n\ge1$.
As a consequence,
$$\frac{1}{i_x(n)}\sum_{j=0}^{i_x(n)-1}R\circ F^j(x)<\frac{n}{i_x(n)}=\bigg(\frac{1}{n}\#\{0\le j<n\,;\,f^j(x)\in\co_F^+(x)\}\bigg)^{-1}\le$$
$$\le\underbrace{\frac{i_x(n)+1}{i_x(n)}}_{\alpha(n)}\bigg(\frac{1}{i_x(n)+1}\sum_{j=0}^{i_x(n)}R\circ F^j(x)\bigg),$$
for every $n>R(x)$ (note that $n>R(x)$ $\implies$ $i_x(n)\ge1$).
That is, if $x\in A_0$ and $n>R(x)$,
\begin{equation}\label{Equation6545jg}
  \frac{1}{i_x(n)}\sum_{j=0}^{i_x(n)-1}R\circ F^j(x)<\frac{1}{\frac{1}{n}\#\{0\le j<n;f^j(x)\in\co_F^+(x)\}}\le\frac{\alpha(n)}{i_x(n)+1}\sum_{j=0}^{i_x(n)}R\circ F^j(x),
\end{equation}
with $\lim_n\alpha(x)=1$.
So, taking ``$\liminf$'' in the first inequality of (\ref{Equation6545jg}), we get that
$$\liminf_{n\to\infty}\frac{1}{n}\sum_{j=0}^{n-1}R\circ F^j(x)\le\liminf_{n\to+\infty}\;\,\frac{1}{i_x(n)}\sum_{j=0}^{i_x(n)-1}R\circ F^j(x)\le$$
$$\le\liminf_{n\to+\infty}\frac{1}{\frac{1}{n}\#\{0\le j<n;f^j(x)\in\co_F^+(x)\}}=\frac{1}{\theta_F(x)}.$$
If $\lim_{n}\frac{1}{n}\sum_{j=0}^{n-1}R\circ F^j(x)$ exists, then (\ref{Equation6545jg}) implies that 
$$\lim_{n\to+\infty}\frac{1}{\frac{1}{n}\#\{0\le j<n;f^j(x)\in\co_F^+(x)\}}=\lim_{n\to+\infty}\frac{1}{n}\sum_{j=0}^{n-1}R\circ F^j(x)
$$
and so, $\frac{1}{\theta_F(x)}=\lim_{n\to+\infty}\frac{1}{n}\sum_{j=0}^{n-1}R\circ F^j(x)$, concluding the proof.
\end{proof}

Before going further, let we recall the ergodic version of Kac's theorem and use Lemma~\ref{FolkloreResultA}  and \ref{LemmaCalculoBOM} above to prove it. 

\begin{T}[Kac]\label{KacResult} 
Let $\mu$ be an ergodic $f$ invariant probability and $R:A\to\NN$ the first return time to $U\subset\XX$ by $f$, where $A=\{x\in U\,;\,\co_f^+(f(x))\cap U\ne\emptyset\}$. If $\mu(U)>0$ then $\int_{U} R d\mu=1$.
\end{T}
\begin{proof}
As $F$ is the first return map to $U$, $f^j(x)\in\co_F^+(x)$ $\iff$ $f^j(x)\in U$ for every $x\in A$. Hence, by Birkhoff, $\lim_{n\to\infty}\frac{1}{n}\#\{0\le j<n\,;\,f^j(x)\in\co_F^+(x)\}$ $=$ $\lim_n\frac{1}{n}\#\{0\le j<n\,;\,f^j(x)\in U\}=\mu(U)$ for $\mu$ almost every $x\in A$.

On the other hand, it follows from Lemma~\ref{FolkloreResultA} that $\nu:=\frac{1}{\mu(U)}\mu|_U$ is an ergodic $F$-invariant probability. Thus, it follows from Birkhoff and Lemma~\ref{LemmaCalculoBOM} that $$\frac{1}{\mu(U)}\int_U Rd\mu=\int Rd\nu=\frac{1}{\lim_{n\to\infty}\frac{1}{n}\#\{0\le j<n\,;\,f^j(x)\in\co_F^+(x)\}}=\frac{1}{\mu(U)}$$
for $\nu$ almost every $x\in U$, which concludes the proof.\end{proof}

\subsection{Half lifting}

Theorem~\ref{HalfLifting} below is crucial in the proof of many results of this paper. 

\begin{T}[Half lifting]\label{HalfLifting} 
Let $(\XX,\mathfrak{A})$ be a measurable space, $f:\XX\to\XX$ a bimeasurable injective map and $F:A\to\XX$ a measurable $f$ induced map with induced time $R$.

If $\mu$ is a $f$-invariant probability, then $\ca_F:=\bigcap_{n\ge0}F^n\big(\bigcap_{j\ge0}F^{-j}(\XX)\big)$ and $F(\ca_F)$ are measurable sets such that  $F(\ca_{F})$ $\subset$ $\ca_F$ and $\mu(\ca_F\setminus F(\ca_F))=0$. Furthermore, the following statements are equivalent.
	\begin{enumerate}
		\item $\mu(\ca_{F})>0$.
		\item $\mu(\ca_{F})>0$ and $\nu:=\frac{1}{\mu(\ca_{F})}\mu|_{\ca_{F}}$ is a $F$-invariant probability.
		\item There exists a $F$-invariant probability $\nu\ll\mu$.
	\end{enumerate} 
\end{T}
\begin{proof} Let $U\subset A_0:=\bigcap_{j\ge0}F^{-j}(\XX)$ be a measurable set and $n\ge1$.
As $f$ is bimeasurable and injective, we get that $f^n(U\cap\{R=n\})$ is measurable and so, $F(U)=\bigcup_{n\ge1}f^n(U\cap\{R=n\})$ is a measurable set. 
That is, also $F$ is bimeasurable.
In particular $\ca_F=\bigcap_{n\ge0}F^n(A_0)$ is measurable, as $A_0$ is a measurable set.
Moreover, $$F(\ca_F)=F\bigg(\bigcap_{n\ge0}F^n(A_0)\bigg)\subset\bigcap_{n\ge1}F^n(A_0)=A_0\cap \bigcap_{n\ge1}F^n(A_0)=\bigcap_{n\ge0}F^n(A_0)=\ca_F.$$

As $f$ is injective, it follows from the invariance of $\mu$, $\mu(F(U\cap\{R=n\}))=\mu(f^n(U\cap\{R=n\}))=\mu(U\cap\{R=n\})$. Hence, 
$$\mu(F(U))=\mu\bigg(\bigcup_{n\ge1}f^n(U\cap\{R=n\})\bigg)\le$$
$$\le\sum_{n\ge1}\mu(f^n(U\cap\{R=n\}))=\sum_{n\ge1}\mu(U\cap\{R=n\})=\mu(U)$$
for every measurable set $U\subset A_0$.

As $A_0\supset F(A_0)\supset\cdots\supset F^n(A_0)\nearrow \ca_F=\bigcap_{n\ge0}F^{n}(A_0)\subset A_0$, it follows from Lemma~\ref{LemmadeMEDIDA} of Appendix that $\mu(\ca_F)=\lim_n(F^n(A_0))$.
Furthermore, as $F^n(A_0)=(F^n(A_0)\setminus\ca_F)\cup\ca_F$, we get that $$\mu(\ca_F)\ge\mu(F(\ca_F))\ge\mu(F(F^n(A_0)))-\mu(F(F^n(A_0)\setminus\ca_F))\ge$$
$$\ge\mu(F^{n+1}(A_0))-\mu(F^n(A_0)\setminus\ca_F))\to\mu(\ca_F),$$
proving that $\mu(F(\ca_F))=\mu(\ca_F)$.
	
{\em (1)$\implies$(2).} Suppose that $\mu(\ca_F)>0$.
	As $\mu(\ca_F\triangle F(\ca_f))=0$, given $U\subset \ca_F$ consider a measurable set $V\subset\ca_F$ such that $U=F(V)\mod\mu$, i.e., $V\subset\ca_F\cap F^{-1}(U)\mod\mu$.
Thus, $\mu(\ca_F\cap U)=\mu(U)=\mu(F(V))\le\mu(V)\le\mu(\ca_F\cap F^{-1}(U))$.
As a consequence, $$\mu|_{\ca_F}(U)\le\mu|_{\ca_F}(F^{-1}(U))$$ for every measurable set $U\subset\XX$ and this implies that $\nu:=\mu|_{\ca_F}$ is a $F$ invariant finite measure.
Indeed, we already have that $\nu(F^{-1}(U))\ge\nu(U)$ for every measurable set $U\subset\XX$ and so, we only need to show the reverse inequality. For that, given a measurable set $U\subset\XX$, note that $$\nu(\XX)-\nu(F^{-1}(U))=\nu(F^{-1}(\XX\setminus U))\ge\mu(\XX\setminus U)=\nu(\XX)-\nu(U)$$
and so, $\nu(U)\ge\nu(F^{-1}(U))$, proving that $\nu(U)=\nu(F^{-1}(U))$ for every measurable set $U\subset\XX$.

{\em (2)$\implies$(3).} There is nothing to prove. 

{\em (3)$\implies$(1).} Let $\nu\ll\mu$ be a $F$-invariant probability. Then, it follows from Lemma~\ref{LemmaTeoErg876y7ui} of Appendix that $\nu(\ca_F)=1$. Hence, as $\nu\ll\mu$, we get that $\mu(\ca_F)>0$.
\end{proof}

\begin{Corollary}\label{CorollaryINTEGRAL}
Let $f$ be a measure-preserving automorphism of a probability space $(\XX,\mathfrak{A},\mu)$ and $F:A\to\XX$, $A\in\mathfrak{A}$,  a measurable induced map with induced time $R$.
Suppose that $f$ is bimeasurable and injective, $\mu$ is  $f$ ergodic and $F$-liftable. If $F$ is orbit coherent then $\mu(\ca_F)>0$ and $\nu:=\frac{1}{\mu(\ca_F)}\mu|_{\ca_F}$ is $F$-ergodic and it is the unique $F$-lift of $\mu$. 
\end{Corollary}
\begin{proof} 
Suppose $\nu$ is a $F$-lift of $\mu$. So, 
we have that $\nu\ll\mu$ and it is $F$-invariant. 
By the $F$ invariance of $\nu$, we get that $\nu(\ca_F)=1$ and so, as $\nu\ll\mu$, $\mu(\ca_F)>0$.
Thus, it follows from Theorem~\ref{HalfLifting} that $\nu:=\frac{1}{\mu(\ca_{F})}\mu|_{\ca_{F}}$ is a $F$-invariant probability.
As $\nu(\ca_F)=1$ and $\nu\ll\mu$, we get that 
 $\nu\ll\frac{1}{\mu(\ca_F)}\mu|_{\ca_F}$.
 By  Proposition~\ref{PropositionErGOS} $\mu$ is $F$-ergodic and so, by Lemma~\ref{LemmaSingSing}, $\frac{1}{\mu(\ca_F)}\mu|_{\ca_F}$ is $F$ ergodic. Hence, we must have $\nu=\frac{1}{\mu(\ca_F)}\mu|_{\ca_F}$.
\end{proof}

As the natural extension will be needed in the proof of Theorem~\ref{TheoremLift} and \ref{TheoremFiftErgDec}, let us set its notation.

\subsection{Natural extension}\label{SectionNatExt} Let $(\XX,\mathfrak{A})$ be a measurable space.
Let $\XX^{\infty}$ be the set of all maps $\overline{x}:\{0,1,2,3,\cdots\}\to\XX$. 
Let $\pi_n:\XX^\infty\to\XX$, $n\ge0$, be the projection $\pi_n(\overline{x})=\overline{x}(n)$ and define the {\bf\em natural projection} $\pi:\XX_f\to\XX$ by $\pi=\pi_0|_{\XX_f}$.
The {\bf\em domain of the natural extension of $f$} is the set  $$\XX_f=\{\overline{x}:\{0,1,2,3,\cdots\}\to\XX\,;\,f(\overline{x}(j+1))=\overline{x}(j),\,\forall\,j\ge0\}\subset\XX^\infty.$$
As a consequence of the definition of $\XX_f$, if $A\subset\XX_f$ then
$$f^j(\pi_{n+j}(A))=\pi_n(A)\text{ for every }n,j\ge0.$$
In particular,
$$\pi_{n+j}(A)\subset f^{-j}(\pi_n(A))\text{  for every }n,j\ge0.$$

Define the {\em cylinder} on $\XX_f$ generated by measurable sets $A_0,\cdots,A_n\in\mathfrak{A}$, $n\ge0$, as the set
$$[A_0,\cdots,A_n]:=\{\overline{x}\in\XX_f\,;\,(\overline{x}(0),\cdots,\overline{x}(n))\in A_0\times\cdots\times A_n\}.$$

Denote the set of all cylinders of $\XX_f$ by $\cy(\XX_f)$. Let $\overline{\mathfrak{A}}$ be the $\sigma$-algebra of subsets of $\XX_f$ generated by $\cy(\XX_f)$. The pair $(\XX_f,\overline{\mathfrak{A}})$ is a measurable space.  Note that \begin{equation}\label{Eqesfr4}
  \pi_{n+j}([A_0,\cdots,A_n])=f^{-j}(A_n)\;\;\forall\,j\ge0
\end{equation}
 and that  
\begin{equation}\label{Eghju765r654}
[A_0,A_1,\cdots,A_n]=\bigg[A_0,f^{-1}(A_0)\cap A_1, f^{-2}(A_0)\cap f^{-1}(A_1)\cap A_2,\cdots,\bigcap_{j=0}^{n-1}f^{-(n-j)}(A_j)\bigg]
\end{equation}

The {\bf\em natural extension of $f$} is the map $\overline{f}:\XX_f\to\XX_f$ given by $$\overline{f}((\overline{x}(0),\overline{x}(1),\overline{x}(2),\cdots))=(f(\overline{x}(0)),\overline{x}(0),\overline{x}(1),\overline{x}(2),\cdots).$$
It is easy to check that $\overline{f}$ is injective and that $f\circ \pi=\pi\circ\overline{f}$.
Furthermore, if $f$ is $\mathfrak{A}$-measurable then $\overline{f}$ is $\overline{\mathfrak{A}}$-measurable and $\pi$ is $(\mathfrak{A},\overline{\mathfrak{A}})$-measurable (i.e., $\pi^{-1}(A)\in\overline{\mathfrak{A}}$ whenever $A\in\mathfrak{A}$). Moreover, if $f$ is bimeasurable (with respect to $\mathfrak{A}$) then $\overline{f}$ is bimeasurable (with respect to $\overline{\mathfrak{A}}$).

We give a proof of Rokhlin result (Proposition~\ref{PropRokhlin} below) about ``lifting'' an invariant measure to $\overline{f}$ in Appendix. 

\begin{Proposition}[Rokhlin]\label{PropRokhlin}
		If $\mu$ is a $f$-invariant probability, then the probability $\overline{\mu}$ on $\XX_f$ defined by $\overline{\mu}(U)=\lim_{n\to\infty}\mu(\pi_n(U))$, $\forall\,U\subset\XX_f$ measurable, is the unique $\overline{f}$-invariant probability $\overline{\mu}$ such that $\mu=\pi_*\overline{\mu}$. Furthermore, if $\mu$ is ergodic then $\mu$ is also ergodic.
	\end{Proposition}

\subsection{Proofs of Theorem~\ref{TheoremLift} and \ref{TheoremFiftErgDec}}
\begin{proof}[Proof of Theorem~\ref{TheoremLift}]
(1)$\implies$(2) Suppose that $\nu$ is a $F$-lift of $\mu$. In this case, $\nu$ is $F$-invariant and $\int Rd\nu<+\infty$. By Birkhoff Theorem, the limit $\lim_n\frac{1}{n}\sum_{j=0}^{n-1}R\circ F^j(x)$ exists and belongs to $[1,+\infty)$, for $\nu$ almost every $x\in A_0$.
Thus, it follows from Lemma~\ref{LemmaCalculoBOM}, that $\theta_F(x)=\big(\lim_n\frac{1}{n}\sum_{j=0}^{n-1}R\circ F^j(x)\big)^{-1}>0$ for $\nu$ almost every $x\in A_0$. As $\nu\ll\mu$, we get (2).

(2)$\implies$(3) This implication follows from Lemma~\ref{LemmaCalculoBOM}.

(3)$\implies$(4) As $\mu(\{x\in A_0\,;\, \liminf_{n}\frac{1}{n}\sum_{j=0}^{n-1}R\circ F^{j}(x)<+\infty\big\})>0$, there is $1\le\gamma<+\infty$ such that $\mu(A(\gamma))>0$, where $$A(\gamma)=\bigg\{x\in A_0\,;\,\liminf_{n}\frac{1}{n}\sum_{j=0}^{n-1}R\circ F^{j}(x)\le\gamma\bigg\}.$$

Let $\overline{f}:\XX_f\to\XX_f$ be the natural extension of $f$ and define $\overline{F}(\overline{x})=\overline{f}^{\overline{R}(\overline{x})}(\overline{x})$, where $\overline{R}(\overline{x})=R(\pi(\overline{x}))$ and $\pi$ is that natural projection, see Section~\ref{SectionNatExt} above for more details.
	By Rokhlin (see Proposition~\ref{PropRokhlin}), there is a $\overline{f}$ invariant probability $\overline{\mu}$ such that $\pi_*\overline{\mu}=\mu$.
	As $f$ is bimeasurable, $\overline{f}$ is bimeasurable and injective.       

Note that $\overline{A}_0:=\bigcap_{n\ge0}\overline{F}^{-n}(\XX_f)=\pi^{-1}(A_0)$ and $$\overline{A}(\gamma):=\bigg\{\overline{x}\in \overline{A}_0\,;\,\liminf_{n}\frac{1}{n}\sum_{j=0}^{n-1}\overline{R}\circ \overline{F}^{j}(\overline{x})\le\gamma\bigg\}=\pi^{-1}(A(\gamma)).$$

For each $\overline{x}\in\overline{A}(\gamma)$, let $\NN(\overline{x})=\{n\ge0\,;\,\frac{1}{n}\sum_{j=0}^{n-1}\overline{R}\circ\overline{F}^j(\overline{x})\le2\gamma\}$ and $r_{\overline{x}}(n)=\sum_{j=0}^{n-1}\overline{R}\circ\overline{F}^j(\overline{x})$.
Thus, if $n\in\NN(\overline{x})$ and $V\subset\XX_f$, we get  
$$\bigg(\frac{1}{n}\sum_{j=0}^{n-1}\delta_{\overline{F}^j(\overline{x})}\bigg)(V)=\frac{r_{\overline{x}}(n)}{n}\;\;\frac{\#\{0\le j<n\,;\,\overline{F}^j(\overline{x})\in V\}}{r_{\overline{x}}(n)}\le$$
$$\le2\gamma\frac{\#\{0\le j<n\,;\,\overline{F}^j(\overline{x})\in V\}}{r_{\overline{x}}(n)}=2\gamma \frac{\#\{0\le j<r_{\overline{x}}(n)\,;\,\overline{f}^j(\overline{x})\in\co_{\overline{F}}^+(\overline{x})\cap V\}}{r_{\overline{x}}(n)}\le$$
$$\le2\gamma \frac{\#\{0\le j<r_{\overline{x}}(n)\,;\,\overline{f}^j(\overline{x})\in V\}}{r_{\overline{x}}(n)}=2\gamma\bigg(\frac{1}{r_{\overline{x}}(n)}\sum_{j=0}^{r_{\overline{x}}(n)-1}\delta_{\overline{f}^j(\overline{x})}\bigg)(V).$$
Hence, as $\overline{F}(\overline{A}(\gamma))\subset\overline{A}(\gamma)$, we get that 
$\lim_{k\to\infty}(\frac{1}{k}\sum_{j=0}^{k-1}\delta_{\overline{F}^j(\overline{x})})\big(\,\overline{F}^\ell\big(\,\overline{A}(\gamma)\big)\big)=1$ $\forall\,\ell\ge0$. So, it follows from Birkhoff that 
\begin{equation}\label{Equationkuytjhgfu}
\overline{\mu}(\overline{F}^{\ell}(\overline{A}(\gamma)))\ge\frac{1}{2\gamma}\;\;\forall\,\ell\ge0\text{ and $\overline{\mu}$ almost every }\overline{x}\in\overline{A}(\gamma).
\end{equation}

Thus, taking $\ca_{\overline{F}}=\bigcap_{n\ge0}\overline{F}^n\big(\,\overline{A}_0\big)$, we have by Lemma~\ref{LemmadeMEDIDA} of Appendix that 
$$\overline{\mu}(\ca_{\overline{F}})\ge\overline{\mu}(\ca^*)\ge\frac{1}{2\gamma},$$ where $\ca^*:=\bigcap_{n\ge0}\overline{F}^n(\overline{A}(\gamma)\big)$.
Therefore, it follows from Theorem~\ref{HalfLifting} that $\overline{\nu}_0
:=\frac{1}{\overline{\mu}(\ca_{\overline{F}})}\overline{\mu}|_{\ca_{\overline{F}}}
$ is a $\overline{F}$-invariant probability.
As $\overline{F}(\ca^*)\subset \ca^*$, we get that $$\overline{\nu}=\frac{1}{\overline{\nu}_0(\ca^*)}\overline{\nu}_0|_{\ca^*}=\frac{1}{\overline{\mu}(\ca^*)}\overline{\mu}|_{\ca^*}$$
is a $\overline{F}$-invariant probability with $\overline{\nu}\le C\overline{\mu}$, where $C=1/\overline{\mu}(\ca^*)$.
As, by Birkhoff, $\int\overline{R} d\overline{\nu}\le C\gamma<+\infty$, we get that $\overline{\nu}$ is a $\overline{F}$-lift of $\overline{\mu}$ and so, $\nu:=\pi_*\overline{\nu}=\overline{\nu}\circ\pi^{-1}$ is a $F$-lift of $\mu$ with $\nu\le C\mu$.
Furthermore, if $F$ is orbit-coherent then $\overline{F}$ is also orbit-coherent and, by Corollary~\ref{CorollaryINTEGRAL},  $\overline{\nu}$ is $\overline{F}$-ergodic and the unique $\overline{F}$-lift of $\overline{\mu}$.
This implies the same to $\nu$, i.e., $\nu$ is the unique $F$-lift of $\mu$ and it is $F$-ergodic. 

As (4)$\implies$(1) is immediate, we conclude the proof of the theorem.
\end{proof}

The result below (Theorem~\ref{TheoremZWEIM000000}) is closely related with Zweimuller's result \cite{Zw}.
Indeed, in \cite{Zw}, Zweimuller consider induced times  $\tau:A\to\NN\cup\{+\infty\}$ and maps $F:A^*\to A\subset\XX$, where $F(x)=f^{\tau(x)}(x)$ and $A^*=\{\tau<+\infty\}$.
Note that, $\int \tau d\mu<+\infty$ implies that $0<\int_{A_0}\tau d\mu<+\infty$, where $A_0=\bigcap_{n\ge0}F^{-n}(\XX)$. That is, if the map $f$ is bimeasurable, Theorems~1.1 of \cite{Zw} is a Corollary of Theorem~\ref{TheoremZWEIM000000}.

\begin{T}\label{TheoremZWEIM000000}
Let $(\XX,\mathfrak{A})$ be a measurable space, $f:\XX\to\XX$ a bimeasurable map and $F:A\to\XX$ a measurable $f$ induced map with induced time $R:A\to\NN$.
Let $\mathfrak{A}\ni V\subset A_0:=\bigcap_{n\ge0}F^{-n}(\XX)$ be a $F$-forward invariant set. 
If $\mu$ is an ergodic $f$-invariant probability and  $0<\int_{V} R d\mu<+\infty$, then $\mu$ is $F$-liftable.
Moreover, there exist $C>0$ and a finite collection of ergodic $F$-invariant probabilities $\nu_1,\cdots,\nu_n$
such that $\nu_j\le C\mu|_V$ $\forall\,j$ and each $F$-lift $\nu$ with $\mu(V)=1$ can be written as
$$\nu=\sum_{j=1}^n a_j\nu_j$$
for some $a_1,\cdots,a_n\in[0,1]$ \end{T}
\begin{proof}
First assume that $f$ is injective. 
Let $\cw$ be the collection of all  $F$-forward invariant Borel subset of $V$.
\begin{Claim}\label{Claimlkjuyfidutyerytiuu}
	There exists $\delta_0>0$ such that $\mu(U)=0$ or $\mu(U)\ge\delta_0$ for every $U\in\cw$.
\end{Claim}
\begin{proof}[Proof of the claim]
It follows from Lemma~\ref{LemmaGenInd1} that $f(\widetilde{U})\subset\widetilde{U}=\bigcup_{n\ge1}\bigcup_{j=0}^{n-1}f^j\big(U\cap \{R=n\}\big)$ whenever $F(U)\subset U\subset A$.
So, if $U\in\cw$ has $\mu(U)>0$, it follows from the ergodicity of $\mu$ that $\mu(\widetilde{U})=1$. Nevertheless, as $\mu(f(B))=\mu(B)$ for every Borel set $B$, we get $$1=\mu(\widetilde{U})\le\bigcup_{n\ge1}\bigcup_{j=0}^{n-1}\mu\big(f^j\big(U\cap \{R=n\}\big)\big)=\sum_{n\ge1}\sum_{j=0}^{n-1}\mu(U\cap \{R=n\}\big)=\int_{U}R d\mu.$$
Thus, as $\int_{V}R d\mu<+\infty$ by hypothesis, there exists $\delta_0>0$ such that $\mu(U)\ge\delta_0$ for every $U\in\cw$ with $\mu(U)>0$.
\end{proof}

The first consequence of Claim~\ref{Claimlkjuyfidutyerytiuu} is that  $\mu(\cv_F)>0$, where $\cv_F=\bigcap_{n\ge0}F^n(V)$. Indeed,  since $\int_{V}R d\mu>0$, we get that $\mu(V)>0$. 
As $V\supset F(V)\supset F^2(V)\supset \cdots\supset F^n(V)\nearrow \bigcap_{j\ge0}F^j(V)=\cv_F$, if $\mu(\cv_F)=0$ then, by Lemma~\ref{LemmadeMEDIDA} of Appendix, there exists $n\ge1$ big enough so that $\mu(F^n(V))<\delta_0/2$.
As $F$ is $\mu$ non-singular (Lemma~\ref{LemmaSingSing}), we also have $\mu(F^n(V))>0$ and this contradicts the claim above, as $F^n(V)\in\cw$.

As $\mu(\ca_F))\ge\mu(\cv_F)>0$, where $\ca_F=\bigcap_{n\ge0}F^n(A_0)$, 
it follows from Theorem~\ref{HalfLifting} that $\frac{1}{\mu(\ca_F)}\mu|_{\ca_F}$ is a $F$-invariant probability. Since $\mu(\cv_F)>0$ and $F(\cv_F)\subset\cv_F\subset\ca_F$, $\frac{1}{\mu(\cv_F)}\mu|_{\cv_F}$ is also a $F$-invariant probability.

A second consequence of Claim~\ref{Claimlkjuyfidutyerytiuu} is that $\mu|_{\cv_F}$ can have at most $1/\delta_0$ $F$-ergodic components.
That is, there exist $1\le n\le1/\delta_0$ and measurable sets $U_1,\cdots,U_n\subset\cv_F$ such that $F(U_j)=U_j\mod\mu$, $\mu(U_j)>0$, $\mu|_{U_j}$ is $F$-ergodic and $\mu|_{\cv_F}=\sum_{j=1}^n \mu|_{U_j}$ (see, for instance, Proposition~3.12 in \cite{Pi11}).

Taking $\nu_j:=\frac{1}{\mu(U_j)}\mu_j|_{U_j}$, we get that is an ergodic $F$-invariant probability and $\int R d\nu_j=\frac{1}{\mu(U_j)}\int_{U_j}R d\mu\le \int_{V}R d\mu/\mu(U_j)<+\infty$.
That is, each $\nu_j$ is a $F$-lift of $\mu$.
On the other hand, if $\nu$ is a $F$-lift of $\mu$ and $\mu(V)=1$, then $\nu(\cv_F)=1$ (Lemma~\ref{LemmaTeoErg876y7ui} of Appendix).
Thus, $\nu\ll\mu|_{\cv_F}$ and $\nu(\cv_F)=\sum_{j=1}^n\nu(U_j)=1$.
Hence, $\frac{1}{\nu(U_j)}\nu|_{U_j}$ is $F$-ergodic probability wherever $\nu(U_j)>0$.
As $\nu|_{U_j}\ll\mu|_{U_j}$, it follows from the ergodicity (and $F$-invariance) that
$\frac{1}{\nu(U_j)}\nu|_{U_j}=\frac{1}{\mu(U_j)}\mu|_{U_j}$
whenever $\nu(U_j)>0$.
Thus, $$\nu=\sum_{j=1}^{n}\nu|_{U_j}=\sum_{j=1}^{n}a_j \frac{1}{\mu(U_j)}\mu|_{U_j},$$
where $a_j=\nu(U_j)$, concluding the proof of the theorem when $f$ is injective.

The case when $f$ is non injective follows straightforward from the injective case applied to the natural extension $\overline{f}:\XX_f\to\XX_f$ of $f$ and, as in the proof of Proof of Theorem~\ref{TheoremLift},  the induced map is given by $\overline{F}(\overline{x})=\overline{f}^{\overline{R}(\overline{x})}(\overline{x})$, where $\overline{R}(\overline{x})=R(\pi(\overline{x}))$ and $\pi$ is the natural projection. \end{proof}

\begin{proof}[Proof of Theorem~\ref{TheoremFiftErgDec}]
As in the proof of Theorem~\ref{TheoremZWEIM000000}, we may assume that $f$ is injective, since the result when  $f$ is not injective follows from the result for the natural extension  of $f$.
Hence, suppose that $f$ is infective and $\mu$ is $F$-liftble.  Let $\nu$ be  a $F$-lift of $\mu$.
So, $\nu$ is $F$-invariant, $\nu\ll\mu$ and $\int R d\nu<+\infty$.
As $\nu$ is $F$-invariant, it follows from Lemma~\ref{LemmaTeoErg876y7ui} of Appendix that $\nu(\ca_F)=1$, where $\ca_F=\bigcap_{n\ge0}F^n(A_0)$ and $A_0=\bigcap_{j\ge0}F^{-j}(\XX)$.
Moreover, as $\int R d\nu<+\infty$, we get by Birkhoff that $\lim_n\frac{1}{n}\sum_{j=0}^{n-1}R\circ F^j(x)<+\infty$ for $\nu$ almost every $x$.
Hence, $\nu(\ca_F^*)=1$, with $A_F^*:=\big\{x\in \ca_F\,;\, \liminf_{n}\frac{1}{n}\sum_{j=0}^{n-1}R\circ F^{j}(x)<+\infty\big\}$.
As a consequence,  $\mu(\ca_F^*)>0$, as $\nu\ll\mu$.

Because $\ca_F^*=\bigcup_{\ell\in\NN}V_\ell$, where $$V_\ell=\bigg\{x\in\ca_F^*\,;\,\liminf_{i}\frac{1}{i}\sum_{j=0}^{i-1}R\circ F^j(x)\in[\ell, \ell+1)\bigg\},$$
we get that $\cn=\{\ell\in\NN\,;\,\mu(V_{\ell})>0\}$ is a nonempty set.

As we are assuming that $f$ is injective and $\mu$ $F$-liftble, it follows from Theorem~\ref{HalfLifting} and Birkhoff that $1\le\ell\le \int_{V_{\ell}}R d\mu\le\ell+1<+\infty$. 

Since $F(V_\ell)\subset V_\ell$, if $\ell\in\cn$ then $\nu_{\ell}=\frac{1}{\nu(V_{\ell})}\nu|_{V_{\ell}}$ is a $F$-invariant probability and $\nu=\sum_{\ell\in\cn}\nu(V_{\ell})\nu_{\ell}$.

As $V_{\ell}$ is $F$-forward invariant and $0<\int_{V_\ell}R d\mu<+\infty$ when $\ell\in\cn$, it follows from Theorem~\ref{TheoremZWEIM000000} that, if $\ell\in\cn$ then there exists a finite collection of ergodic $F$-invariant probabilities $\eta_{\ell,1},\cdots,\eta_{\ell,n_{\ell}}$ and a constant $0<c_{\ell}<+\infty$ such that $\eta_{\ell,j}\le c_{\ell}\,\mu|_{V_{\ell}}$.
Moreover, every $F$-lift $\eta$ of $\mu$ with $\eta(V_{\ell})=1$ can be uniquely written as a convex combination of $\eta_{\ell,1},\cdots,\eta_{\ell,n_{\ell}}$.
In particular, there exists $0\le c_{\ell,j}\le 1$, with $\sum_{j=1}^{n_{\ell}}c_{\ell,j}=1$, such that $\nu_{\ell}=\sum_{j=1}^{n_{\ell}}c_{\ell,j}\eta_{\ell,j}$.
Therefore, we can write
\begin{equation}\label{Equation7675dfjuju8uy}
  \nu=\sum_{\ell\in\cn}\nu(V_{\ell})\,\nu_{\ell}=\sum_{\ell\in\cn}\nu(V_{\ell})\,\bigg(\sum_{j=1}^{n_{\ell}}c_{\ell,j}\eta_{\ell,j}\bigg)=\sum_{\ell\in\cn}\sum_{j=1}^{n_{\ell}}\nu(V_{\ell})\,c_{\ell,j}\,\eta_{\ell,j}.
\end{equation}

As $1=\nu(\XX)=\sum_{\ell\in\cn}\sum_{j=1}^{n_{\ell}}\nu(V_{\ell})\,c_{\ell,j}\,\eta_{\ell,j}(\XX)=\sum_{\ell\in\cn}\sum_{j=1}^{n_{\ell}}\nu(V_{\ell})\,c_{\ell,j}$, we get that (\ref{Equation7675dfjuju8uy}) is a convex combination of the countable collection of $F$-ergodic probabilities $\{\eta_{\ell,j}\}_{\ell,j}$. As $\eta_{\ell_1,j_1}$ and $\eta_{\ell_2,j_2}$ are mutual singular when $(\ell_1,j_1)\ne(\ell_2,j_2)$, we get that the convex combination is unique.

Finally, if $\int R d\mu<+\infty$ then we can apply Theorem~\ref{TheoremZWEIM000000} to $V=A_0$ and get the finiteness of the collection $\{\eta_{\ell,j}\}_{\ell,j}$. \end{proof}

\section{Coherent schedules and Pliss times}\label{SecPlissESy}

In all this section, $(\XX,\mathfrak{A})$ is a measurable space and $f:\XX\to\XX$ is a measurable map. For ease of reading, we rewrite here some of the definitions already presented in the introduction (Section~\ref{IntroductionAndStatementsOfMainResults}).
Let $2^{\NN}$ be the {\em power set} of the natural numbers, i.e., the set of all subsets of $\NN=\{1,2,3,\cdots\}$. Consider $2^{\NN}$ with the following metric $$\dist(U,V)=\begin{cases}
	2^{-\min(U\triangle V)} & \text{ if }A\ne B\\
	0 & \text{ if }A=B
\end{cases}$$
where $A\triangle B=(\NN\setminus A)\cap(\NN\setminus B)$ is the symmetric difference of the sets $A$ and $B$.  
Note that $i:2^{\NN}\to\Sigma_2^+=\{0,1\}^{\NN}$ given by 
\begin{equation}\label{EquationItinerario}
  i(U)(n)=
\begin{cases}
	1 &\text{ if }n\in U\\
	0 &\text{ if }n\notin U
\end{cases}
\end{equation}
is an isomorphism between $(2^{\NN},\dist)$ and $(\Sigma_2^+,\delta)$, where the distance $\delta$ on $\Sigma_2^+$ is the usual one, that is, $\delta(x,y)=2^{-\min\{j\ge1\,;\,x(j)\ne y(j)\}}$. In particular, $(2^{\NN},\dist)$ is a compact metric space. Consider $2^{\NN}$ with the structure of the measurable space $(2^{\NN},\mathfrak{B})$, where $\mathfrak{B}$ is the $\sigma$-algebra of the Borel sets of $(2^{\NN},\dist)$.
We define the {\bf\em shift} of a set $U\in 2^{\NN}$ as $$\sigma U=(U-1)\cap\NN=\{j-1\,;\,2\le j\in U\}.$$

Given $U\subset\NN=\{1,2,3,\cdots\}$, the {\bf\em upper density of $U$} is defined as $$\dd_{\NN}^+(U)=\limsup_{n}\frac{1}{n}\#\{1\le j\le n\,;\,j\in U\}.$$
The {\bf\em lower density} of $U$ is $$\dd_{\NN}^-(U)=\liminf_{n}\frac{1}{n}\#\{1\le j\le n\,;\,j\in U\}.$$

If $\lim_{n}\frac{1}{n}\#\{1\le j\le n\,;\,j\in U\}$ exists, then we say that $$\dd_{\NN}(U):=\lim_{n}\frac{1}{n}\#\{1\le j\le n\,;\,j\in U\}$$
is the {\bf\em natural density} of $U$.

\begin{Definition}[Schedule of events]\label{Defineschedule}
A {\bf\em schedule of events} (or, for short, a {\bf\em schedule}) on $\XX$ is a measurable map  $\cu:\XX\to 2^{\NN}$. A schedule $\cu$ is called $f$ {\bf\em asymptotically invariant} if 
for each $x\in\XX$, $\exists\,a_x> b_x\ge0$ such that  $\sigma^{a_x}\cu(x)=\sigma^{b_x}\cu(f(x))$. 	An element $\ell$ of $\cu(x)$ is called a {\bf\em $\cu$-event} or a {\bf\em $\cu$-time} for $x$. 
\end{Definition}

One can ask many questions about the behavior of the map $f$ at  ``$\cu$-times'', i.e., $f^n(x)$ with $n\in\cu(x)$.
In particular, to ask about the existence of attractors and statistical attractors in $\cu$-times.
For instance, in the presence of an ergodic probability $\mu$ (not necessarily an invariant one), there is a compact set $\AA\subset\XX$, the attractor at $\cu$-times, such that $\AA=\omega_{\cu}(x)$ for $\mu$-almost every $x\in\XX$, where $\omega_{\cu}(x)$ is the set of accumulating points of $\co_{\cu}(x):=\{f^n(x)\,;\,n\in\cu(x)\}$.
In \cite{Pi11}, $\co_{\cu}(x)$ is called 
an {\em asymptotically invariant collection} and one can find the proof of the existence of those attractors in Lemma~3.9 of \cite{Pi11}. 
Here, we are more concerned with the problem of synchronizing two schedules, as defined below.

\begin{Definition}[Synchronization]
	Given two schedules $\cu_0$ and $\cu_1$ on $\XX$, we say that $\cu_0$ and $\cu_1$ are (statistically) {\bf\em  synchronizable} at a point $x\in\XX$ if there is $\ell\ge0$ such that $$\dd_{\NN}^+(\{j\in\NN\,;\,(j,j+\ell)\in\cu_0(x)\times\cu_1(x)\})=\dd_{\NN}^+(\cu_0(x)\cap\sigma^{\ell}\cu_1(x))>0.$$
	If there exist such a $x\in\XX$ and $\ell$, we say that $\cu_0$ and $\cu_1$ are {\bf\em $\ell$-synchronized at $x$}. 
\end{Definition}

If $\dd_{\NN}^-(\cu_0(x))+\dd_{\NN}^-(\cu_1(x))>1$, we always have that $\cu_0(x)$ and $\cu_1(x)$ are $\ell$-synchronized for every $\ell\ge0$.
On the other hand,  in general, we cannot expect that $\cu_0(x)$ and $\cu_1(x)$ are synchronizable when $\dd_{\NN}^+(\cu_0(x))+\dd_{\NN}^+(\cu_1(x))\le1$ or even when $\dd_{\NN}^-(\cu_0(x))+\dd_{\NN}^-(\cu_1(x))=1$.
In Example~\ref{ExampleNonSyn} below we have two continuous  schedules $\cu_0$ and $\cu_1$ that are not synchronizable at any point and such that $\dd_{\NN}(\cu_0(x))=\dd_{\NN}(\cu_1(x))=1/2$ for every $x\in\XX\setminus\{0\}$.

\begin{Example}\label{ExampleNonSyn}
	Let $\XX=\{2^{-n}\,;\,n\in\NN\}\cup\{0,1\}$ and set $f(x)=x/2$.
	  Set $x_0$ and $y_0\in\Sigma_2^+$ by
$$x_0=(1,0,
\underbrace{1,1}_{2\text{ times}},
\underbrace{0,0}_{2\text{ times}},
\underbrace{1,1,1}_{3\text{ times}},
\underbrace{0,0,0}_{3\text{ times}},
\underbrace{1,1,1,1}_{4\text{ times}},
\underbrace{0,0,0,0}_{4\text{ times}},
\underbrace{1,1,1,1,1}_{5\text{ times}},
\underbrace{0,0,0,0,0}_{5\text{ times}},
\cdots)$$
and
$$y_0=(0,1,
\underbrace{0,0}_{2\text{ times}},
\underbrace{1,1}_{2\text{ times}},
\underbrace{0,0,0}_{3\text{ times}},
\underbrace{1,1,1}_{3\text{ times}},
\underbrace{0,0,0,0}_{4\text{ times}},
\underbrace{1,1,1,1}_{4\text{ times}},
\underbrace{0,0,0,0,0}_{5\text{ times}},
\underbrace{1,1,1,1,1}_{5\text{ times}},
\cdots)$$
Let, for $\ell\ge1$, $x_{\ell}\in\Sigma_2^+$ be given by $i(\ell+i^{-1}(x_0))$, where $\ell+U=\{u+\ell\,;\,u\in U\}\in2^{\NN}$ for any $U\in2^{\NN}$. That is, 
$$x_{\ell}(n)=
\begin{cases}
	0 & \text{ if }n\le\ell\\
	x_0(n-\ell) & \text{ if }n>\ell
\end{cases}.$$
Set also $y_{\ell}=i(\ell+i^{-1}(y_0))$ for every $\ell\ge1$. 
Set $\cu_0(0)=\emptyset$, $\cu_0(1)=i^{-1}(x_0)$ and $\cu_0(2^{-\ell})=\ell+i^{-1}(x_0)=i^{-1}(x_{\ell})$ for any $\ell\in\NN$. Similarly, let $\cu_1(0)=\emptyset$, $\cu_1(1)=i^{-1}(y_0)$ and $\cu_1(2^{-\ell})=\ell+i^{-1}(y_0)=i^{-1}(y_{\ell})$ for any $\ell\in\NN$. 

 It is not difficult to check that $\dd_{\NN}(\cu_0(p))=\dd_{\NN}(\cu_1(p))=\frac{1}{2}$ for every $p\ne0$. Nevertheless, 
 $\dd_{\NN}^+(\cu_0(p)\cap\sigma^{\ell}\cu_1(p))=0$ for every $\ell\ge1$ and $p\in\XX$.
 Indeed, $\dd_{\NN}^+(\cu_0(0)\cap\sigma^j\cu_1(0))=0$, as $\cu_0(0)=\emptyset$.
 So assume that $p\ne0$.
 Define the product $\cdot:\Sigma_2^+\times\Sigma_2^+\to\Sigma_2^+$ by $(x\cdot y)(n)=x(n)\cdot y(n)$.
 Note that $i(U\cap V)=i(U)\cdot i(V)$ for every $U,V\in2^{\NN}$.
 As, for $n>\ell$, we have $$z:=x_0\cdot \sigma^{\ell}y_0=\big(z(1),\cdots,z(n(n-1)),$$
 $$\overbrace{\,\underbrace{0,\cdots,0}_{n-\ell\text{ times}},\underbrace{1,\cdots,1}_{\ell\text{ times}},\underbrace{0,\cdots,0}_{n\text{ times}}\,}^{2n\text{ times}},\overbrace{\underbrace{0,\cdots,0}_{n+1-\ell\text{ times}}\underbrace{1,\cdots,1}_{\ell\text{ times}},
 \underbrace{0,\cdots,0}_{n+1\text{ times}}\;\,}^{2(n+1)\text{ times}},
\overbrace{\underbrace{0,\cdots,0}_{n+2-\ell\text{ times}}\underbrace{1,\cdots,1}_{\ell\text{ times}},
 \underbrace{0,\cdots,0}_{n+2\text{ times}}\;\,}^{2(n+2)\text{ times}},\cdots\big),
 $$
we can conclude that $\dd_{\NN}^+(\cu_0(1)\cap\sigma^{\ell}\cu_1(1))=\limsup_{n}\frac{1}{n}\sum_{j=1}^n z(j)=0$ for every $\ell\ge1$.
Thus, $\dd_{\NN}^+(\cu_0(2^{-j})\cap\sigma^{\ell}\cu_1(2^{-j}))=\dd_{\NN}^+(\sigma^j(\cu_0(1)\cap\sigma^{\ell}\cu_1(1)))=\dd_{\NN}^+(\cu_0(1)\cap\sigma^{\ell}\cu_1(1))=0$, proving that $\dd_{\NN}^+(\cu_0(p)\cap\sigma^{\ell}\cu_1(p))=0$ for every $\ell\ge1$ and $p\in\XX$.
\end{Example}

As $(2^\NN,\dist)$ is a (perfect) totally disconnected compact metric space, if $\XX$ is a connect topological space, then any continuous schedule of events on $\XX$ must be a constant. A large class of schedules having a mild continuity is the class of partially continuous schedules of events defined below. 

\begin{Definition}
We say that a schedule of events $\cu$ on a metric space $\XX$ is {\bf\em partially continuous} if
$$\ell\in\cu(x_n)\;\;\forall\,n\implies\ell\in\cu(\lim_nx_n),$$
for any convergent sequence $x_n\in\XX$ and $\ell\in\NN$.
\end{Definition}
Note that any continuous schedule is partially continuous.  
One way to produces partially continuous schedules, whenever $f$ is continuous, is by using Birkhoff's averages. 
 Given a continuous function $\varphi:\XX\to\RR$ and $\gamma\in\RR$, let
\begin{equation}\label{Equationnh7gh}
  \cu(x)=\bigg\{n\ge1\,;\,\frac{1}{n}\sum_{j=0}^{n-1}\varphi\circ f^j(x)\ge\gamma\bigg\}.
\end{equation}
Note that $\cu$ is partially continuous and if $f$ has some complexity (i.e., positive topological entropy), one can find $\varphi$ and $\gamma$ to obtain a non constant schedule $\cu$ using (\ref{Equationnh7gh}).

\begin{Definition}[Coherent schedule]\label{DefPlissschedule}
A schedule of events $\cu$ on $\XX$ is called {\bf\em $f$-coherent} if
\begin{enumerate}
	\item[(P1)] $n\in\cu(x)\implies n-j\in\cu(f^j(x))$ for every $x\in\XX$ and $n>j\ge0$;
	\item [(P2)] $n\in\cu(x)$ and $m\in\cu(f^n(x))$ $\implies$  $n+m\in\cu(x)$ for every $x\in\XX$ and $n,m\ge1$.
\end{enumerate} 
\end{Definition}

\begin{Lemma}\label{Lemmakiuyghjko76}
If $\cu$ is a $f$-coherent schedule of events then $\sigma^{a_x}\cu(x)=\sigma^{a_x-1}\cu(f(x))$ for every $x\in\XX$ with $\cu(x)\ne\emptyset$, where $a_x=\min\cu(x)$.
In particular, $$\dd_{\NN}^-(\cu(x))=\dd_{\NN}^-(\cu(f(x)))\;\text{ and }\;\dd_{\NN}^+(\cu(x))=\dd_{\NN}^+(\cu(f(x)))$$
for every $x\in\XX$ with $\cu(x)\ne\emptyset$.
\end{Lemma}
\begin{proof}
Suppose that $\cu(x)\ne\emptyset$ for some $x\in\XX$ and let $a_x=\min\cu(x)$. First, let us show that $\sigma^{a_x}\cu(x)\subset\sigma^{a_x-1}\cu(f(x))$. If $n\in\sigma^{a_x}\cu(x)$ then $n+a_x\in\cu(x)$ and so, it follows from (P1) that $a_x+n-1\in\cu(f(x))$. As a consequence, $n\in\sigma^{a_x-1}\cu(f(x))$. Conversely, if $n\in\sigma^{a_x-1}\cu(f(x))$ then $n+a_x-1\in\cu(f(x))$. Thus, by (P1), $n=n+a_x-1-(a_x-1)\in\cu(f^{a_x-1}(f(x)))=\cu(f^{a_x}(x))$. Finally, as $a_x\in\cu(x)$ and $n\in\cu(f^{a_x}(x))$, it follows from (P2) that $a_x+n\in\cu(x)$ and so, $n\in\sigma^{a_x}\cu(x)$.
\end{proof}

There are many examples of coherent schedules. For instance, if $B\subset\XX$ is measurable and $A=\{x\in\XX\,;\,\co_f^+(f(x))\in B\}$, then
$$\cu(x)=
\begin{cases}
	\emptyset & \text{ if }x\notin A\\
	\{j\ge1\,;\,f^j(x)\in B\} & \text{ if }x\in A
\end{cases}
$$ is a $f$-coherent schedule.
It is easy to check that the intersection of two coherent schedules satisfies  (P1) and (P2). that is, if $\cu_1$ and $\cu_2$ are $f$-coherent schedules, then $\cu_1\cap\cu_2$ is a $f$-coherent schedule, where $(\cu_1\cap\cu_2)(x)=\cu_1(x)\cap\cu_2(x)$.
The translation to the left $\sigma^{\ell}\cu$ of a coherent schedule $\cu$ is a coherent schedule.
In general the union $\cu_1\cup\cu_2$ of two coherent schedules is not  coherent, where $(\cu_1\cup\cu_2)(x)=\cu_1(x)\cup\cu_2(x)$. Nevertheless, the union of a finite number of translations to the left of a coherent schedule is coherent, as one can see in Lemma~\ref{LemmaUniauyg6} below.

\begin{Lemma}\label{LemmaUniauyg6}
	If $\cu$ is a $f$-coherent schedule then $\sigma^{\ell_1}\cu\cup\cdots\cup\sigma^{\ell_m}\cu$ is coherent for any $0\le\ell_1,\cdots,\ell_m\in\ZZ$. 
\end{Lemma}
\begin{proof}
First we claim that if $\cv$ is a coherent schedule then $\cw:=\cv\cup\sigma^{\ell}\cv$ is coherent for any $\ell\ge1$.
Indeed, 
suppose that  $j\in\cw(x)$ and $t\in\cw(f^j(x))$.
If $j\in\cv(x)$ and $t\in\cv(f^{j}(x))$ we get $j+t\in\cv(x)\subset\cw(x)$ from the coherence of $\cv$.
For the same reasoning, $j+t\in\sigma^{\ell}\cv(x)\subset\cw(x)$ if $j\in\sigma^{\ell}\cv(x)$ and $t\in\sigma^{\ell}\cv(f^{j}(x))$.
If $j\in\cv(x)$ and $t\in\sigma^{\ell}\cv(f^{j}(x))$, we have that $j\in\cv(x)$ and $t+\ell\in\cv(f^j(x))$ and so, by the coherence of $\cv$, $j+t+\ell\in\cv(x)$.
Hence, $j+t\in\sigma^{\ell}\cv(x)\subset\cw(x)$.
Similarly, one can show that $j+t\in\sigma^{\ell}\cv(x)\subset\cw(x)$ when $j\in\sigma^{\ell}\cv(x)$ and $t\in\cv(f^{j}(x))$, proving that $\cw$ satisfies (P2).
Let $n\in\cw(x)$ and $0\le j<n$.
If $n\in\cv(x)$, by the coherence of $\cv$, we have that $n-j\in\cv(f^j(x))\subset\cw(f^j(x))$.
If $n\in\sigma^{\ell}\cv(x)$, then $n+\ell\in\cv(x)$ and again, by the coherence of $\cv$, $n+\ell-j\in\cv(f^j(x))$.
Thus, $n-j\in\sigma^{\ell}(f^j(x))\subset\cw(f^j(x))$, proving that $\cw$ satisfies (P1) and concluding the proof of the claim.

{\em Proof that $\cw:=\sigma^{\ell_1}\cu\cup\cdots\cup\sigma^{\ell_m}\cu$ satisfies (P2)}. 
If $j\in\cw(x)$ and $t\in\cw(f^j(x))$ then 
there are $k$ and $i\in\{1,\cdots,m\}$ such that $j\in\sigma^{\ell_k}\cu(x)$ and $t\in\sigma^{\ell_i}\cu(f^j(x))$.
Suppose that $\ell_k\le\ell_i$, the proof for the other case is similar.
In this case writing $\cv:=\sigma^{\ell_k}\cu$ and $\ell=\ell_i-\ell_k$, we get that $j\in\cv(x)\subset \cv\cup\sigma^{\ell}\cv(x)$ and $t\in\sigma^{\ell}\cv(f^j(x))\subset \cv\cup\sigma^{\ell}\cv(f^j(x))$.
As, by the claim, $\cv\cup\sigma^{\ell}\cv$ is coherent, we get that $j+t\in\cv\cup\sigma^{\ell}\cv(x)=\sigma^{\ell_k}\cu(x)\cup\sigma^{\ell_i}\cu(x)\subset\cw(x)$. Thus, $\cw$ satisfies (P2). 

{\em Proof that $\cw$ satisfies (P1)}. Let $n\in\cw(x)$ and $0\le j<n$. We have that $n\in\sigma^{\ell_k}\cu(x)$ for some $1\le k\le m$. The proof of $n-j\in\sigma^{\ell_k}\cu\subset\cw(x)$ is similar to the proof of (P1) in the claim.
\end{proof}

Lemma~\ref{PlissLemma} below, known as Pliss's Lemma (see \cite{Pl}), nevertheless of simple proof, turns out to be a useful tool in many problems in Dynamics.
In particular, one can use it to give examples of coherent schedules. 

A sequence of real numbers $a_n$ is called {\bf\em subadditive} if $a_{n+m}\le a_n+a_m$ for every $n$ and $m\ge1$. One can find a proof of Lemma~\ref{PlissLemma} below in Appendix.

\begin{Lemma}[(Subadditive) Pliss Lemma]\label{PlissLemma}
	Given $0<c_0<c_1\le C$, let $0<\theta=\frac{c_1-c_0}{C-c_0}<1$.
	Let $a_1,\cdots,a_{n}\in(-\infty, C]$ be a subadditive sequence of real numbers, i.e., $a_{i+j}\le a_i+a_j$.
	If $\frac{1}{n}a_n\ge c_1$, then there is $\ell\ge\theta n$ and a sequence $1<n_1<n_2<\cdots<a_{\ell}$ such that $\frac{1}{n_j-k}\,a_{n_j-k}\ge c_0$ for every $1\le j\le\ell$ and $0\le k<n_j$.
\end{Lemma}

\begin{Definition}[Pliss times]\label{DefinitionPlissTime}
	Given $\gamma\in\RR$ and a map $\varphi:\NN\times\XX\to\RR$,  we say that $n\ge1$ is a {\bf\em $(\gamma,\varphi)$-Pliss time} for $x\in\XX$, with respect to $f$, if
\begin{equation}\label{EqPlissTime}
  \frac{1}{n-k}\varphi(n-k,f^k(x))\ge\gamma\text{ for every }0\le k<n.
\end{equation}
\end{Definition}

\begin{Lemma}
Suppose that $\XX$ is a metric space. Let $\gamma>0$ and $\varphi:\NN\times\XX\to\RR$ a measurable map. Given $x\in\XX$, let $\cu(x)\subset2^{\NN}$ be the set of all $(\gamma,\varphi)$-Pliss times for $x$. If $f$ and $\varphi$ are continuous then $\cu:\XX\to2^{\NN}$ is a partially continuous schedule of events. 
\end{Lemma}
\begin{proof}
Let $p_{\ell}\in\XX$ be a sequence converging to $p\in\XX$ and suppose that $n\in\cu(p_{\ell})$ for every $\ell\ge1$. 
Suppose by contradiction that $n\notin\cu(p)$. So, there exists $0\le k<n$ such that $\frac{1}{n-k}\varphi(n-k,f^k(p))<\gamma$.
If $f$ and $\varphi$ are continuous, $\XX\ni x\mapsto\varphi(n-k,f^k(x))\in\RR$ is continuous and so,
\begin{equation}\label{Equationkjhghjk}
  \lim_{\ell\to\infty}\varphi(n-k,f^k(p_{\ell}))=\varphi(n-k,f^k(p))<\gamma.
\end{equation}
Nevertheless, as $n\in\cu(p_{\ell})$ for every $\ell\ge1$, we get $\varphi(n-k,f^k(p_{\ell}))\ge\gamma$ $\forall\ell\ge1$, which implies that $\lim_{\ell\to\infty}\varphi(n-k,f^k(p_{\ell}))\ge\gamma$ and leads to  a contradiction with (\ref{Equationkjhghjk}).\end{proof}

A measurable function $\varphi:\NN\times\XX\to\RR$ is called a {\bf\em $f$ subadditive cocycle}, a {\bf\em $f$ additive cocycle} or a {\bf\em $f$ sup-additive cocycle} if it satisfies, respectively, (1), (2) or (3) below.
\begin{enumerate}
	\item $\varphi(n+m,x)\le\varphi(n,x)+\varphi(m,f^n(x))$ for every $n,m\in\NN$ and $x\in\XX$.
	\item $\varphi(n+m,x)=\varphi(n,x)+\varphi(m,f^n(x))$ for every $n,m\in\NN$ and $x\in\XX$.
	\item $\varphi(n+m,x)\ge\varphi(n,x)+\varphi(m,f^n(x))$ for every $n,m\in\NN$ and $x\in\XX$.
\end{enumerate}

\begin{Lemma}\label{Lemma67ygtgh}
	Let $\varphi:\NN\times\XX\to\RR$ be a sup-additive cocycle.
	If $\gamma>0$ and $\cu(x)$ is the set of all $(\gamma,\varphi)$-Pliss times to $x$, then $\cu:\XX\to2^{\NN}$ is a $f$-coherent schedule of events.
\end{Lemma}
\begin{proof}
	Suppose that $n\in\cu(x)$ and $m\in\cu(f^n(x))$. Let $0\le k<n+m$.
If $k<n$ then, as $n\in\cu(x)$, it follows from (\ref{EqPlissTime}) that $\varphi(n-k,f^k(x))\ge(n-k)\gamma$.
As $m\in\cu(f^n(x))$, we also get from (\ref{EqPlissTime}) that $\varphi(m,f^n(x))\ge m\gamma$.
Hence, as $\varphi$ is a sup-additive cocycle, $$\varphi(m+n-k,f^k(x))\ge\varphi(n-k,f^k(x))+\varphi((m+n-k)-(n-k),f^{n-k}(f^k(x)))=$$
$$=\varphi(n-k,f^k(x))+\varphi(m,f^n(x))\ge(n+m-k)\gamma.$$
Similarly, we get that $\varphi(m+n-k,f^k(x))\ge(n+m-k)\gamma$ when $n\le k<n+m$.
Thus, $\cu(x)$ satisfies $(P2)$. 
As the proof of (P1) follows straightforward from (\ref{EqPlissTime}), we conclude that $\cu$ is a $f$-coherent schedule.

\end{proof}

\begin{Lemma}\label{Lemma67ygtghXXXX}
	Let $\varphi:\NN\times\XX\to\RR$ be a subadditive cocycle with $\sup\varphi<+\infty$.
	If $\gamma>0$ and $\cu(x)$ is the set of all $(\gamma/2,\varphi)$-Pliss times to $x$, then 
	\begin{enumerate}
		\item $\limsup\frac{1}{n}\varphi(n,x)\ge\gamma$ $\implies$ $\dd_{\NN}^+(\cu(x))\ge\frac{\gamma}{2(\sup\varphi)-\gamma}>0$;
		\item $\liminf\frac{1}{n}\varphi(n,x)\ge\gamma$ $\implies$ $\dd_{\NN}^-(\cu(x))\ge\frac{\gamma}{2(\sup\varphi)-\gamma}>0$.
	\end{enumerate}
\end{Lemma}

\begin{proof}

Taking $0<c_0:=\gamma/2<c_1:=\gamma<C:=\sup\varphi$, we get that $\theta=\frac{c_1-c_0}{C-c_0}=\frac{\gamma-\gamma/2}{\sup\varphi-\gamma/2}=\frac{\gamma}{2(\sup\varphi)-\gamma}>0$ and so, items (1) and (2) follow directly from the subadditive Pliss Lemma.
\end{proof}

\begin{Definition}\label{Definition654yuikj}
	Given a schedule of events $\cu:\XX\to2^{\NN}$, define the {\bf\em first $\cu$-time} $R_{\cu}:\XX\to\NN\cup\{+\infty\}$ by
$$R_{\cu}(x)=
\begin{cases}
\min\cu(x) & \text{ if }\cu(x)\ne\emptyset\\
+\infty & \text{ if }\cu(x)=\emptyset

\end{cases}.
$$
\end{Definition}

\begin{Lemma}\label{LemmaPArtuhgdj}
 If $\XX$ is a metric space and $\cu$ a partially continuous schedule of events to $f$, then $R_{\cu}:\XX\to\NN\cup\{+\infty\}$ is lower semicontinuous, i.e., $\liminf_{x\to p}R_{\cu}(x)\ge R_{\cu}(p)$ for every $p\in\XX$.
\end{Lemma}
\begin{proof} 
If $\liminf_{x\to p}R_{\cu}(x)<R_{\cu}(p)$ for some $p\in\XX$, then there is a sequence $p_{\ell}\in\XX$ with $\lim_{\ell}p_{\ell}=p$ such that $\lim_{\ell\to\infty}R_{\cu}(p_{\ell})<R_{\cu}(p)$.
As a consequence, $\lim_{\ell\to\infty}R_{\cu}(p_{\ell})=n<+\infty$, even if $R_{\cu}(p)=+\infty$.
Taking $\ell_0\ge1$ so that  $|R_{\cu}(p_{\ell})-n|<1/2$, we get that $R_{\cu}(p_{\ell})=n$ for every $\ell\ge\ell_0$. This implies, 
as $\cu$ is partially continuous, that $\cu(p)\ni n<R_{\cu}(p)=\min\cu(p)$ which is a contradiction.
\end{proof}

Given a $f$-coherent schedule $\cu:\XX\to2^{\NN}$ and  $j\in\NN\cup\{\infty\}$, denote
 $$\XX_{j}(\cu)=\{x\in\XX\,;\,\#\cu(x)\ge j\}.$$

The Lemmas below connect coherent induced times (Definition~\ref{DefPlissTime}) and  coherent schedules (Definition~\ref{DefPlissschedule}).

\begin{Lemma}\label{LemmaCoerenteToCoerente00}
	If $\cu:\XX\to 2^{\NN}$ is $f$-coherent, then $F(\XX_{\infty}(\cu))\subset f(\XX_{\infty}(\cu))\subset\XX_{\infty}(\cu)=\bigcap_{n\ge0}F^{-n}(\XX)$, where $F(x)=f^{R_{\cu}(x)}(x)$ is the first $\cu$-time map.
\end{Lemma}
\begin{proof}
	Let $x\in\XX_{\infty}(\cu)$.
	It follows from (P1) that $\cu(f(x))\supset\{n-1\,;\,2\le n\in\cu(x)\}$.
	So, $\#(\cu(f(x)))\ge\#\cu(x)-1=\infty$, proving that $f(\XX_{\infty}(\cu))\subset\XX_{\infty}(\cu)$.
	As $f^n(\XX_{\infty}(\cu))\subset\XX_{\infty}(\cu)$ $\forall\,n\ge1$, if $x\in\XX_{\infty}(\cu)$ then $F(x)=f^{R_{\cu}(x)}(x)\in\XX_{\infty}(\cu)$. 
	Thus, $F(\XX_{\infty}(\cu))\subset\XX_{\infty}(\cu)$ and so, $\XX_{\infty}(\cu)\subset F^{-1}(F(\XX_{\infty}(\cu)))\subset F^{-1}(\XX_{\infty}(\cu))$.
	Inductively, we get that $\XX_{\infty}(\cu)\subset F^{-n}(\XX_{\infty}(\cu))$ $\forall\,n\ge0$, showing that $\XX_{\infty}(\cu)\subset\bigcap_{n\ge0}F^{-n}(\XX_{\infty}(\cu))$.
	On the other hand, if $x\in\bigcap_{n\ge0}F^{-n}(\XX_{\infty}(\cu))$, then $F^n(x)=f^{\sum_{j=0}^{n-1}R_{\cu}\circ F^j(x)}(x)$ is well defined for every $n\ge1$.
	Furthermore, using (P1) inductively, we get that $\sum_{j=0}^{n-1}R_{\cu}\circ F^j(x)\in\cu(x)$ $\forall\,n\ge1$ which proves that $\#\cu(x)=\infty$, i.e., that $x\in\XX_{\infty}(\cu)$. 
	\end{proof}

\begin{Lemma}\label{LemmaCoerenteToCoerente}
	If $\cu:\XX\to 2^{\NN}$ is $f$-coherent, then $R_{\cu}$ is a coherent induced time. 
	Furthermore,  $\cu(x)=\{j\in\NN\,;\,f^j(x)\in\co_F^+(x)\}$ $\forall x\in\XX_1(\cu)$ such that $f|_{\co_f^+(x)}$ is injective, where $F(x)=f^{R_{\cu}(x)}(x)$.
	Conversely, if $R:A\to\NN$ is a coherent induced time  defined on a measurable set $A$, then the map $\cu_R:\XX\to2^{\NN}$ given by 
\begin{equation}\label{EquationtVilma}
  \cu_R(x)=
	\bigg\{\sum_{j=0}^n \overline{R}\circ F^j(x)\,;\,j\ge0\bigg\}
\end{equation} is a coherent schedule of events such that $R(x)=R_{\cu_R}(x)$ for every $x\in A$, 
where $F(x)=F^{\overline{R}(x)}(x)$ and $\overline{R}:\XX\to\NN$ is the extension of $R$ to $\XX$ (i.e., $\overline{R}|_{A}=R$) given by $\overline{R}(x)=1$ when $x\in\XX\setminus A$.
\end{Lemma}
\begin{proof}
	Suppose that $x$ and $f^j(x)\in \XX_1(\cu)$ and that $n:=R_{\cu}(x)>j\ge0$.
	By the coherence of $\cu$, $n-j\in\cu(f^j(x))$ and so, $R_{\cu}(f^j(x))=\min\cu(f^j(x))\le n-j$, proving that $R$ is a coherent induced time.
	
	Let $x\in\XX_1(\cu)$ such that $f|_{\co_f^+(x)}$. If $f^j(x)\in\co_F^+(x)$, for $j\ge1$, then
	$f^j(x)=$ $F^s(x)=$ $f^{\sum_{i=0}^{s-1}R_{\cu}\circ F^i(x)}(x)$ for some $s\ge1$.
	As $r_n=\sum_{j=0}^{n-j}R_{\cu}\circ F^j(x)$ is a strictly increasing sequence of natural numbers and $\varphi:\NN\ni n\mapsto f^n(x)\in\co_f^+(f(x))$ is a bijection, we get that $j=\varphi^{-1}(F^s(x))=r_s$. This implies, using (P2), that $j\in\cu(x)$.
	On the other hand, if $m\in\cu(x)$, set $s=\max\{j\ge1\,;\,\sum_{i=0}^j R\circ F^j(x)\le m\}$ and $r=\sum_{i=0}^s R\circ F^j(x)$.
	If $r=m$, we get that $f^r(x)=F^s(x)\in\co_F^+(x)$.
	So, suppose that $m\ne s$. In this case, setting $y=F^s(x)$, we must have $n-r<R_{\cu}(y)=\min\cu(y)$.
	Thus, it follows from P1 that $m-r\in\cu(y)$ and this implies that $m-r\ge\min\cu(y)=R_{\cu}(y)$, a contradiction, proving that $r=m$ and so, $f^r(x)\in\co_F^+(x)$.

Now assume that $R:A\to\NN$ is a coherent induced time and that $\cu_R$ is defined by (\ref{EquationtVilma}) and note that the extension $\overline{R}$ is also a coherent induced time. Furthermore,  $A_0:=\bigcap_{j\ge0}F^{-j}(\XX)=\XX$. Thus, it follows from Lemma~\ref{LemmaCanCoh1} that,
if $0\le j<\overline{R}(x)$ for some $x\in\XX$, then $\overline{R}(x)=j+\sum_{k=0}^b\overline{R}\circ {F}^k(f^j(x))$, for some $b\ge1$.
Thus, $\overline{R}(x)-j=\sum_{k=0}^b\overline{R}\circ {F}^k(f^j(x))\in\cu_{R}(f^j(x))$, proving P1 to $\cu_{R}$.
If $n\in\cu_{R}(x)$ and $m\in\cu_{R}(f^n(x))$ then $n=\sum_{j=0}^a \overline{R}\circ{F}^j(x)$ and $m=\sum_{j=0}^b\overline{R}\circ{F}^j(f^n(x))=\sum_{j=0}^b\overline{R}\circ{F}^j({F}^a(x))$.
Thus, $n+m=\sum_{j=0}^{a+b}\overline{R}\circ{F}^j(x)\in\cu_{R}(x)$, proving P2. Finally, it follows from the definition of $\cu_{R}(x)$ that $R_{\cu_{R}}(x)=\min\cu_{R}(x)=\overline{R}(x)=R(x)$ for every $x\in A$.
\end{proof}

\section{Coherent blocks and synchronization results}

In all this section, $(\XX,\mathfrak{A})$ is a measurable space and $f:\XX\to\XX$ is a measurable map.

Theorem~\ref{TheoremPlissBlock} below is in the core of this paper and it relates the natural density of a coherent schedule of events with the measure of the {\em coherent block} associated to it. The concept of coherent blocks  (Definition~\ref{CoerBlo}) was inspired by the idea of {\em Pesin sets and Hyperbolic Blocks} of Pesin theory (see, for instance, \cite{PuS}) and also by the properties of the set that survives after the ``redundancy elimination algorithm'' in the work of Castro \cite{Cas}.
The main property of a coherent block $B_{\cu}$ associated to a coherent schedule $\cu$ is that $$f^j(x)\in B_{\cu} \implies j\in\cu(x)$$ for every $j\ge1$ and $x\in\XX$.

\begin{Definition}[Coherent block]\label{CoerBlo}
If $f:\XX\to\XX$ and injective map and $\cu$ a $f$-coherent schedule on $\XX$ then define the {\bf\em  $f$-coherent block for $\cu$} or, for short, the {\bf\em $\cu$-block}, as $$B_{\cu}=\bigg\{x\in\bigcap_{n\ge0}f^{n}(\XX)\,;\,j\in\cu(f^{-j}(x))\,\forall\,j\ge1\bigg\}.$$

On the other hand, if $f:\XX\to\XX$ is a non injective, and $\cu$ a $f$-coherent schedule on $\XX$, 
it is easy to check that $\overline{\cu}:\XX_f\to\NN$, given by $\overline{\cu}(\overline{x})=\cu(\pi(\overline{x}))$, is a $\overline{f}$-coherent schedule on $\XX_f$, where $\overline{f}:\XX_f\to\XX_f$ is the natural extension of $f$ and $\pi$ is the natural  projection (see Section~\ref{SectionNatExt}). Therefore, when $f$ is non-injective, define  the {\bf\em $\cu$-block} as $B_{\cu}=\pi(B_{\overline{\cu}})$.

\end{Definition}

By Purves' result \cite{Pu}, if one assumes that $\XX$ is a complete separable metric space then, in Theorem~\ref{TheoremPlissBlock}, $f$ can be taken only as being injective and measurable.

Given any bimeasurable map $f:\XX\to\XX$ and $f$-coherent schedule $\cu$, define the {\bf\em $\cu$-absorbing set} as $$\ca_{\cu}:=\ca_F=\bigcap_{n\ge0}F^n\bigg(\bigcap_{j\ge0}F^{-j}(\XX)\bigg),$$
where $F(x)=f^{\min\cu(x)}(x)$ is the first $\cu$-time map. Using the absorbing set instead of the coherent block, we can extend the result above to non-injective maps.

\begin{Lemma}\label{LemmaAbsorbing}
If $f$ is injective and $\cu$ a $f$-coherent schedule, then $\ca_{\cu}\subset B_{\cu}$.
\end{Lemma}
\begin{proof} Let $A_0:=\bigcap_{j\ge0}F^{-j}(\XX)$.
Consider any $x\in\ca_{\cu}$. Given $j\ge1$, choose $y\in A_0$ such that $F^{j+1}(y)=x$.
By (P2), $t:=\sum_{i=0}^{j}R_{\cu}\circ F^i(y)\in\cu(y)$.
As $t>j$, it follows from (P1) that $j\in\cu(f^{t-j}(y))$.
So, because $f$ is injective and $f^{j}(f^{t-j}(y))=f^t(y)=x$, we get that $f^{-j}(x)=f^{t-j}(y)$, proving that $f^{-j}(x)$ is well defined for every given $j\ge1$ and $j\in\cu(f^{t-j}(y))=\cu(f^{-j}(x))$. That is, $x\in  B_{\cu}$.
\end{proof}

\begin{Lemma}\label{LemmaPlissBlock} 
Suppose that $f$ is injective and $\cu$ a coherent schedule of events.
If $F:\XX_{1}(\cu)\to\XX$ is the first $\cu$-time map, then the following statements are true.
\begin{enumerate}
\item \label{ItemTheoremPlissBlock67876}
$F(B_{\cu}\cap\XX_1(\cu))\subset B_{\cu}$, $F|_{B_{\cu}\cap\XX_1(\cu)}$ is injective and it is the first return map, by $f$, to $B_{\cu}$.
\item\label{ItemTheoremPlissBlock456987}  If $f$ is bimeasurable then $B_{\cu}$ is a measurable set and $B_{\cu}=\ca_{\cu}\mod\mu$, for every ergodic $f$ invariant probability $\mu$.
\end{enumerate}
\end{Lemma}
\begin{proof}
	
{\bf\em Item (\ref{ItemTheoremPlissBlock67876}).} 
Let $x\in B_{\cu}\cap\XX_1(\cu)$ and $j\ge1$. If $0<j<R_{\cu}(x)$, it follows from (P1) that $j=R_{\cu}(x)-(R_{\cu}(x)-j)\in\cu(f^{R_{\cu}(x)-j}(x))=\cu(f^{-j}(F(x)))$.
Of Course that $R_{\cu}(x)\in\cu(x)=\cu(f^{-R_{\cu}(x)}(F(x)))$.
If $j> R_{\cu}(x)$ set $y=f^{-j}(F(x))=f^{-(j-R_{\cu}(x))}(x)$. 
As $j-R_{\cu}(x)\in\cu(f^{-(j-R_{\cu}(x))}(x))$ $=$ $\cu(y)$ and $R_{\cu}(x)\in\cu(x)=\cu(f^{j-R_{\cu}(x)}(y))$, it follows from (P2) that $j=(j-R_{\cu}(x))+R_{\cu}(x)\in\cu(y)$.
Thus, we conclude that $F(B_{\cu}\cap\XX_1(\cu))\subset B_{\cu}$. 

As $f$ is injective, the injectiveness of $F|_{B_{\cu}}$ follows from the fact that $F|_{B_{\cu}\cap\XX_{1}(\cu)}$ is the first return map to $B_{\cu}$ by $f$.
So, to complete the prove of item (1), we only need to show this fact.
Therefore, suppose that $F|_{B_{\cu}\cap\XX_1(\cu)}$ is not the first return map to $B_{\cu}$.
So, there is $x\in B_{\cu}\cap\XX_1(\cu)$ and $0<j<R_{\cu}(x)$ such that $f^j(x)\in B_{\cu}$.
As $y:=f^j(x)\in B_{\cu}$, it follows from (P2) that $j\in\cu(f^{-j}(y))=\cu(x)$.
This implies that $R_{\cu}(x)=\min\cu(x)\le j$, which is a contradiction.

{\bf\em Item~(\ref{ItemTheoremPlissBlock456987}).} If $f$ is bimeasurable then $F$ is measurable map. This implies that $\bigcap_{j\ge0}F^{-j}(\XX)$ and $\bigcap_{j\ge0}f^{j}(\XX)$ are measurable.
On the other hand, note that $j\in\cu(f^{-j}(x))$ $\iff$ $x\in(\pi_j\circ i\circ \cu \circ f^{-j})^{-1}(1)$, where $\pi_j:\Sigma_2^+\to\{0,1\}$ is the projection $\pi_j(z)=z(j)$ and $i:2^n\to\Sigma_2^+$ is the bijection given by (\ref{EquationItinerario}).
As a composition of measurable maps is measurable, $\pi_j\circ i\circ \cu \circ f^{-j}$ is measurable and so, $$B_{\cu}=\bigg(\bigcap_{n\ge0}f^{n}(\XX)\bigg)\cap\bigg(\bigcap_{n\ge0}(\pi_j\circ i\circ \cu \circ f^{-j})^{-1}(1)\bigg)\text{ is 
 measurable.}$$

As $\ca_{\cu}\subset B_{\cu}$ (Lemma~\ref{LemmaAbsorbing}), we may assume that $\mu(B_{\cu})>0$, otherwise $\mu(B_{\cu})=\mu(\ca_{\cu})=0$. 
As $F|_{B_{\cu}\cap\XX_{1}(\cu)}$ is the first return map to $B_{\cu}$ by $f$ (see item~(\ref{ItemTheoremPlissBlock67876})), it follows from Lemma~\ref{FolkloreResultA} that $\nu:=\frac{1}{\mu(B_{\cu})}\mu|_{B_{\cu}}$ is an ergodic $F$ invariant probability.
Hence, it follows from Theorem~\ref{HalfLifting} that $\ca_{\cu}\subset B_{\cu}$ is a $F$ forward invariant set with $\nu(\ca_{\cu})=\mu(\ca_{\cu})/\mu(B_{\cu})>0$.
Thus, by the ergodicity and $F$ invariance of $\nu$, we get that $\nu(\ca_{\cu})=1$ and so, $\mu(B_{\cu}\triangle\ca_{\cu})=\mu(B_{\cu}\setminus\ca_{\cu})+\mu(\ca_{\cu}\setminus B_{\cu})=\mu(B_{\cu})-\mu(\ca_{\cu})=0$.

\end{proof}

\begin{T}
\label{TheoremPlissBlock} Let $(\XX,\mathfrak{A})$ be a measurable space, $f:\XX\to\XX$ a bimeasurable injective map and $\cu$ a $f$ coherent schedule of events.
If $\mu$ is an ergodic $f$-invariant probability then
$$\mu(B_{\cu})=\dd_{\NN}(\cu(x)):=\lim_{n\to\infty}\frac{1}{n}\#(\{1,2,3,\cdots,n\}\cap\cu(x))$$ for $\mu$-almost every $x\in\XX$. 
\end{T}

\begin{proof}
As $\mu(B_\cu)\le\dd_{\NN}^+(\cu(x))$, if $\mu(\XX_{\cu}^+)=0$ then $\mu(B_{\cu})=\dd_{\NN}(\cu(x))=\dd_{\NN}^+(\cu(x))=0$ for $\mu$ almost every $x$.
Therefore, we may assume that  $\mu(\XX_{\cu}^+)>0$.
It follows from  Lemma~\ref{Lemmakiuyghjko76} that $\dd_{\NN}^+(\cu(x))=\dd_{\NN}^+(\cu(f(x)))$ for every  $x\in\XX_{\cu}^+$.
Thus, as $\mu$ is ergodic and $f$ invariant, we get that $\dd_{\NN}^+(\cu(x))>0$ for $\mu$ almost every $x\in\XX$, that is, $\XX=\XX_{\cu}^+\mod\mu$.
As a consequence, it follows from Lemma~\ref{LemmaCoerenteToCoerente00} that $$\XX=\XX_{\cu}^+=\bigcap_{n\ge0}F^{-n}(\XX)\mod\mu.$$
As $f$ is injective, $f|_{\co_f^+(x)}$ is injective for every $x\in\XX$. 
Thus, it follows from Lemma~\ref{LemmaCoerenteToCoerente}, that $\cu(x)=\{j\in\NN\,;\,f^j(x)\in\co_F^+(x)\}$ for $\mu$ almost every $x\in\XX$.
So,
\begin{equation}\label{Equationjf45k}
\frac{1}{n}\#\{1\le j\le n\,;\,f^j(x)\in\co_F^+(x)\}=\frac{1}{n}\#\big(\{1,\cdots,n\}\cap\cu(x)\big)
\end{equation}
for $\mu$ almost every $x\in\XX$.
In particular,
$$  \theta_F(x)=\limsup_{n\to\infty}\frac{1}{n}\#\{0\le j<n\,;\,f^j(x)\in\co_F^+(x)\}=\dd_{\NN}^+(\cu(x))>0$$
for $\mu$ almost every $x\in\XX$.
Therefore, it follows from Lemma~\ref{LemmaCalculoBOM}, that $\liminf_{n}\frac{1}{n}\sum_{j=0}^{n-1}R_{\cu}\circ F^j(x)<+\infty$ for $\mu$ almost every $x\in\XX$.
Thus, it follows from Theorem~\ref{TheoremLift} that $\mu$ is $F$-liftable.
In particular, there exists a $F$ invariant probability $\nu_0\ll\mu$.
From Theorem~\ref{HalfLifting}, we get that $\mu(\ca_{\cu})>0$ and so, $\mu(B_{\cu})\ge\mu(\ca_{\cu})>0$.
As $F|_{\XX_1(\cu)\cap B_{\cu}}$ is the first return map to $B_{\cu}$ by $f$ (Lemma~\ref{LemmaPlissBlock}), we get that 
From Lemma~\ref{FolkloreResultA} that $\nu:=\frac{1}{\mu(B_{\cu})}\mu|_{B_{\cu}}$ is an ergodic $F$ invariant probability.

Hence, it follows from Kac's result (Theorem~\ref{KacResult}) that  $\int R_{\cu}|_{\XX_1(\cu)\cap B_{\cu}}d\mu=1$ and so, $\int R_{\cu} d\nu=\frac{1}{\mu(B_{\cu})}$.
 By Birkhoff, $\lim_n\frac{1}{n}\sum_{j=0}^{n-1}R_{\cu}\circ F^j(x)=\int R_{\cu} d\nu$ for $\mu$ almost every $x\in B_{\cu}$.
Thus, using  Lemma~\ref{LemmaCalculoBOM},  (\ref{Equationjf45k}) and the $f$ ergodicity of $\mu$, we get that 
 $$\mu(B_{\cu})=\frac{1}{\lim_n\frac{1}{n}\sum_{j=0}^{n-1}R_{\cu}\circ F^j(x)}=\lim_{n\to\infty}\frac{1}{n}\#\{0\le j<n\,;\,f^j(x)\in\co_F^+(x)\}=$$
 $$=\lim_{n\to\infty}\frac{1}{n}\#\big(\{1,\cdots,n\}\cap\cu(x)\big)=\dd_{\NN}(\cu(x))
=\dd_{\NN}^+(\cu(x))>0$$
 for $\mu$ almost every $x\in\XX$.  
\end{proof}

\begin{Remark} We may have that $F(B_{\cu}\cap\XX_{\infty}(\cu))\subsetneqq B_{\cu}\cap\XX_{\infty}(\cu)$, as one can see in the example illustrated by Figure~\ref{NaoCadeia.png}.
	
\begin{figure}
  \begin{center}\includegraphics[scale=.3]{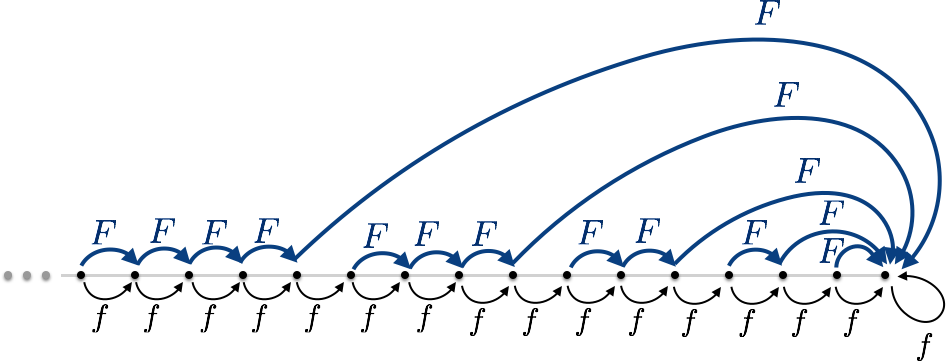}\\
 \caption{In this picture, the black dots represent the (total) orbit of a point $p$. The arrows labeled with $f$ (or $F$) indicate how a point moves under the action of $f$ (or $F$). The induced map $F$ is the first $\cu$-time for a coherent schedule of events $\cu$. The dot, say $p$, at the extreme right is a fixed point to $f$. Although $p$  belongs to $B_{\cu}\cap\XX_{\infty}(\cu)$, we have that $F^{-1}(p)\cap B_{\cu}=\emptyset$.}\label{NaoCadeia.png}
  \end{center}
\end{figure}
\end{Remark}

\begin{Corollary}
\label{CorollaryPlissBlock} 
Let $(\XX,\mathfrak{A})$ be a measurable space, $f:\XX\to\XX$ a bimeasurable map and $\cu$ a $f$ coherent schedule of events.
If $\mu$ is an ergodic $f$-invariant probability, then $$\mu(B_{\cu})\ge\dd_{\NN}(\cu(x))=\lim_{n\to\infty}\frac{1}{n}\#(\{1,2,3,\cdots,n\}\cap\cu(x))$$ for $\mu$-almost every $x\in\XX$.
Moreover, $F(B_{\cu})\subset B_{\cu}$, where $F(x)=f^{\min\cu(x)}(x)$.
\end{Corollary}
\begin{proof}
Let $\overline{\mu}$ be the ergodic $\overline{f}$-invariant probability given by Proposition~\ref{PropRokhlin}.
Note that $\overline{f}$ is an injective bimeasurable map, as required by Theorem~\ref{TheoremPlissBlock}.
Furthermore, $\mu=\pi_*\overline{\mu}=\overline{\mu}\circ\pi^{-1}$ and $f\circ \pi=\pi\circ\overline{f}$.

Define
	$\overline{\cu}:\XX_f\to2^{\NN}$ by 
	$\overline{\cu}=\cu\circ \pi$.
Note that $R_{\overline{\cu}}(\overline{x}):=R_{\cu}(\pi(\overline{x}))=R_{\cu}(\overline{x}(0))$, where $R_{\overline{\cu}}$ is the first $\overline{\cu}$-time.
Note that $F\circ\pi=\pi\circ\overline{F}$, where $\overline{F}(\overline{x})=\overline{f}^{R_{\overline{\cu}}(\overline{x})}(\overline{x})$ and so, $F(B_{\cu})=F(\pi(B_{\overline{\cu}}))=\pi(\overline{F}(B_{\overline{\cu}}))\subset\pi(B_{\overline{\cu}})=B_{\cu}$.

As $\overline{f}$ is injective, and $\overline{\cu}$ is a  $\overline{f}$ is a coherent schedule of events it follows form Theorem~\ref{TheoremPlissBlock} that $\overline{\mu}(B_{\overline{\cu}})=\dd_{\NN}(\overline{\cu}(\overline{x}))$ for $\overline{\mu}$-almost every $\overline{x}\in\XX_f$.
Hence,  
as $\dd_{\NN}(\overline{\cu}(\overline{x}))=\dd_{\NN}(\pi(\overline{\cu}(\overline{x})))$ for every $\overline{x}\in\XX_f$, $\dd_{\NN}(\cu(x))$ exists for $\mu$-almost every $x\in\XX$ and $$\mu(B_{\cu})=\mu(\pi(B_{\overline{\cu}}))=\overline{\mu}(\pi^{-1}(\pi(B_{\overline{\cu}})))\ge \overline{\mu}(B_{\overline{\cu}})=\dd_{\NN}(\cu(x)),$$
for $\mu$-almost every $x\in\XX$.
\end{proof}

\subsection{Proof of Theorem~\ref{TheoremSynINV}}
\begin{proof}
For $j=0,\cdots,m$, write $R_j(x)=R_{\cu_j}(x)$ and $F_j(x)=f^{R_{j}(x)}(x)$ for every $x\in A_j:=\XX_1(\cu_j)=\{x;\#\cu_j(x)\ge1\}$. 

If $f$ is an injective bimeasurable map, the coherent block $B_{\cu_j}$ is well defined for each $j\in\{0,\cdots,m\}$. Write $B_j=B_{\cu_j}$ $\forall\,0\le j\le m$.
It follows from Theorem~\ref{TheoremPlissBlock} that $\mu(B_j)>0$ for every $0\le j\le m$. Define $\ell_0=0$ and, 
as $\mu$ is ergodic, define $\ell_1=\min\{j\ge n_1\,;\,\mu(B_0\cap f^{-j}(B_1))>0\}$ and inductively define, for every $2\le k\le m$, $$\ell_{k}=\min\{j\ge0\,;\,\mu(B_0\cap f^{-\ell_1}(B_1)\cap\cdots\cap f^{-\ell_{k-1}}(B_{k-1})\cap f^{-j}(B_k))>0\}.$$
Taking $\BB=B_0\cap f^{-\ell_1}(B_1)\cap\cdots\cap f^{-\ell_{m}}(B_m)$, we get that $\theta:=\mu(\BB)>0$.
As $\mu$ is ergodic $\lim_{n\to\infty}\frac{1}{n}\#\{j\ge0\,;\,f^j(x)\in\BB\}=\theta$ for $\mu$ almost every $x\in\XX$.
Furthermore, if $f^j(x)\in\BB\subset B_0$ then $f^{j+\ell_k}(x)\in B_k$ for every $0\le k\le m$.
Thus,  $(j,j+\ell_1,\cdots,j+\ell_m)\in\cu_0(x)\times\cdots\times\cu_m(x)$ for $\mu$ almost every $x$ and any $j\ge1$ such $f^j(x)\in\BB$, which concludes the proof of the theorem when $f$ is an injective map.

If $f$ is bimeasurable but not injective, consider the natural extension $\overline{f}:\XX_f\to\XX_f$ of $f$ and let $\overline{\mu}$ be the ergodic $\overline{f}$ invariant probability given by Proposition~\ref{PropRokhlin}.
	
Note that $\overline{f}$ is an injective bimeasurable map, as required by Theorem~\ref{TheoremPlissBlock}.
	Furthermore, $\mu=\pi_*\overline{\mu}=\overline{\mu}\circ\pi^{-1}$ and $f\circ \pi=\pi\circ\overline{f}$.

	Let
	$\overline{\cu}_k:\overline{A}_k\to2^{\NN}$ and $\overline{R}_k:\overline{A}_k\to\NN$ be given by 
	$\overline{\cu}_k=\cu_k\circ \pi$ and
	$\overline{R}_k=R_k\circ\pi$, where $\overline{A}_k=\pi^{-1}(A_k)$.
	Let $\overline{F}_k:\overline{A}\to\overline{A}$ be given by $\overline{F}_k(\overline{x})=\overline{f}^{\overline{R}_k(\overline{x})}(\overline{x})$. Note that $F_k\circ\pi=\pi\circ\overline{F}_k$.

As $\overline{f}$ is injective, let $\overline{B}_k$ be the coherent block for $\overline{\cu}_k$. 	
Thus, it follows from the injective case that there are $\ell_1,\cdots,\ell_m\ge0$ such that  $\overline{\BB}:=\overline{B}_0\cap \overline{f}^{\,-\ell_1}(\overline{B}_1)\cap\cdots\cap \overline{f}^{\,-\ell_{m}}(\overline{B}_m)$ has $\overline{\mu}$ positive measure and $(j,j+\ell_1,\cdots,j+\ell_m)\in\overline{\cu}_0(\overline{x})\times\cdots\times\overline{\cu}_m(\overline{x})$ whenever $\overline{f}^j(\overline{x})\in\overline{\BB}$. 

As $\mu=\overline{\mu}\circ\pi^{-1}$, for $\mu$ almost every $x\in\XX$ there is a $\overline{x}\in\pi^{-1}(x)$ such that $\lim_n\frac{1}{n}\#\{0\le j<n\,;\,\overline{f}^j(\overline{x})\in\overline{\BB}\}=\overline{\mu}(\overline{\BB})$. Moreover, as $\overline{f}^n(\overline{x})\in\overline{\BB}$ $\implies$ $(j,j+\ell_1,\cdots,j+\ell_m)\in\overline{\cu}_0(\overline{x})\times\cdots\times\overline{\cu}_m(\overline{x})$ $\iff$ $(j,j+\ell_1,\cdots,j+\ell_m)\in\cu_0(x)\times\cdots\times\cu_m(x)$, we get that 
$$\lim_n\frac{1}{n}\#\big\{0\le j<n\,;\,(j,j+\ell_1,\cdots,j+\ell_{m})\in\cu_0(x)\times\cdots\times\cu_m(x)\big\}=$$
$$=\lim_n\frac{1}{n}\#\big\{0\le j<n\,;\,(j,j+\ell_1,\cdots,j+\ell_{m})\in\overline{\cu}_0(\overline{x})\times\cdots\times\overline{\cu}_m(\overline{x})\big\}=\overline{\mu}(\overline{\BB})>0$$
for $\mu$ almost every $x\in\XX$. Thus, taking $\theta=\overline{\mu}(\BB)$, we conclude the proof.
\end{proof}

\section{Pointwise synchronization}\label{SectionPointwiseSynchronization}

In this section, unless otherwise noted, $(\XX,\mathfrak{A})$ is a measurable space and $f:\XX\to\XX$ a measurable map.

\begin{Lemma}\label{Lemmahdhhg620}
Let $\varphi:\NN\times\XX\to[-\infty,+\infty]$ be a $f$ sup-additive cocycle. 
If $\mu$ is a $f$ invariant probability, $\int_{x\in\XX}|\varphi(1,x)|d\mu<+\infty$ and $\liminf_n\frac{1}{n}\varphi(n,x)\ge\lambda\in(0,+\infty)$ for $\mu$ almost every $x\in\XX$, then there is a $\ell_0\ge1$ such that
$$\int_{x\in\XX}\varphi(\ell,x)d\mu\ge \frac{\lambda}{4}\ell$$
for all $\ell\ge\ell_0$. 
\end{Lemma}
\begin{proof} The proof of this lemma was based on the proof of Lemma~3.5 in \cite{KO}.
Let $\XX'=\{x\in\XX\,;\,\liminf_n\frac{1}{n}\varphi(n,x)\ge\lambda>0\}$.

Define, for $x\in\XX'$, $n_x=\min\{n\ge1\,;\,\frac{1}{j}\varphi(j,x)>\lambda/2\;\forall j\ge n\}$ and let $C_j=\{x\in\XX'\,;\,n_x\le j\}$, for $j\ge1$. As $\varphi$ is sup-additive, $\varphi(j,x)\ge\varphi(1,x)+\varphi(j-1,f(x))\ge\varphi(1,x)+\varphi(1,f(x))+\varphi(j-2,f^2(x))$ $=$ $\cdots$ $=$ $\sum_{i=0}^{j-1}\varphi(1,f^i(x))$. Thus,
$$\int_{x\in\XX}\frac{1}{j}\varphi(j,x)d\mu=\int_{x\in C_j}\frac{1}{j}\varphi(j,x)d\mu+\int_{x\in\XX\setminus C_j}\frac{1}{j}\varphi(j,x)d\mu\ge$$
$$\ge\frac{\lambda}{2}\mu(C_j)+ \int_{x\in\XX\setminus C_j}\frac{1}{j}\varphi(j,x)d\mu\ge\frac{\lambda}{2}\mu(C_j)+\int_{x\in\XX\setminus C_j}\frac{1}{j}\sum_{i=0}^{j-1}\varphi(1,f^i(x))d\mu.$$
Let $\psi_n(x):=\frac{1}{n}\sum_{i=0}^{n-1}\varphi(1,f^i(x))$. As $\int_{x\in\XX}|\varphi(1,x)|d\mu<+\infty$, it follows from Birkhoff that $\psi\in\cl^1(\mu)$, where $\psi(x):=\lim_n\psi_n(x)$ for $\mu$ almost every $x\in\XX$.
By Birkhoff (indeed, by a corollary of it, see Corollary 1.14.1 in \cite{Wa}), $\lim_n\int_{\XX}|\psi_n-\psi|d\mu=0$.
So, 
given any $0<\varepsilon<\lambda/12$ there is a $m_0\ge1$ such that $\int_{\XX}|\psi_{\ell}-\psi|d\mu<\varepsilon/2$ for every $\ell\ge m_0$.
As $\int|\psi|d\mu<+\infty$ and $\lim_j\mu(\XX\setminus C_j)=0$, there is $\ell_0\ge m_0$ such that $\mu(C_{\ell})>2/3$ and $\int_{\XX\setminus C_{\ell}}|\psi| d\mu<\varepsilon/2$ for every $\ell\ge\ell_0$.
Thus, $$\int_{\XX\setminus C_{\ell}}|\psi_{\ell}|d\mu\le\int_{\XX\setminus C_{\ell}}|\psi_{\ell}-\psi|d\mu+\int_{\XX\setminus C_{\ell}}|\psi| d\mu\le\int_{\XX}|\psi_{\ell}-\psi|d\mu+\int_{\XX\setminus C_{\ell}}|\psi| d\mu<\varepsilon<\lambda/12,$$
for every $\ell\ge\ell_0$.
Therefore, 
$$\int_{x\in\XX}\frac{1}{\ell}\varphi(\ell,x)d\mu>\frac{\lambda}{2}\frac{2}{3}-\frac{\lambda}{12}=\lambda/4,$$
for every $\ell\ge\ell_0$.
\end{proof}

\begin{Lemma}\label{LemmaDynInt17}
If $\varphi:\XX\to\RR$ is a measurable function and $\mu$ is an ergodic $f$ invariant probability then
$$
  \int|\varphi|d\mu=\int_0^{\infty}\dd_{\NN}^+\big(\{j\ge0\,;\,|\varphi\circ f^j(x)|\ge r\}\big)dr
$$
 for $\mu$ almost every $x$.	
\end{Lemma}
\begin{proof}
Consider $\psi:[0,+\infty)\to[0,1]$ given by $\psi(r)=\mu(\{|\varphi|\ge r\})$ and let $W=\{r\ge0\,;\,\mu(\{\varphi=r\})>0\}$. 
We claim that $\varphi$ is continuous on $[0,+\infty)\setminus W$. Indeed, given $r\in [0,+\infty)\setminus W$, let  $U_{\varepsilon}=\{|\varphi|< r+\varepsilon\}$.
Given $0<a_n\searrow0$, we get that $U_{a_1}\supset U_{a_2}\supset U_{a_n}\supset\cdots$ and that $\bigcap_{n}U_{a_n}=\{|\varphi|\le r\}$.
Thus, by Lemma~\ref{LemmadeMEDIDA} of Appendix, $\lim_n(1-\psi(r+a_n))=\lim_n\mu(U_{a_n})=\mu(\{|\varphi|\le r\})=1-\mu(\{|\varphi|>r\})=1-\psi(r)$ (note that, as $r\notin W$, $\mu(\{|\varphi|>r\})=\mu(\{|\varphi|\ge r\})=\psi(r)$).
That is, $\lim_{\varepsilon\downarrow0}\varphi(r+\varepsilon)=\varphi(r)$, the right-hand limit is equal to $\varphi(r)$.
Similarly, taking $V_{\varepsilon}=\{|\varphi|\ge r-\varepsilon\}$ and a sequence $0<a_n\searrow0$,
we get that $V_{a_1}\supset V_{a_2}\supset V_{a_3}\supset\cdots$ and $\bigcap_n V_{a_n}=\{|\varphi|\ge r\}$.
So, by Lemma~\ref{LemmadeMEDIDA} of Appendix, $\lim_n\psi(r-a_n)=\lim_n\mu(V_{a_n})=\mu(\{|\varphi|\ge r\})=\psi(r)$,
proving that $\lim_{\varepsilon\uparrow0}\varphi(r+\varepsilon)=\varphi(r)=\lim_{\varepsilon\downarrow0}\varphi(r+\varepsilon)$ which implies the continuity of $\psi$ for $r\notin W$.

As $W$ must be countable, we get that its Riemann integral $\int_{r=0}^{r=a}\psi(r)dr=\int_{r=0}^{r=a}\mu(\{|\varphi|\ge r\})dr$ is well defined.
Hence $$\int_{r=0}^{r=\infty}\mu(\{x\in\XX\,;|\varphi(x)|\ge r\})dr=\int_{r\in[0,+\infty)}\mu(\{x\in\XX\,;\,|\varphi(x)|\ge r\})d\leb=$$
$$=(\mu\times\leb)(\{(x,r)\in\XX\times[0,+\infty)\,;\,\varphi(x)\ge r\})=$$
$$=\int_{x\in\XX}\underbrace{\leb(\{r\in[0,+\infty)\,;\,|\varphi(x)|\ge r\})}_{|\varphi(x)|}d\mu=\int|\varphi|d\mu.$$
That is, $\int_0^{\infty}\psi(r)dr=\int|\varphi|d\mu$ and so,
using Birkhoff, we get that $$\dd_{\NN}^+(\{j\ge0\,;\,|\varphi\circ f^j(x)|\ge r\})=\lim\frac{1}{n}\#\{0\le j<n\,;\,f^j(x)\in\{|\varphi|\ge r\}\}=$$
$$=\mu(\{|\varphi|\ge r\})=\psi(r),$$
showing that $\int _0^{\infty}\dd_{\NN}^+(\{j\ge0\,;\,|\varphi\circ f^j(x)|\ge r\})=\int|\varphi|d\mu$.
\end{proof}

Motivated by Lemma~\ref{LemmaDynInt17} above, we consider the following definition.
\begin{Definition}
	The {\bf\em $f$-tail sum} of a function $R:\XX\to\NN\cup\{+\infty\}$ at a point $x\in\XX$ is defined as $$I_f(R)(x):=\sum_{n=1}^{+\infty} \underbrace{\dd_{\NN}^+\big(\{j\ge0\,;\,R\circ f^j(x)\ge n\}\big)}_{a_n}.$$
	If $I_f(R)(x)<+\infty$, then we say that {\bf\em $R$ has a summable $f$-tail at $x$}. As $a_n\in[0,1]$ is a decreasing and bounded sequence, $\lim_{n}a_n$ always exists. Thus, define the {\bf\em $R$-residue} at the orbit of $x$ as
	$$\Res_f(R,x):=\lim_{n\to\infty}\dd_{\NN}^+\big(\{j\ge0\,;\,R\circ f^j(x)\ge n\}\big).$$
\end{Definition}

In order to study the synchronization of coherent schedule with respect to a non invariant probability, whenever $\XX$ is a metric space, define the {\bf\em weak-omega set} of $x$, $\mho_f(x)$, as the set of all accumulating points (on the weak* topology) of $\frac{1}{n}\sum_{j=0}^{n-1}\delta_{f^j(x)}$.

\begin{Lemma}\label{Lemma87gnd2cw}
If $\XX$ is a metric space, $f:\XX\to\XX$ is measurable and 
$R:\XX\to\NN\cup\{+\infty\}$ is lower semicontinuous and $x\in\XX$, then $$\mu(\{R\ge n\})\le\dd_{\NN}^+(\{j\ge0\,;\,R\circ f^j(x)\ge n\})$$ for every $\mu\in\mho_f(x)$ and $n\ge1$.
In particular, $\int_{\XX} R d\mu\le I_f(R)(x)$ for every $\mu\in\mho_f(x)$.
Furthermore,
 if $\Res_f(R,x)=0$ then $\mu(\{R=+\infty\})=0$ $ \forall\mu\in\mho_f(x)$.
\end{Lemma}
\begin{proof}
Given $\mu\in\mho_f(x)$, let $n_i\to\infty$ be such that $\mu=\lim_i\mu_i$, where $\mu_i=\frac{1}{n_i}\sum_{j=0}^{n_i-1}\delta_{f^j(x)}$.
As $R$ is lower semicontinuous, $\{R\ge n\}=\{R>n-1\}$ is an open subset of $\XX$.
Applying Lemma~\ref{LemmaMedTeoOLD} of  Appendix, we get that $$\mu(\{R\ge n\})\le\liminf_i\mu_i(\{R>n\})=$$
$$=\liminf_i\frac{1}{n_i}\#\{0\le j<n_i\,;\,f^j(x)\in\{R\ge n\}\}\le\dd_{\NN}^+(\{j\ge0\,;\,R\circ f^j(x)\ge n\}).$$
Thus, as $\int_{\XX}R d\mu=\sum_{n=1}^{\infty}\mu(\{R\ge n\})$,
we get that $\int_{\XX} R d\mu\le I_f(R)(x)$.

Suppose that $\mu\in\mho_f(x)$ and $\Res_f(R,x)=0$.
By Lemma~\ref{LemmadeMEDIDA} of Appendix, $$\mu(\{R=\infty\})=\mu\bigg(\bigcap_{n\ge1}\{R\ge n\}\bigg)=\lim_n\mu(\{R\ge n\})\le\lim_n\dd_{\NN}^+(\{j\ge0\,;\,R\circ f^j(x)\ge n\})=0,$$
concluding the proof.

\end{proof}

\begin{Lemma}\label{LemmaDSAJNSTFS2I}
	Let $\mu$ be a $f$-nonsingular probability (not necessarily an invariant one), $\gamma\in\RR$, $\varphi:\NN\times\XX\to[-\infty,+\infty]$ sup-additive cocycle, $\cu(x)=\{n\in\NN\,;\,\varphi(n,x)\ge n\gamma\}$ and
	$$R(x)=
	\begin{cases}
	\min\cu(x) & \text{if }\,\cu(x)\ne\emptyset\\
	+\infty & \text{if }\,\cu(x)=\emptyset 
	\end{cases}.$$
If $\mu(\{R=+\infty\})=0$, then $\limsup_n\frac{1}{n}\varphi(n,x)\ge\gamma$ for $\mu$ almost every $x\in\XX$. 
\end{Lemma}
\begin{proof}
Suppose that $\mu(\{R=+\infty\})=0$. 
As $\mu$ is $f$-nonsingular, it follows from Lemma~\ref{LemmaSingSing} that $\mu$ is also $F$-nonsingular, where $F(x):=f^{R(x)}(x)$ for every $x\in U$. Thus, $\bigcup_{j\ge0}F^{-j}(\{R=+\infty\})=0$.
Hence, $\mu(U)=1$, where $U=\bigcap_{j\ge0}F^{-j}(\{R<+\infty\})$.
Given $n\ge1$, $j\ge0$ and $x\in U$, let $a_j=R(F^j(x))$ and $s(n)=a_0+\cdots+a_{n-1}=\sum_{j=0}^{n-1}R(F^j(x))$. As $\varphi$ is sup-additive,
$$\varphi(s(n),x)\ge\varphi(a_0,x)+\varphi(a_1,f^{a_0}(x))+\cdots+\varphi(a_{n-1},f^{a_0,+\cdots+a_{n-1}}(x))=$$
$$=\varphi(R(x),x)+\varphi(R(F(x)),F(x))+\cdots+\varphi(R(F^{n-1}(x)),F^{n-1}(x))=$$
$$=\sum_{j=0}^{n-1}\varphi(R(F^j(x)),F^j(x)).$$
By the definition of $R$, $\varphi(R(y),y)\ge\gamma R(y)$ for every $y\in\{R<+\infty\}$.
Thus, $$\varphi(s(n),x)\ge\gamma\sum_{j=0}^{n-1}R(F^j(x))=\gamma s(n)$$
for every $n\ge1$. Hence $\limsup_n\frac{1}{n}\varphi(n,x)\ge\gamma$ for every $x\in U$.
\end{proof}


\begin{T}[Pointwise synchronization]\label{TheoremPointWiseSyn}
Let $\XX$ be a compact metric space,  $f:\XX\to\XX$ a continuous map, $\lambda>0$ and $\varphi:\NN\times\XX\to\RR$ a continuous sup-additive $f$-cocycle.
Given $x\in\XX$, let $\cu(x)=\{n\in\NN\,;\,\varphi(n,x)\ge n\lambda\}$ and
	$$R(x)=
	\begin{cases}
	\min\cu(x) & \text{if }\,\cu(x)\ne\emptyset\\
	+\infty & \text{if }\,\cu(x)=\emptyset 
	\end{cases}.$$
	If $\Res_f(R,p)=0$ for some $p\in\XX$, then there is a $\ell_0\ge1$ such that
	$$\liminf_{n\to\infty}\frac{1}{n}\sum_{j=0}^{n-1}\varphi(2^\ell,f^{j}(p))\ge\frac{\lambda}{5}2^\ell\;\;\;\forall\,\ell\ge\ell_0.$$
\end{T}
\begin{proof}
As $f$ is continuous and $\XX$ compact, $\cm_f^1(\XX)\supset \mho_f(p)\ne\emptyset$, where $\cm_f^1(\XX)$ is the set of all $f$ invariant Borel probabilities.
Given any $\mu\in\mho_f(p)$, consider a sequence $n_i\nearrow\infty$ such that $\mu:=\lim_i\mu_i\in\mho_f(p)$, where $\mu_i:=\frac{1}{n_i}\sum_{j=0}^{n_i-1}\delta_{f^j(p)}$.

It follows from the continuity of $f$ and $\varphi$ that $\cu$ is partially continuous and so, by Lemma~\ref{LemmaPArtuhgdj}, $R$ is lower semicontinuous.
Using Lemma~\ref{Lemma87gnd2cw}, we get that $\mu(\{R=+\infty\})=0$.
Thus, by Lemma~\ref{LemmaDSAJNSTFS2I}, $\limsup_n\frac{1}{n}\varphi(n,x)\ge\lambda$ for $\mu$ almost every $x\in\XX$.

As $\psi:=-\varphi$ is a subadditive $f$-cocycle and $\int_{x\in\XX}|\psi(1,x)|d\mu=\int_{x\in\XX}|\varphi(1,x)|d\mu<+\infty$, it follows from Kingman's Subadditive Ergodic Theorem that $\lim_n\frac{1}{n}\psi(n,x)$ exists for $\mu$ almost every $x$.
Thus, $\lim_n\frac{1}{n}\varphi(n,x)=-\lim_n\frac{1}{n}\psi(n,x)$ also exists for $\mu$ almost every $x\in\XX$.
As $$-\lim_n\frac{1}{n}\varphi(n,x)=\lim_n\frac{1}{n}\psi(n,x)=\liminf_n\frac{1}{n}\psi(n,x)=-\limsup_n\frac{1}{n}\varphi(n,x)\le-\lambda,$$
we get that
\begin{equation}\label{Equationhbiwqe2}
  \lim_n\frac{1}{n}\varphi(n,x)\ge\lambda\;\text{ 
for }\mu\text{ almost every }x\in\XX.
\end{equation}

As  $\int_{x\in\XX}|\varphi(1,x)|d\mu$ $<$ $+\infty$, it follows from Lemma~\ref{Lemmahdhhg620} that there is a $\ell_0\ge1$ such that
\begin{equation}\label{Equationnnnshgf3}
\int_{x\in\XX}\varphi(\ell,x)d\mu\ge \frac{\lambda}{4}\ell\text{ 
for all }\ell\ge\ell_0.
\end{equation}

Note that, as $\mu$ is $f$-invariant 
and $\varphi(2^{n+1},x)\ge\varphi(2^{n},f^{2^{n}}(x))+\varphi(2^{n},x)$ $\forall\,n\ge0$, 
 $$\int_{x\in\XX}\varphi(2^{n+1},x)d\mu\ge\int_{x\in\XX}\varphi(2^{n},f^{2^{n}}(x))d\mu+\int_{x\in\XX}\varphi(2^{n},x)d\mu=$$
 $$=2\int_{x\in\XX}\varphi(2^{n},x)d\mu$$ for every $n\ge0$.
Thus,
\begin{equation}\label{Equation76tyhFFFyfde32}
  \int_{x\in\XX}\varphi(2^{n},x)d\mu\ge\frac{\lambda}{5}2^n\implies\int_{x\in\XX}\varphi(2^{\ell},x)d\mu\ge\frac{\lambda}{5}2^\ell,
\end{equation}
for every $\ell\ge n$.
As a consequence, it follows from (\ref{Equationnnnshgf3}) and (\ref{Equation76tyhFFFyfde32}) that, for each $\mu\in\mho_f(p)$, there exists a unique $n_0(\mu)\in\{0,1,2,3,\cdots\}$ such that $$\int_{x\in\XX}\varphi(2^{n},x)d\mu\;
\begin{cases}
<\frac{\lambda}{5}2^n & \text{ if }n<n_0(\mu)\\
\ge\frac{\lambda}{5}2^n	& \text{ if }n\ge n_0(\mu)
\end{cases}
$$

Suppose that $N\subset\NN$ with $\#N=+\infty$ such that $\liminf\frac{1}{n}\sum_{j=0}^{n-1}\varphi(2^a,f^j(p))<\lambda2^a/5$ for every $a\in N$.
In this case, for each $a\in N$, there is a sequence $m_i=m_i(a)\to\infty$ such that $\lim\frac{1}{m_i}\sum_{j=0}^{m_i-1}\varphi(2^a,f^j(p))<\lambda2^a/5$.
Taking a subsequence if necessary, we may assume that $\lim_i\frac{1}{m_i}\sum_{j=0}^{m_i-1}\delta_{f^j(p)}=\eta_{a}$, for some $\eta_{a}\in\mho_{f}(p)$.
As $\XX\ni x\mapsto\varphi(2^a,x)$ is continuous, we get that  $$\int_{x\in\XX}\varphi(2^a,x) d\eta_a=\lim_i\int_{x\in\XX}\varphi(2^a,x) d\bigg(\frac{1}{m_i}\sum_{j=0}^{m_i-1}\delta_{f^j(p)}\bigg)=$$
$$=\lim\frac{1}{m_i}\sum_{j=0}^{m_i-1}\varphi(2^a,f^j(p))<\lambda2^a/5.$$
Therefore,
\begin{equation}\label{Equationhgfhg2125678}
	n_0(\eta_a)>a\text{ for every }a\ge1
\end{equation}
As $\mho_f(p)$ is compact, there is a sequence $N\ni a_i\to\infty$ and $\eta\in\mho_f(p)$ such that $\eta=\lim_i\eta_{a_i}$.
Using (\ref{Equationnnnshgf3}), we can choose $\ell\ge1$ such that $ \int_{x\in\XX}\varphi(2^\ell,x)d\eta\ge\frac{\lambda}{4}2^\ell$. 
As $\XX\ni x\mapsto\varphi(2^\ell,x)$ is continuous, $\lim_i\int_{x\in\XX}\varphi(2^\ell,x) d\eta_{a_i}=\int_{x\in\XX}\varphi(2^\ell,x) d\eta\ge\lambda2^{\ell}/4$.
Thus, $\int_{x\in\XX}\varphi(2^\ell,x) d\eta_{a_i}>\frac{\lambda}{5}2^\ell$ for any big $i$ and so $n_0(\eta_{a_i})\le\ell$ for every $i\ge1$ big enough, contradicting (\ref{Equationhgfhg2125678}). 
\end{proof}

\subsection{Proofs of Theorems~\ref{MainTHEOVianaConjDiffeo} and \ref{MainTheoremPartial}}

The statement and notations of Theorems~\ref{MainTHEOVianaConjDiffeo} and \ref{MainTheoremPartial} can be found in Section~\ref{IntroductionAndStatementsOfMainResults}.

\begin{proof}[\bf Proof of Theorem~\ref{MainTHEOVianaConjDiffeo}]
	The proof of Theorem~\ref{MainTHEOVianaConjDiffeo} is a direct consequence of Theorem~\ref{Theoremkjgjbkljhl} below and the fact already observed in Section~\ref{IntroductionAndStatementsOfMainResults} that, under the hypothesis of Lebesgue almost every point having only positive Lyapunov exponents, the residue to be zero on a set of positive Lebesgue measure is a necessary condition for the existence of an absolutely continuous invariant measure.
\end{proof}

\begin{T}\label{Theoremkjgjbkljhl}
	Let $f:M\to M$ be a $C^{1+}$ local diffeomorphism. 
	If $H\subset M$ is a measurable set such that every $x\in H$ has only positive Lyapunov exponents and zero Lyapunov residue, that is, 
	\begin{equation}\label{Equationtauuuuu}
  \Res(x)=0<\limsup_{n\to+\infty}\frac{1}{n}\log\|(Df^n(x))^{-1}\|^{-1}\;\;\forall\,x\in H,
\end{equation}
then Lebesgue almost every $x\in H$ belongs to the basin of some ergodic absolute continuous invariant measure. In particular, $f$ admits a SRB measure.
\end{T}

\begin{proof} As (\ref{Equationtauuuuu}) is an asymptotic condition, taking $\bigcup_{n\ge0} f^n(H)$ instead of $H$ if necessary, we may assume that $f(H)\subset H$.
Let $$H_n=\{x\in H\,;\,\limsup_{j\to+\infty}\frac{1}{j}\log\|(Df^j(x))^{-1}\|^{-1}\ge2/n\}.$$
	Set $\varphi:\NN\times M\to[-\infty,+\infty)$ as $\varphi(j,x)=\log\|(Df^j(x))^{-1}\|^{-1}$.
	Given $n\ge1$, let $\cu_n(x)=\{j\in\NN\,;\,\varphi(j,x)\ge\frac{1}{j\,n}\}\subset2^{\NN}$ and
	$$R_n(x)=
	\begin{cases}
		\min\,\cu_n(x) & \text{ if }\cu_n(x)\ne\emptyset\\
		+\infty  & \text{ if }\cu_n(x)=\emptyset
	\end{cases}
	$$
As $R_n(x)\le \Res(x)$ for every $x\in H_n$, we get that $\{j\ge0\,;\,R_n\circ f^j(x)\ge k\}$ $\subset$ $\{j\ge0\,;\,\Res\circ f^j(x)\ge k\},$ proving that  $$0\le \lim_{k\to+\infty}\dd_{\NN}^+(\{j\ge0\,;\, R_n\circ f^j(x)\ge k\})\le \Res(x)=0$$ for every $x\in H_n$.

It follows from Theorem~\ref{TheoremPointWiseSyn} that, for each $n\ge1$ and $x\in H_n$ there is $\ell_0(n,x)\ge1$ such that
$$
  	\liminf_{i\to\infty}\frac{1}{i}\sum_{j=0}^{i-1}\varphi(2^\ell,f^{j}(x))\ge\frac{1/n}{5}2^\ell\;\;\;\forall\,\ell\ge\ell_0(n,x).
$$
Given $n,\ell\ge1$, let $H_n(\ell)=\{x\in H_n\,;\,\ell_0(n,x)=\ell\}$.
Thus, for each $x\in H_n(\ell)$ there is at least one $0\le \alpha(x)<2^{\ell}$ such that
\begin{equation}\label{Equationiugjijhg6}
  	\liminf_{i\to\infty}\frac{1}{i}\sum_{j=0}^{i-1}\varphi(2^\ell,f^{j2^\ell}(f^{\alpha(x)}(x)))\ge\frac{1/n}{5}.
\end{equation}

Taking $m:=2^{\ell}$ and $H_n(\ell,j)=\{f^j(x)\,;\,x\in H_n(\ell)$ and $\alpha(x)=j\}$, we get that $H_n(\ell)\subset\bigcup_{j=0}^{m-1}f^{-j}(H_n(\ell,j))$ with $f^m(H_n(\ell,j))\subset H_n(\ell,j)$ and  
$$\liminf_{i\to\infty}\frac{1}{i}\sum_{j=0}^{i-1}\log\|(Df^m\circ f^{mj}(x))^{-1}\|^{-1}\ge\frac{1}{5n}$$ for every $x\in H_n(\ell,j)$. 
Thus, it follows from Theorem~A of \cite{Pi11}  that Lebesgue almost every $x\in H_n(\ell,j)$ belongs to the basin of some ergodic absolute continuous $f^m$-invariant measure (if $f$ is $C^2$ one can use Theorem~C of \cite{ABV} instead of \cite{Pi11}).
As $\widetilde{\mu}:=\frac{1}{m}\sum_{i=0}^{m-1}\mu\circ f^{-i}$ is an ergodic absolute continuous $f$-invariant measure whenever $\mu$ is an ergodic absolute continuous $f^m$-invariant measure, we get that Lebesgue almost every $x\in H_n(\ell,j)$ belongs to the basin of some ergodic absolute continuous $f$-invariant measure.
Finally, as $H\subset\bigcup_{n\ge1}\bigcup_{\ell\ge1}\bigcup_{j=0}^{m-1}f^{-j}(H_n(\ell,j))$, we conclude the proof.
\end{proof}

\begin{proof}[\bf Proof of Theorem~\ref{MainTheoremPartial}]

As $T_{U}M=\EE^{cu}\oplus\EE^{cs}$ is a dominated splitting, it is continuous and extends uniquely and continuously to a splitting of  $T_{\overline{U}}M$ to $\overline{U}$ (see for instance Lemma~14 of \cite{COP}). 
Thus, we get that $\varphi^u$ and $\varphi^s:\NN\times \overline{U}\to\RR$ given respectively by $\varphi^u(n,x)=\log\|(Df^n|_{\EE^{cu}(x)})^{-1}\|^{-1}$ and $\varphi^s(n,x)=-\log\|Df^n|_{\EE^{cs}(x)}\|$ are continuous sup-additive $f|_{\overline{U}}$\,-\,cocycles. 
By compactness of $\overline{U}$, we get that $\limsup_n\frac{1}{n}\sum_{j=0}^{n-1}|\varphi^{u,s}(1,f^j(x))|$ $\le$ $\sup_{p\in\overline{U}}|\varphi^{u,s}(1,p)|<+\infty$ for every $x\in\overline{U}$.

Define $H_n$ as the set of all $x\in H$ such that 
$$\liminf_{j\to+\infty}\frac{1}{j}\varphi^{u}(j,x)\;\text{ and }\;\liminf_{j\to+\infty}\frac{1}{j}\varphi^{s}(j,x)\ge2/n.$$

Setting $\Phi(i,x)=\min\{\varphi^u(i,x),\varphi^s(i,x)\}$, we get that $\Phi:\NN\times\overline{U}\to\RR$ is also a continuous sup-additive $f|_{\overline{U}}$\,-\,cocycle. Let $\cu_n(x)=\{n\in\NN\,;\,\Phi(i,x)\ge i/n\}$ and 
$$R_n(x)=
	\begin{cases}
		\min\,\cu_n(x) & \text{ if }\cu_n(x)\ne\emptyset\\
		+\infty  & \text{ if }\cu_n(x)=\emptyset
	\end{cases},$$ for any  $x\in H_n$.
As $R_n(x)\le\max\{\Res^{u}(x),\Res^s(x)\}$ and $\Res^{u}(x)=\Res^s(x)=0$ for every $x\in H_n$, we get that $$\lim_{k\to+\infty}\dd_{\NN}^+(\{j\ge0\,;\, R_n\circ f^j(x)\ge k\})=0$$ for every $x\in H_n$.

It follows from Theorem~\ref{TheoremPointWiseSyn} that, for each $n\ge1$ and $x\in H_n$, there is $\ell_0(n,x)\ge1$ such that
$$
  	\liminf_{i\to\infty}\frac{1}{i}\sum_{j=0}^{i-1}\Phi(2^\ell,f^{j}(x))\ge\frac{1/n}{5}2^\ell\;\;\;\forall\,\ell\ge\ell_0(n,x).
$$
Given $n,\ell\ge1$, let $H_n(\ell)=\{x\in H_n\,;\,\ell_0(n,x)=\ell\}$.
Thus, for each $x\in H_n(\ell)$ there is at least one $0\le \alpha(x)<2^{\ell}$ such that
$$
  	\liminf_{i\to\infty}\frac{1}{i}\sum_{j=0}^{i-1}\Phi(2^\ell,f^{j2^\ell}(f^{\alpha(x)}(x)))\ge\frac{1}{5n}.
$$

Letting $H_n(\ell,j)=\{f^j(x)\,;\,x\in H_n(\ell)$ and $\alpha(x)=j\}$ and writing $m:=2^{\ell}$, we get that $H_n(\ell)=\bigcup_{j=0}^{m-1}f^{-j}(H_n(\ell,j))$ with $f^m(H_n(\ell,j))\subset H_n(\ell,j)$ and  
\begin{equation}\label{Equationiuyghj33hg6}
  	\liminf_{i\to\infty}\frac{1}{i}\sum_{j=0}^{i-1}\Phi(m,f^{jm}(x))\ge\frac{1}{5n}.
\end{equation} for every $x\in H_n(\ell,j)$.
It follows from Pliss Lemma that the additive $f^m$ cocycle $\varphi(i,x):=\sum_{j=0}^{i-1}\Phi(m,f^{jm}(x))$ has positive lower natural density of Pliss times.
That is, taking $\cv(x)$ as the set of all $n\in\NN$ such that $n$ is a $(1/(10n),\varphi)$-Pliss time to $x$ (with respect to $f^m$), then $\dd_{\NN}^-(\cv(x))\ge1/(10 n\sup\varphi-1)>0$ (see Lemma~\ref{Lemma67ygtghXXXX}).
As $\Phi(i,x)=\min\{\varphi^u(i,x),\varphi^s(i,x)\}$, we have that every $n\in\cv(x)$ is a simultaneous hyperbolic time as asked in Proposition~6.4 of \cite{ABV}, we get that Lebesgue almost every $x\in H_n(\ell,j)$ belongs to the basin of some SRB measure for $f^m$ with support contained in $\bigcap_{j=0}^{+\infty}f^{mj}\big(\overline{U}\big)\subset\bigcap_{j=0}^{+\infty}f^{j}\big(\overline{U}\big)$.
Furthermore, as $H=\bigcup_{n\ge1}\bigcup_{\ell\ge1}\bigcup_{j=0}^{m-1}f^{-j}(H_n(\ell,j))$ and $\widetilde{\mu}:=\frac{1}{m}\sum_{i=0}^{m-1}\mu\circ f^{-i}$ is a SRB measure measure for $f$ whenever $\mu$ is a SRB measure for $f^m$, we can conclude that $\leb$ almost every $x\in H$ belongs to the basin of some SRB measure for $f$ with supported on  $\bigcap_{j=0}^{+\infty}f^j\big(\overline{U}\big)$, which finish the proof.
\end{proof}

\appendix

\section{Synchronizing Lyapunov Exponents to produce SRB}
\label{SecHYpTIm}

In this section we give an example of a class of dynamics having only non zero Lyapunov exponents for Lebesgue almost every point.
In this class,  we can use  Theorem~\ref{TheoremSynINV} for synchronizing Lyapunov exponents to produce a non uniformly expanding/hyperbolic dynamics (in the sense of \cite{ABV}) and so, to prove the existence of absolutely continuous invariant probability. We want to emphasize that, in this class, we don't know  a priori if (\ref{EquationSynNUEDEForte}) is satisfied by the map $f$ (or some iterate of it).
Hence, we can't use  \cite{ABV} technology (directly) to produce an absolutely continuous invariant probability.

\begin{Definition}[Geometric hyperbolic times]\label{DefinitionGeomHyTimes}
We say that $n\in\NN$ is a $(\lambda,\delta)$-geometric hyperbolic time for a point $x\in M$ with respect to a map $f:M\to M$ when there 
is an open neighborhood $V_n(x)$ of $x$ such that
\begin{enumerate}
	\item $f^n$ sends diffeomorphically $\overline{V_n(x)}$ to $\overline{B_{\delta}(f^n(x))}$; 
	\item $\|(Df^{n-j}(f^j(y)))^{-1}\|^{-1}>\lambda^{n-j}$ for every $y\in\overline{V_n(x)}$ and $0\le j<n$.
\end{enumerate}
\end{Definition}

The set $V_n(x)$ in the definition above is called a {\bf\em $(\lambda,\delta)$-hyperbolic $f$-pre-ball of center $x$ and order $n$}.

In Lemma~\ref{LemmaHyperbolicPreBalls} below, consider $M$ a compact Riemannian manifold, $\cc\subset M$ a compact set and $f:M\setminus\cc\to M$ a $C^1$ local diffeomorphism with $\|Df(x)\|\le L$ for every $x\in M\setminus\cc$.
Let $\psi:[0,+\infty)\to[0,+\infty)$ and $\varphi:\NN\times M\to [0,+\infty)$ be measurable functions with
$\lim_{t\to0}\psi(t)=\psi(0)=0$.

Suppose that $\|(Df^{n}(x))^{-1}\|^{-1}\ge\varphi(n,x)$ for every $x\in M\setminus\bigcup_{j=0}^{n-1}f^{-j}(\cc)$ and  $n\in\NN$.
Consider some fixed $\delta, r >0$ and $\lambda>1$ and, for any given $x\in M$, define $\cv(x)$ as the  set of all $n\in\NN$ such that $\dist(f^{n-k}(x),\cc)\ge\delta\lambda^{-(k-1)/2}$ for every $0< k\le n$ and $\varphi(n-k,f^k(x))>r\lambda^{n-k}$ for every $0\le k<n$.

Lemma~\ref{LemmaHyperbolicPreBalls} generalizes Lemma~5.2 at \cite{ABV}.
Indeed, Lemma~5.2 follows from the lemma below by taking $r=1=\ell$, $\varphi(n,x)=\prod_{j=0}^{n-1}\|(Df\circ f^j(x))^{-1}\|^{-1}$ and $\psi(t)=Bt^{\beta}$ for some $B$ and $\beta>0$.

If $\Gamma:[0,1]\to M$ is a $C^1$ curve, let $\length(\Gamma)=\length(\im(\Gamma))=\int_{t\in[0,1]}|\Gamma'(t)|dt$ be the length of $\im(\Gamma)$.
Let $r_0>0$ small such that for every $x\in M$ and $y\in B_{r_0}(x)$ there exists a unique {\bf\em geodesic segment} contained in $B_{r_0}(x)$, beginning at $x$ and ending at $y$ (this geodesic segment is denoted by $[x,y]$, with $[x,y]:[0,1]\to  B_{r_0}(x)$, $[x,y](0)=x$ and $[x,y](1)=y$). 

\begin{Lemma}[Hyperbolic pre-balls]
\label{LemmaHyperbolicPreBalls}
Take $\ell\ge1$ such that $r\lambda^{\ell/2}\ge\sqrt{\lambda}$ and suppose that
\begin{equation}\label{Equationhgfivo66}
 \varphi(n-k,f^k(y))\ge\varphi(n-k,f^k(x))\, e^{-\sum_{j=k}^{n-1}\psi\big(\frac{\length(f^j(\Gamma))}{\dist(f^j(\Gamma),\cc)}\big)}
\end{equation}
for every $x,y\in M\setminus\bigcup_{j=0}^{n-1}f^{-j}(\cc)$, $n\in\cv(x)$ and $0\le k<n$ and any $C^1$ curve $\Gamma:[0,1]\to M\setminus\bigcup_{j=0}^{n-1}f^{-j}(\cc)$ with $\Gamma(0)=p$ and $\Gamma(1)=q$.

If we choose any $0<\delta_1<\frac{\delta/2}{(L\sqrt{\lambda})^{\ell-1}}$ satisfying  $\psi(t)\le\frac{\log\lambda}{2\ell}(1-1/\sqrt{\lambda})$ for every $t\in[0,2\frac{\delta_1}{\delta}(L\sqrt{\lambda})^{\ell-1}]$, then for every $n\in\cv(x)$, with $n\ge\ell$, there 
is an open neighborhood $V_n^0(x)$ of $x$ such that
\begin{enumerate}
	\item $f^n$ sends diffeomorphically $\overline{V_n^0(x)}$ to $\overline{B_{\delta_1}(f^n(x))}$; 
	\item $\|(Df^{n-j}(f^j(y)))^{-1}\|^{-1}>r\lambda^{j/2}$ for every $y\in\overline{V_n^0(x)}$ and $\ell\le j<n$.
\end{enumerate}
In particular, if we define $\cu(y)$ as the set of all $(\sqrt{\lambda},\delta_1)$-hyperbolic time for a given  $y\in M$ with respect to $f^{\ell}$, then for each $x\in M$ there is a $0\le s<\ell$ such that $$\#\big(\cu(f^s(x))\cap\{1,\cdots,n\}\big)\ge\frac{1}{\ell}\#\big(\cv(x)\cap\{1,\cdots,n\}\big),\text{ for every }n\ge\ell.$$
\end{Lemma}

\begin{proof} 
Consider any $n\in\cv(x)$ with $n\ge\ell$. By hypothesis $\varphi(\ell,f^{n-\ell}(x))>r\lambda^{\ell}=r\lambda^{\ell/2}\lambda^{\ell/2}\ge \lambda$ and $\dist(f^{n-j}(x),\cc)\ge\delta_0:=\delta/(L\sqrt{\lambda})^{\ell-1}$ for every $0<j\le\ell$. 
Hence, if $y\in\overline{B_{\delta_1}(f^{n-\ell}(x))}$, then $f^j(y)\notin\cc$ for every $0\le j<\ell$.
Moreover, $$\dist(f^j(y),\cc)\ge\dist(f^j(f^{n-\ell}(x)),\cc)-\dist(f^j(y),f^j(f^{n-\ell}(x)))>$$
$$>\delta\lambda^{-(\ell-1)/2}-\delta_1 L^{\ell-1}>\delta\lambda^{-(\ell-1)/2}/2$$
and $$\length(f^j([y,f^{n-\ell}(x)]))\le L^{j}\delta_1\le L^{\ell-1}\delta_1$$
for every $0<j\le\ell$ and $y\in\overline{B_{\delta_1}(f^{n-\ell}(x))}$.
So, $$\frac{\length(f^j([y,f^{n-\ell}(x)]))}{\dist(f^{j}([f^{n-\ell}(x)),y]),\cc)}<\frac{L^{\ell-1}\delta_1}{\delta\lambda^{-(\ell-1)/2}/2}=2\frac{\delta_1}{\delta}(L\sqrt{\lambda})^{\ell-1},$$
for every $0<j\le\ell$ and $y\in\overline{B_{\delta_1}(f^{n-\ell}(x))}$.
As a consequence, 
$$\varphi(\ell,y)\ge  \varphi(\ell,f^{n-\ell}(x)) e^{-\sum_{j=0}^{\ell-1}\psi\big(\frac{\length(f^j([y,f^{n-\ell}(x)])}{\dist(f^j([y,f^{n-\ell}(x)]),\cc)}\big)}\ge $$

$$>r\lambda^{\ell}\big(e^{-\frac{\log\lambda}{2\ell}(1-1/\sqrt{\lambda})}\big)^{\ell}=r\lambda^{\ell}\lambda^{-1/2}\ge r\lambda^{\ell/2}\ge\sqrt{\lambda}.$$
Thus, $\|(Df^{\ell}(y))^{-1}\|^{-1}>\sqrt{\lambda}$ for every $y\in\overline{B_{\delta_1}(f^{n-1}(x))}$.
In particular, $$f^{\ell}(\overline{B_{\delta_1}(f^{n-\ell}(x)})\supset\overline{B_{\delta_1}(f^{n}(x))}.$$
Therefore, taking $$V_{\ell}^0(f^{n-\ell}(x))=(f^{\ell}|_{B_{\delta_1}(f^{n-\ell}(x))})^{-1}(B_{\delta_1}(f^{n}(x))),$$
we have that $f^{\ell}$ sends diffeomorphically $\overline{V_{\ell}^0(f^{n-\ell}(x))}$ to $\overline{B_{\delta_1}(f^n(x))}$ and $(f^{\ell}|_{\overline{V_{\ell}^0(f^{n-\ell}(x))}})^{-1}$ is a $\lambda^{-1/2}$-contraction. Define, for $0\le j<\ell$,
$$V_{j}^0(f^{n-j}(x))=f^{\ell-j}(V_{\ell}^0(f^{n-\ell}(x))).$$

 Now, let $\ell\le s\le n$ and suppose by induction that there is an open set $V_{s}^0(f^{n-s}(x))$ containing  $f^{n-s}(x)$ such that $f^{s}$ sends $\overline{V_{s}^0(f^{n-s}(x))}$ diffeomorphically to $\overline{B_{\delta_1}(f^{n}(x))}$ and $(f^s|_{\overline{B_{\delta_1}(f^{n}(x))}})^{-1}$ is a $\lambda^{-(s+1-\ell)/2}$-contraction.

Given $y\in\overline{B_{\delta_1}(f^n(x))}$, and $0\le j\le s$, let $\Gamma_{j,y}:[0,1]\to\overline{V_{\ell}^0(f^{n-j}(x))}$ be the $C^1$ curve defined by  $$\Gamma_{j,y}(a)=(f^{j}|_{\overline{V_{j}^0(f^{n-j}(x))}})^{-1}\circ[y,f^n(x)](a).$$

\begin{Claim}\label{Claimohihy7890}
$\frac{\length(\Gamma_{j,y})}{\dist(\Gamma_{j,y},\,\cc)}$ $<$ $2\frac{\delta_1}{\delta}(L\sqrt{\lambda})^{\ell-1}$
 for every $y\in\overline{B_{\delta_1}(f^n(x))}$ and $0<j\le s$.
\end{Claim}
\begin{proof}[Proof of the claim]
As $(f^j|_{\overline{B_{\delta_1}(f^{n}(x))}})^{-1}$ is a $\lambda^{-(j+1-\ell)/2}$-contraction for every $\ell\le j\le s$, we get that $\length(\Gamma_{j,y})\le\frac{\delta_1}{2}\lambda^{-(j+1-\ell)/2}$.
On the other hand, as $\dist(f^{n-j}(x),\cc)\ge\delta\lambda^{-(j-1)/2}$ and $\im(\Gamma_{j,y})\subset B_{\delta_1}(f^{n-j}(x))$, we get that
$$\dist(\Gamma_{j,y},\cc)\ge\delta\lambda^{-(j-1)/2}-\frac{\delta_1}{2}\lambda^{-(j+1-\ell)/2}>\delta\lambda^{-(j-1)/2}-\frac{1}{2}\frac{\delta/2}{(L\sqrt{\lambda})^{\ell-1}}\lambda^{-(j+1-\ell)/2}>\frac{\delta}{2}\lambda^{-(j-1)/2}.$$
and so, $$\frac{\length(\Gamma_{j,y})}{\dist(\Gamma_{j,y},\cc)}<\frac{(\delta_1/2)\lambda^{-(s+1-\ell)/2}}{(\delta/2)\lambda^{-s/2}}=\frac{\delta_1}{\delta}\lambda^{(\ell-1)/2}<2\frac{\delta_1}{\delta}(L\sqrt{\lambda})^{\ell-1}.$$
\end{proof}

 Let $U$ be the connected component of $f^{-1}(\overline{V_{s}^0(f^{n-s}(x))})\cap\overline{B_{\delta_1 \lambda^{-(s+2-\ell)/2}}(f^{n-(s+1)}(x))}$ containing $f^{n-(s+1)}(x)$.

\begin{Claim}\label{Claimiuytr5678nhb}
If $u\in U$ then $\im(\Gamma_{s,f^{s+1}(u)})\subset U$  and $\|(Df^{s+1})^{-1}(y)\|^{-1}>r\lambda^{(s+1)/2}$ for every $y\in (f|_{U})^{-1}(\im(\Gamma_{s,f^{s+1}(u)}))$.	
\end{Claim}
\begin{proof}[Proof of the claim]
If $u\in U$ then $f^{s+1}(u)\in B_{\delta_1}(f^n(x))$ and, as $(f^{s}|_{V_{s}^0(f^{n-s}(x)})^{-1}$ is a $\lambda^{-(s-\ell+1)/2}$-contraction, $\length(\Gamma_{s,f^{s+1}(u)})<\delta_1\lambda^{-(s+1-\ell)/2}$.
Let $$a=\max\{b\in[0,1]\,;\,\Gamma_{s,f^{s+1}(u)}([0,b])\subset U\}$$ and write $\Gamma(b)=(f|_U)^{-1}\circ\Gamma_{s,f^{s+1}(u)}(b)$ for $b\in[0,a]$.

As $\dist(\Gamma(b),f^{n-(s+1)}(x))<\delta_1\lambda^{-(s+2-\ell)/2}$, we have that 
$$\dist(\Gamma(b),\cc)\ge\dist(f^{n-(s+1)}(x),\cc)-\dist(y(b),f^{n-(s+1)}(x))>$$
$$>\delta\lambda^{-(s+1-1)/2}-\delta_1 \lambda^{(-(s+2)+\ell)/2}>\delta\lambda^{-s/2}-\frac{\delta/2}{(L\sqrt{\lambda})^{\ell-1}}\; \lambda^{(-(s+2)+\ell)/2}=$$
$$=\delta\lambda^{-s/2}-\frac{\delta}{2}\frac{1}{(L)^{\ell-1}}\; \lambda^{-(s+1)/2}\ge\delta\lambda^{-s/2}(1-\frac{1}{2L^{\ell-1}})>\delta\lambda^{-s/2}/2$$
and so, $$\frac{\dist(\Gamma(b),f^{n-(s+1)}(x))}{\dist(\{\Gamma(b),f^{n-(s+1)}(x)\},\cc)}<\frac{\delta_1\lambda^{-(s+1-\ell)/2}}{\delta\lambda^{-s/2}/2}<2\frac{\delta_1}{\delta}(L\sqrt{\lambda})^{\ell-1}.$$
This implies that 
\begin{equation}\label{Equatioioio87tr}
  \frac{\length(\Gamma([0,b]))}{\dist(\Gamma([0,b]),\cc)}<2\frac{\delta_1}{\delta}(L\sqrt{\lambda})^{\ell-1},\,\;\forall\,b\in[0,a].
\end{equation}

As, for $1\le j\le s+1$,  $f^j\circ\Gamma(b)=\Gamma_{s+1-j,f^{s+1}(u)}(b)$, it follows from (\ref{Equatioioio87tr}) and Claim~\ref{Claimohihy7890} that $$\varphi(s+1,\Gamma(b))\ge\varphi(s+1,f^{n-(s+1)}(x))\exp\bigg(-\sum_{j=0}^s\psi\bigg(\frac{\length(\Gamma_{s+1-j,y})}{\dist(\Gamma_{s+1-j,y},\,\cc)}\bigg)\bigg)>$$
$$> r\lambda^{s+1}e^{-(s+1)(\log\lambda)(1-1/\sqrt{\lambda})/2}>r\lambda^{s+1}(\sqrt{\lambda})^{-(s+1)}=r\lambda^{(s+1)/2},$$
for every $b\in[0,a]$.
This implies that $\|(Df^{s+1})^{-1}(\Gamma(b))\|^{-1}>\lambda^{(s+1-\ell)/2}$ for every $b\in[0,a]$.
If $a<1$ then $\gamma(a)\in\partial B_{\delta_1\lambda^{-(s+1-\ell)/2}}(f^{n-(s+1)}(x))$ and $f(\Gamma(a))=\Gamma_{s,f^{s+1}(x)}(a)\in\partial V_s^0(f^{n-s}(x))$ and so, $f^{s+1}(\Gamma(a))\in\partial B_{\delta_1}(f^n(x))$, but this is a contradiction.
Indeed, it implies that $\dist(\Gamma(a),f^{n-(s+1)}(x))<\delta_1\lambda^{-(s+1-\ell)/2}$.
Hence $a=1$ and so, $\im(\Gamma_{s,f^{s+1}(u)})$.
\end{proof}

As $f^s(\overline{V_s^0(f^{n-s}(x))})=\overline{B_{\delta_1}(f^n(x))}$, it follows from Claim~\ref{Claimiuytr5678nhb} above that $\Gamma_{s+1,y}:=(f|_U)^{-1}\circ\Gamma_{s,y}$ is a $C^1$ curve and $(f^{s+1}|_U)^{-1}:\overline{B_{\delta_1}(f^n(x))}\to U\subset\overline{B_{\delta_1\lambda^{-(s+1-\ell)/2}}(f^{n-(s+1)}(x))}$
is given by $$(f^{s+1}|_U)^{-1}(y)=\Gamma_{s+1,y}(1).$$
In particular, $\|D(f^{s+1}|_U)^{-1})(y)\|<(r\lambda^{(s+1)/2})^{-1}$ for every $y\in\overline{B_{\delta_1}(f^n(x))}$.
Hence, taking $$V_{s+1}^0\big(f^{n-(s+1)}(x)\big)=(f^{s+1}|_U)^{-1}(B_{\delta_1}(f^n(x))=$$
$$=\bigg(f\big|_{B_{\delta_1\lambda^{-(s+1-\ell)/2}}\big(f^{n-(s+1)}(x)\big)}\bigg)^{-1}\big(V_{s}^0(f^{n-s}(x))\big),$$
we get that $\overline{V_{s+1}^0(f^{n-(s+1)}(x))}$ is sent diffeomorphically by $f^{s+1}$ to $\overline{B_{\delta_1}(f^n(x))}$ and  $$\|(Df^{s+1})^{-1}(y)\|^{-1}>r\lambda^{(s+1)/2}$$ for every $y\in V_{s+1}^0(f^{n-(s+1)}(x))$.
Thus, by induction, we conclude the proof of the lemma.\end{proof}

	A set $\Lambda\subset M$ has {\bf\em slow recurrence to the critical/singular region} (or  {\bf\em satisfies the slow approximation condition}) if
for each $\varepsilon>0$ there is a $\delta>0$
such that
$$
\limsup_{n\to+\infty}
\frac{1}{n} \sum_{j=0}^{n-1}-\log \mbox{dist}_{\delta}(f^j(x),\cc)
\le\varepsilon
$$
for every $x\in\Lambda$,
where $\dist_{\delta}(x,\cc)$ denotes the $\delta$-{\bf\em truncated distance} from $x$ to $\cc$ defined as
$\dist_{\delta}(x,\cc)=\dist(x,\cc)$ if $\dist(x,\cc) \leq \delta$
and $\dist_{\delta}(x,\cc) =1$ otherwise.

A  measure $\mu$ {\bf\em has slow recurrence to the critical/singular region} (or {\bf\em satisfies slow approximation condition}) if there is a full measure set $\Lambda$ (i.e., $\mu(M\setminus\Lambda)=0$) with slow recurrence to the critical/singular region.

\subsection{Example of synchronization of an endomorphism with critical/singular region}

Let $g:[0,1]\to[0,1]$ be a quadratic map $g(x)=4t_0x(1-x)$ for some $t_0\in(0,1)$ such that there exists a $g$-invariant probability $\mu\ll\leb_{[0,1]}$.
Let $S^1=\RR/\ZZ$ and $h_0:S^1\to S^1$ given by $h_0([x])=[2 x]$, where $[x]=\{y\in\RR\,;\,x-y\in\ZZ\}$ is the class of $x\in\RR$.
Let $h_1:S^1\to S^1$ be a $C^2$ preserving orientation diffeomorphism such that $\sigma=\min\{h_1'([x])\,;\,x\in[0,1]\text{ and }j\in\{0,1\}\}\le1$ (for instance, we may consider $h_1$ as in Figure~\ref{Figure2Difeos.png}). 

\begin{Lemma}\label{Lemmakytrdcvbhu65e}If $\mu\ll\leb_{[0,1]}$ then 
	$\int_{x\in[0,1]}|\log|x-r|| d\mu<+\infty$ for $\mu$-almost every $\rho\in[0,1]$.
\end{Lemma}
\begin{proof}
As $\mu\ll\leb$, it follows from Radon-Nikodym 
theorem that for $\mu$ almost every $x\in[0.1]$ there is  $K_x>0$ such that $\mu(B_{\delta}(x))\le K_x\leb(B_{\delta}(x))$. So, setting $R(x)=|\log|x-\rho||$,  we get that 
$\int_{x\in[0,1]}|\log|x-\rho|| d\mu$ $\le$ $\sum_{j=1}^{\infty}(n+1)\mu(R^{-1}([n,n+1)))\le\sum_{j=1}^{\infty}(n+1)\mu(\{R$ $\ge$ $n\}))$ $=$
$\sum_{j=1}^{\infty}(n+1)\mu(B_{e^{-n}}(\rho))\le2K_r\sum_{j=1}^{\infty}(n+1)e^{-n}<+\infty.$
\end{proof}

It is well known that there exists $C>0$ such that $\mu\ge C\leb|_{\supp\mu}$ and that $B_{\delta_0}(1/2)\subset\supp\mu$ for some $\delta_0>0$.
Hence, it follows form Lemma~\ref{Lemmakytrdcvbhu65e} that $U_{\mu}=\{\rho\in(0,\delta_0)\,;\,\int|\log|x-(c\pm\rho)||d\mu<+\infty\}$ contains $\leb_{[0,1]}$ almost every point of $(0,\delta_0)$.
Consider $\rho\in U_{\mu}$ such that $\lambda:=2^{\mu(I_0)}\sigma^{\mu(I_1)}>1$, where $I_0=[0,1]\setminus I_1$ and $I_1=(1/2-\rho,1/2+\rho)$.
Let $M:=[0,1]\times S^1$ and $\ii:[0,1]\to\{0,1\}$ be given by
$$\ii(x)=\chi_{_{I_1}}(x)=
\begin{cases}
0 & \text{ if }x\in I_0\\
1 & \text{ if }x\in I_1
\end{cases}.$$

\begin{figure}
  \begin{center}\includegraphics[scale=.15]{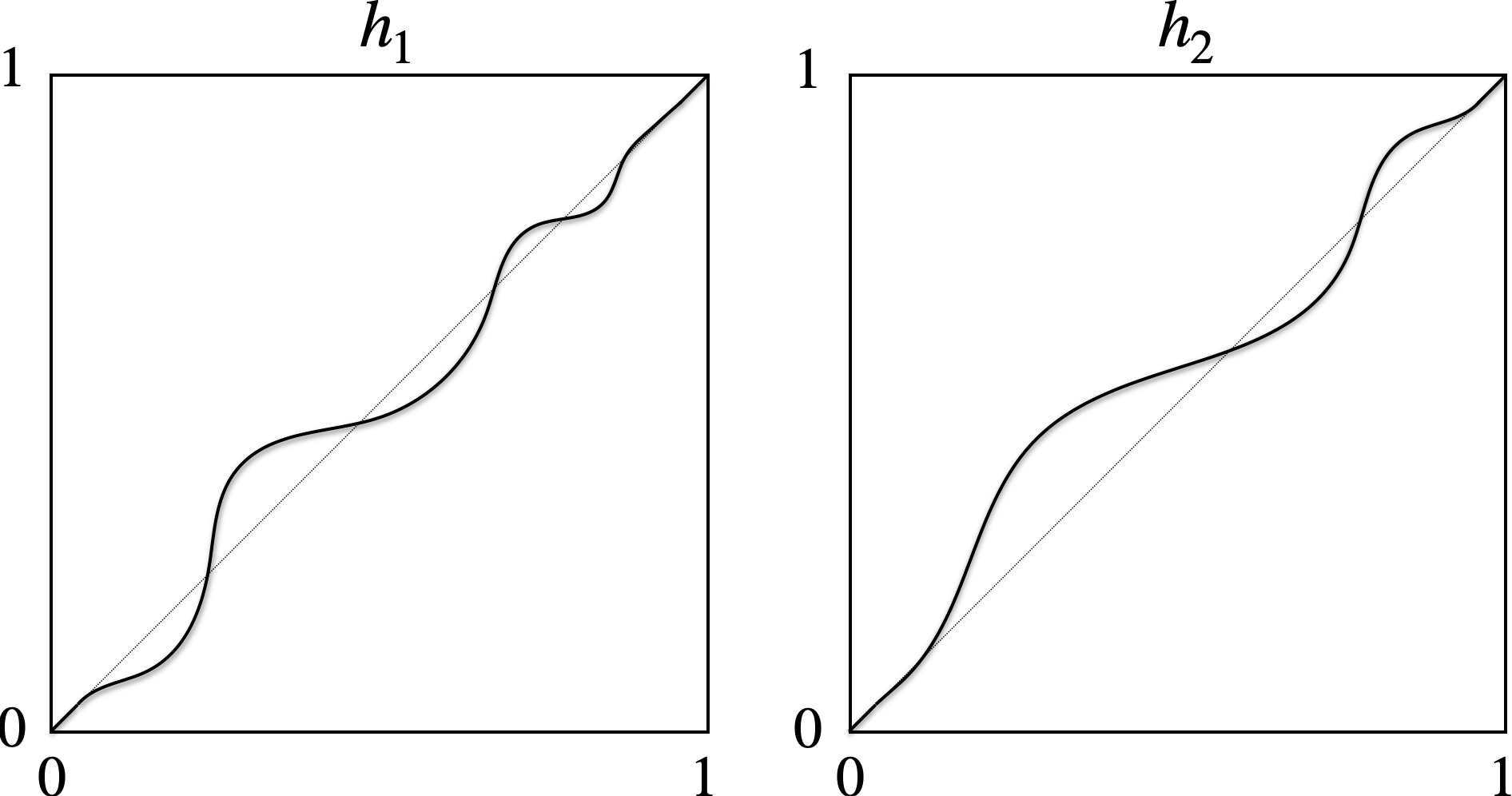}\\
 \caption{Examples of $C^2$ preserving orientation diffeomorphisms of the interval that induces $C^2$ diffeomorphisms of the circle $S^1=\RR/\ZZ$ as asked in Theorem~\ref{TheoremExemploSyEnd}.}\label{Figure2Difeos.png}
  \end{center}
\end{figure}

\begin{figure}
  \begin{center}\includegraphics[scale=.2]{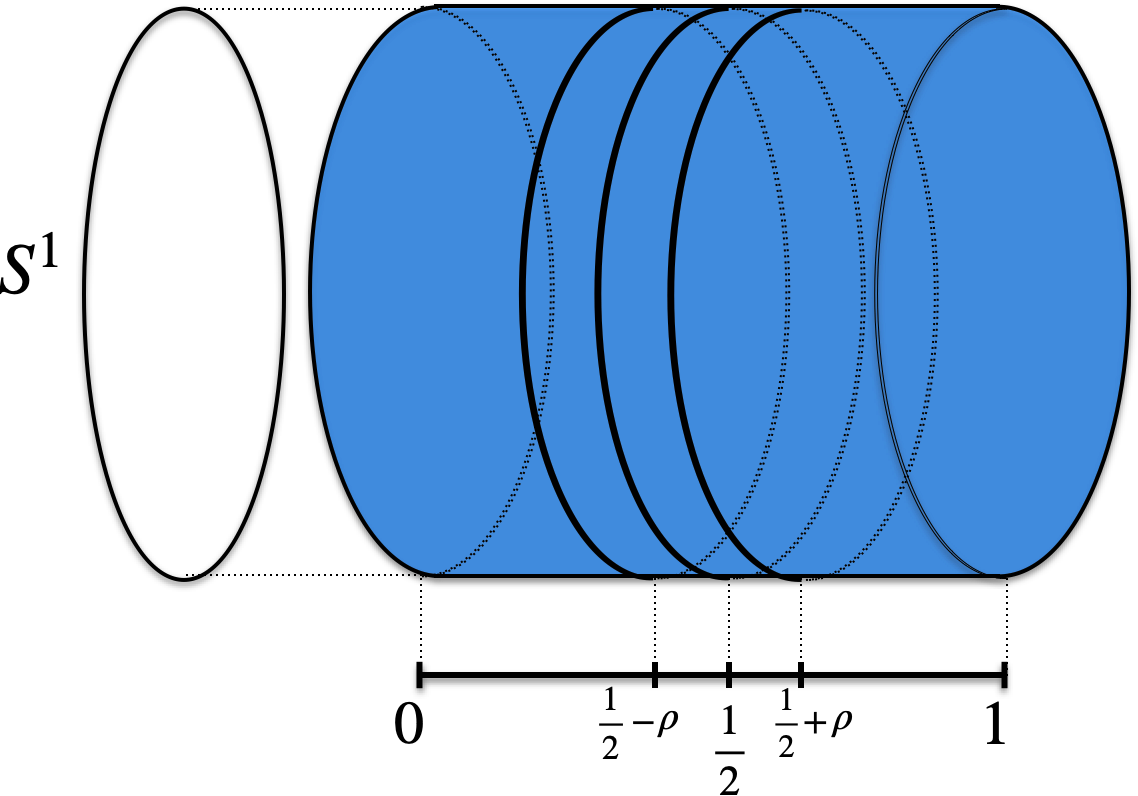}\\
 \caption{In this picture we see the manifold with boundary $M=[0,1]\times S^1$ and the critical/singular region $\cc_f=\{1/2-\rho,1/2,1/2+\rho\}\times S^1$ for the map $f$ in Theorem~\ref{TheoremExemploSyEnd}.}\label{Cilindro.png}
  \end{center}
\end{figure}

\begin{T}\label{TheoremExemploSyEnd} If $f:M\mapsto  M$ is given by $f(x,[y])=\big(g(x),h_{\ii(x)}([y])\big)$ then 
Lebesgue almost every point $p\in M$ has only positive Lyapunov exponents and there is a finite collection $\{\nu_1,\cdots,\nu_k\}$ of ergodic absolutely continuous $f$-invariant probabilities such that $\leb(M\setminus \bigcup_{j=1}^k\beta_f(\nu_j))=0$, where $\beta_f(\nu_j)$ is the basin of attraction \footnote{ The definition is given by equation (\ref{DefBasinOfAttraction}).} of $\nu_j$.
\end{T}
\begin{proof}

The critical/singular region of $f$ is the union of three circles. Precisely, $\cc_f=\{1/2\}\times S^1$ $\cup$ $\{1/2-\rho\}\times S^1$ $\cup$ $\{1/2+\rho\}\times S^1$.
Note that $f$ is discontinuous on the circles $\{1/2\pm\rho\}\times S^1$. Although $f$ is $C^2$ in a neighborhood of $\{1/2\}\times S^1$, $\{1/2\}\times S^1$ is a critical region in the sense that $\det Df(p)=0$ for every $p\in \{1/2\}\times S^1$. Outside of $\cc_f$ $f$  is a $C^{2}$ local diffeomorphism.

Given $p=(x,[y])\in M\setminus\bigcup_{j=0}^{n-1}f^{-j}(\cc_f)$, write $(g^j(x),[y_j],)=f^j(p)$.
As $$Df(p)=\left(\begin{matrix}
 g'(x) & 0 \\
 0 & h_{\ii([x])}'([y]) 
\end{matrix}\right),$$
we have that $$Df^n(p)=\left(\begin{matrix}
 (g^n)'(x) & 0\\
0 & \prod_{j=0}^{n-1} h_{\ii(g^j([x]))}'([y_j])
\end{matrix}\right).$$

Moreover,  $$\prod_{j=0}^{n-1} h_{\ii(g^j([x]))}'[y_j]\ge 2^{\tau(x,n)}\sigma^{n-\tau(x,n)},$$ where $\tau(x,n)$ $=$ $\#\{0\le j<n\,;\,g^j(x)\in I_0\}$.
It follows from Birkhoff that $\lim_n\frac{1}{n}\log\tau(x,n)=\mu(I_0)$ for $\mu$ (and also for $\leb_{[0,1]}$) almost every $x\in[0,1]$ and so, $$\lim_n\frac{1}{n}\log\prod_{j=0}^{n-1} h_{\ii(g^j(x))}'[y_j]\ge\log\lambda>0$$ for Lebesgue almost every $p=(x,[y])\in M$.
This means that Lebesgue almost every point $p\in M$ has one Lyapunov exponent equal to $\int\log|g'|d\mu$ and the other one bigger or equal to $\log\lambda>0$.

Let $\dist$ and $\dist_{[0,1]}$ be the distance on respectively $M$ and $[0,1]$ given by the Riemannian metric on $M$ and the usual distance on the interval. Let $\dist_{\delta}$ and $\dist_{[0,1],\delta}$ be the $\delta$-truncated distance on respectively $M$ and $[0,1]$.
Let $\cc_0=\{1/2-\rho,1/2,1/2+\rho\}$.
As $\dist_{\delta}((x,[y]),\cc_f)=\dist_{[0,1],\delta}(x,\cc_0)$, it follows from  Birkhoff and from $\leb_{[0,1]}(\beta_g(\mu))=1$ that exists $\Lambda_0\subset[0,1]$, with $\leb_{[0,1]}(\Lambda_0)=1$, such that 
\begin{equation}
\label{Equationkjhog6trfv}
\lim_n\frac{1}{n}\sum_{j=0}^{n-1}|\log\dist_{\delta}(f^j(p),\cc_f)|\le\sum_{c\in\cc_0}\int_{x\in B_{\delta}(c)}|\log|x-c||d\mu
\end{equation}
for every $p\in\Lambda\times S^1$.

It follows from $\rho\in U_{\mu}$ that 
$\int_{x\in[0,1]}|\log|x-(1/2\pm\rho)||d\mu<+\infty$ and so, $\int_{x\in B_{\delta}(c)}|\log|x-c||d\mu\to0$ when $\delta\to0$ and $c=1/2\pm\rho$. 
On the other hand, as $\log|g'(x)|=\log(8t)+\log|x-1/2|$ and $\int|\log|g'||d\mu<+\infty$, we get that $\int|\log|x-1/2||d\mu<+\infty$ and so, $\lim_{\delta\to0}\int_{x\in B_{\delta}(1/2)}|\log|x-1/2||d\mu=0$.
Thus, it follows from (\ref{Equationkjhog6trfv}) that $\Lambda_0\times S^1$  has slow recurrence to $\cc_f$ and, as $\leb(\Lambda_0\times S^1)=1$, we have that $\leb$ has slow recurrence to the critical/singular region.

Therefore, there exists a set $\Lambda\subset M$ with full Lebesgue measure, slow recurrence to $\cc_f$ and all Lyapunov exponents bigger or equal to $\log\lambda>0$ for every $p\in\Lambda$.

Let $\cu_0(x)$ be the set of all $(\frac{1}{2}\log\lambda,\Phi_1)$-Pliss times to $x\in[0,1]$ with respect to $g$ and, given $[y]\in S^1$, let 
 $\cu_{2,[y]}(x)$ be the set of all $(\frac{1}{2}\log\lambda,\Phi_{2,[y]})$-Pliss times to $x\in[0,1]$ with respect to $g$ (Definition~\ref{DefinitionPlissTime}), where $\Phi_1,\Phi_{2,[y]}:\NN\times [0,1]\to\RR$ are the $g$-cocycles given by
 $$\Phi_1(n,x)=\log|(g^n)'(x)|$$
and
$$\Phi_{2,[y]}(n,x)=\sum_{j=0}^{n-1}\log h_{\ii\circ g^j(x)}([y_j])|$$

It follows from the Synchronization  Theorem (Theorem~\ref{TheoremSynINV}) that exists $\ell_0\ge0$ and $\theta_0>0$ such that $$\lim\frac{1}{n}\#\big\{0\le j<n\,;\,j\in\cv_{1,[y]}(x))\big\}=\theta_0$$
for Lebesgue almost every $[x]\in S^1$, where  $$\cv_{1,[y]}(x)=\{j\in\NN\,;\,j\in\cu_1(x)\text{ and }j+\ell_0\in\cu_{2,[y]}(x)\}.$$
Given $0<\delta<1$, let $\cv_{\delta,[y]}(x)$ be the set of all  $n\in\NN$ such that $\dist(f^{n-j}((x,[y])),\cc_f)\ge\delta\lambda^{-(j-1)/2}$ for all $0\le j<n$.
As $\leb$ has slow recurrence to $\cc_f$, given any $0<\gamma<1$, one can choose $\delta>0$ small so that $\lim_n\frac{1}{n}\#\{1\le j\le n\,;\,f^j((x,[y]))\in\cv_{\delta,[y]}(x)\}\ge\gamma$ for $\leb$ almost every $(x,[y])\in M$ (\footnote{ This result comes from the Pliss Lemma and it appears in the proof of the existence of hyperbolic times.
See for instance, the proof of Lemma~5.4 of \cite{ABV}.}).
Hence, taking a small $\delta>0$, we have for $\leb$ almost every $(x,[y])\in M$ that 
$$\liminf\frac{1}{m}\#\big\{1\le n\le m\,;\,n\in\cv_{[y]}(x)\big\}\ge\theta_1$$
for some $0<\theta_1\le \theta_0$, where $\cv_{[y]}(x)=\cv_{1,[y]}(x)\cap\cv_{\delta,[y]}(x)$.

Given $n\in\NN $ and $p=(x,[y])$ and $q=(a,[b])\in M$, define $[y_j]$ and $[b_j]$ by $(g^j(x),[y_j])=f^j(p)$ and $(g^j(a),[b_j])=f^j(q)$.
Write $\alpha_0(n,p)=|(g^n)'(x)|$, $\alpha_1(n,p)=\prod_{j=0}^{n-1} h_{\ii(g^j(x))}'([y_j])$ and let $$\varphi(n,p):=\|(Df^n(p)^{-1}\|^{-1}=\min\{\alpha_0(n,p),\alpha_1(n,p)\}.$$
Note that if $n\in\cv_{[y]}(x)$ then
$$\varphi(n-j,f^j((x,[y]))\ge r\lambda^{j/2}$$ 
for every $0\le j<n$, where $r=\sigma^{\ell_0}$.

Note that
$$
  \frac{\alpha_0(n,p)}{\alpha_0(n,q)}=\prod_{j=0}^{n-1}\frac{|g^j(x)-1/2|}{|g^j(a)-1/2|}\ge e^{-\sum_{j=0}^{n-1}\frac{|g^j(x)-g^j(a)|}{\dist(\{g^j(x),g^j(a)\},1/2)}}\ge e^{-\sum_{j=0}^{n-1}\frac{\dist(f^j(p),f^j(q))}{\dist(\{f^j(p),f^j(q)\},\cc_f)}}
$$
and, as $h_1$ is $C^2$ diffeomorphism with $h'_1\ge\sigma>0$, there exists $\gamma\ge1$ such that 
$h'_1([a])\ge h'_1([b])e^{-\gamma\dist_{S^1}([a],[b])}$.
This implies that
\begin{equation}\label{Equationhftiopo}
  \frac{\alpha_1(n,p)}{\alpha_1(n,q)}\ge  e^{-\gamma\sum_{j=0}^{n-1}\dist(f^j(p),f^j(q))}
\end{equation}
for every $p,q\in M$ and $j\in\{0,1\}$.

Given $\Gamma:[0,1]\to M$ be a $C^1$ curve such that $\Gamma(0)=p$ and $\Gamma(1)=q$.
Let $\length(\Gamma)=\length(\im(\Gamma))=\int_{t\in[0,1]}|\Gamma'(t)|dt$ be the length of $\im(\Gamma)$.

\begin{Claim}\label{Claimjnihvud5kjhvy5555} If $\Gamma:[0,1]\to M\setminus\bigcup_{j=0}^{n-1}f^{-j}(\cc_f)$ is a $C^1$ curve with $\Gamma(0)=p$ and $\Gamma(1)=q$ then 
$\varphi(n,q)\ge\varphi(n,p)\,e^{-\gamma\sum_{k=0}^{n-1}\length(f^j(\Gamma))}$.
\end{Claim}
\begin{proof}
Let $${U_0}=\{p\in M\,;\,\varphi(n,p)=\alpha_0(n,p)\}$$ and 
$${U_1}=\{p\in M\,;\,\varphi(n,p)=\alpha_1(n,p)<\alpha_0(n,p)\}.$$

Note that $\{U_0,U_1\}$ is a partition of $M\setminus\bigcup_{j=0}^{n-1}f^{-j}(\cc_f)$ and 
$$(\star) \begin{cases}
\varphi(n,p)=\alpha_0(n,p) & \text{ if }p\in\overline{U_0}\\
\varphi(n,p)=\alpha_1(n,p) & \text{ if }p\in\overline{U_1}
\end{cases}.
$$

Thus, if $p,q\in\overline{{U_j}}\cap\overline{U_k}$ then $$\frac{\varphi(n,q)}{\varphi(n,p)}=\frac{\alpha_j(n,q)}{\alpha_j(n,p)}=\frac{\alpha_k(n,q)}{\alpha_k(n,p)}\ge e^{-\gamma\sum_{j=0}^{n-1}\frac{\dist(f^j(p),f^j(q))}{\dist(\{f^j(p),f^j(q)\},\cc_f)}}.$$

Suppose that $p\in\overline{U_j}$ and $q\in\overline{U_k}$ with $\overline{U_j}\cap\overline{U_k}=\emptyset$.
In this case, there exists $u\in\Gamma\cap \overline{U_j}\cap\overline{U_k}$.
Thus, $$\frac{\varphi(n,q)}{\varphi(n,p)}=\frac{\varphi(n,q)}{\varphi(n,u)}\frac{\varphi(n,u)}{\varphi(n,q)}=\frac{\alpha_k(n,q)}{\alpha_k(n,u)}\frac{\alpha_j(n,u)}{\alpha_j(n,p)}\ge$$
$$\ge e^{-\gamma\sum_{k=0}^{n-1}\frac{\dist(f^k(q),f^k(u))}{\dist(\{f^k(q),f^k(u)\},\cc_f)}+\frac{\dist(f^k(u),f^k(p))}{\dist(\{f^k(u),f^k(p)\},\cc_f)}}\ge e^{-\gamma\sum_{k=0}^{n-1}\frac{\length(f^k(\Gamma))}{\dist(f^k(\Gamma),\cc_f)}}.$$
\end{proof}

Taking $\psi(t)=-\gamma t$, it follows from Claim~\ref{Claimjnihvud5kjhvy5555} that $\varphi$ satisfies the equation (\ref{Equationhgfivo66}) and so, we can apply Lemma~\ref{LemmaHyperbolicPreBalls}.
Therefore, there exists $\ell\ge1$, $0<\delta_1<\delta$, $\theta>0$ and $U\subset M$, with $\leb(M)=1$, such that for each $p\in U$ there is $0\le s(p)<\ell$ satisfying $$\liminf\frac{1}{n}\#\{1\le j\le n\,;\,j\in\cu(f^{s(p)}(p))\ge\theta,$$
where $\cu(x)$ is the set of all $(\sqrt{\lambda},\delta_1)$-hyperbolic times for $x\in M$ with respect to $f^{\ell}$.

Hence, taking $\mu=\leb|_{U_0}$, where $U_0=\{f^{s(p)}(p)\,;\,p\in U\}$, we can use Theorem~C of \cite{Pi11} (page 916) to conclude that there is a finite collection $\mu_1,\cdots,\mu_k$ of ergodic $\mu$ absolutely continuous $f^{\ell}$-invariant probability such that $\mu$ almost every $p\in U_0$ belongs to the basin of one of these probabilities.
This concludes the proof as it implies that $\leb(\bigcup_{j=1}^k\beta_f(\nu_j))=1$, where $\nu_j:=\sum_{i=0}^{\ell-1}f^i_*\mu_j\ll\leb$ is an ergodic $f$-invariant probability.
\end{proof}
 
\subsection{Example of synchronization of a partial hyperbolic diffeomorphism}\label{AppendixPHYsyn}

Let $A:\RR^2\to\RR^2$ be the linear map $A(x,y)=(2x+y,y+x)$,   
 $\TT^2=\RR^2/\ZZ^2$ and $g:\TT^2\to \TT^2$ the linear Anosov map given by $g([(x,y)])=[A(x,y)]$, where $[(x,y)]=\{(p,q)\in\RR^2\,;\,(x-p,y-q)\in\ZZ^2\}$ is the class of $(x,y)\in\RR^2$.
As the same, let $S^1=\RR/\ZZ$ and  $[x]=\{y\in\RR\,;\,x-y\in\ZZ\}$ is the class of $x\in\RR$.
Recall that $a=(3+\sqrt{5})/2\sim2.61803...>2$ and $b=(3-\sqrt{5})/2\sim0.38197...<1/2$ are the eigenvalues of $Dg([(x,y)])=A$.
Let $\vec{\alpha}=\frac{1}{|\vec{\alpha_0}|}\vec{\alpha_0}$,  $\vec{\beta}=\frac{1}{|\vec{\beta_0}|}\vec{\beta_0}$, where $\vec{\alpha}_0=((3+\sqrt{5})/2-1,1)\in\RR^2$ and $\vec{\beta}_0=((3-\sqrt{5})/2-1,1)\in\RR^2$ are  eigenvectors of $A$.
Note that  $A\vec{\alpha}=a\vec{\alpha}$, $A\vec{\beta}=b\vec{\beta}$, $|\vec{\alpha}|=|\vec{\beta}|=1$ and $\langle\vec{\alpha},\vec{\beta}\rangle=0$.

Consider a $C^2$ map
 $v:[0,1]^2\to[0,1]$ such that 
\begin{enumerate}
	\item $v(t,x)=x$ when $(t,x)\in B_{1/5}(\partial([0,1]^2))$,
	\item for any given $t\in[0,1]$, $[0,1]\ni x\mapsto v(t,x)\in[0,1]$ is a diffeomorphism and
	\item $\frac{b}{a}<\frac{1}{2a}\le\partial_xv(t,x)\le\frac{2}{b}<\frac{a}{b}$ for every $(t,x)\in [0,1]^2$.
\end{enumerate}

Write $v_t(x)=v(t,x)$, let $0<r<1/6$, $U_j=\big[B_{r}((j/2,j/2))\big]\subset\TT^2$, $M=\TT^2\times\TT^2$ and define $f:M\to M$ by

$$f([(x,y)],[z\vec{\alpha}+w\vec{\beta}])=
\begin{cases}
(\,g([(x,y)],\,g([z\vec{\alpha}+w\vec{\beta}])\,) & \text{ if }[(x,y)]\notin U_0\cup U_1\\
(\,g([(x,y)],\,g([v_{w}(z)\vec{\alpha}+w\vec{\beta}])\,)	 & \text{ if }[(x,y)]\in U_0\\
(\,g([(x,y)],\,g([z\vec{\alpha}+v_{z}(w)\vec{\beta}])\,)	 & \text{ if }[(x,y)]\in U_1
\end{cases},
$$
where $z$ and $w$ above are taking in $[0,1]$.

Note that the conditions (1), (2) and (3) assure that $\EE^{cu}\oplus\EE^{cs}$ is a dominating splitting of $TM$, where $\EE^{cu}=\{(x\vec{\alpha},y\vec{\alpha})\in\RR^2\times\RR^2\,;\,x,y\in\RR\}$ and $\EE^{cs}=\{(x\vec{\beta},y\vec{\beta})\in\RR^2\times\RR^2\,;\,x,y\in\RR\}$.
Taking $r>0$ small enough, so that
$\lambda_1:=a^{1-\leb(U_0)}2^{-\leb(U_0)}$ $>$ $1$ $>$ $\lambda_2:=b^{1-\leb(U_1)}2^{\leb(U_1)}$, we get that 
$\lim_{n\to+\infty}\frac{1}{n}\log\|(Df^n(p,q)|_{\EE^{cu}})^{-1}\|^{-1}\ge\lambda_1$ and $\lim_{n\to+\infty}\frac{1}{n}\log\|Df^n(p,q)|_{\EE^{cs}}\|\le\lambda_2$  for $\leb|_{\TT^2}$ almost every $p\in\TT^2$ and every $q\in\TT^2$.
Thus, using Theorem~\ref{TheoremSynINV} applied to the $g$-invariant probability $\leb|_{\TT^2}$ and to synchronize the Lyapunov exponents. Changing the metric (for instance, as in Theoremm~\ref{TheoremHipBlock}) or taking an iterated, one can 
Use Proposition 6.4 of \cite{ABV} to conclude the following result. 

\begin{Proposition}
\label{TheoremNUHDSyn}
The map $f$ is a $C^2$ diffeomorphism admitting a dominated splitting $\EE^{cu}\oplus\EE^{cs}$ such that $$\limsup_{n\to+\infty}\frac{1}{n}\log\|Df^{n}(x)|_{\EE^{cs}}\|<0<\liminf_{n\to+\infty}\frac{1}{n}\log\|(Df^n(x)|_{\EE^{cu}})^{-1}\|^{-1}$$
for Lebesgue almost every $x\in T^2\times M$. Furthermore, Lebesgue 
 almost every point of $M=\TT^2\times\TT^2$ belongs to the basin of attraction of some SRB measure.
\end{Proposition}
\section{Theoretical applications}
\label{SectionExamplesOfTheoApplications}

\subsection{Expanding/hyperbolic invariant measures}
 
In this section we will apply  Theorem~\ref{TheoremLift} and \ref{TheoremFiftErgDec}, Theorem~\ref{TheoremPlissBlock} and Corollary~\ref{CorollaryPlissBlock} to refine some results of the non uniform hyperbolic dynamics.
In Theorem~\ref{Theoremkjhiuyt567} we show several additional properties for the full induced Markov maps/Young Towers that appear in \cite{ALP1,ALP,Pi11}.
In Theorem~\ref{TheoremHipBlock} we construct Hyperbolic Block of Pesin theory using the global continuous metric $\langle u,v\rangle_\ell=\langle  Df^{\ell-1}(x)u, Df^{\ell-1}(x)v\rangle$ for some fixed $\ell\in\NN$.

Let $M$ be a Riemannian manifold. We say that $f:M\to M$ is a $C^{1+}$ {\bf\em
non-flat map} map with {\bf\em critical/singular rigion} $\cc\subset {M}$ if $f$ is a local
$C^{1+}$ (i.e., $C^{1+\alpha}$ with $\alpha>0$ ) diffeomorphism in the whole manifold  except  in $\cc$ and 
$\exists\beta,B>0$ such that
the following conditions holds.
\begin{enumerate}
\item[(C.1)]
\quad $\displaystyle{(1/B)\dist(x,\cc)^{\beta }|v|\le|Df(x)v|\le B\,\dist(x,\cc)^{-\beta }}|v|$ for all $v\in T_x {M}$.
\end{enumerate}
For every $x,y\in {M}\setminus\cc$ with
$\dist(x,y)<\dist(x,\cc)/2$ we have
\begin{enumerate}
\item[(C.2)]\quad $\displaystyle{\left|\log\|Df(x)^{-1}\|-
\log\|Df(y)^{-1}\|\:\right|\le
(B/\dist(x,\cc)^{\beta })\dist(x,y)}$.
\item[(C.3)]\quad $\displaystyle{\left|\log|\det Df(x)|-
\log|\det Df(y)|\:\right|\le
(B/\dist(x,\cc)^{\beta })\dist(x,y)}$
\end{enumerate}

The critical/singular set $\cc$ is called {\bf\em purely critical} if $$\lim_{x\to p}|\det Df(x)|=0$$ for every $p\in\cc$.
On the other hand, if $$\lim_{x\to p}|\det Df(x)|=+\infty$$ for every $p\in\cc$, we say that $\cc$ is {\bf\em purely singular}.

\begin{Lemma}\label{LemmaNonFlatToSlowRecurrence}
Let $M$ be a Riemannian manifold and $f:M\to M$ a $C^{1+}$ non-flat map with critical/singular set  $\cc\subset {M}$. Suppose that $\cc$ is either purely critical or purely singular.
If $\mu$ is a $f$-invariant
ergodic probability with all of its Lyapunov exponent finite, i.e., $\lim_n\frac{1}{n}\log|Df^n(x)v|\ne\pm\infty$ for $\mu$ almost every $x$ and every $v\in T_xM\setminus\{0\}$,
then $x\mapsto\log\dist(x,\cc)$ and $x\mapsto\log\|(Df(x))^{-1}\|$ are $\mu$-integrable. In particular, $\mu$ has slow recurrence to the critical/singular region \footnote{Recall the definition of slow recurrence in Appendix~\ref{SecHYpTIm}.}.
\end{Lemma}
\begin{proof}
Consider the function $\varphi:M\to[0,+\infty)$ defined as
$$\varphi(x)=
\begin{cases}
	0 & \text{ if }x\in\cc\\
	\det Df(x) & \text{ if }x\notin\cc \text{ and $\cc$ is purely critical}\\
	\frac{1}{\det Df(x)} & \text{ if }x\notin\cc \text{ and $\cc$ is purely singular}
\end{cases}
$$
As $f$ is $C^{1+}$, we get that $\varphi$ is a Hölder function and so, $\cc=\varphi^{-1}(0)$ is a compact subset of $M$.
We may assume
that $\cc\ne\emptyset$.
As $\varphi$ is Holder, $\exists\,k_0,k_1>0$ such that  $|\varphi(x) - \varphi(y)|$ $\le$ $k_0\dist(x,y)^{k_1}$ $\forall\,x,y\in M$.
Given $x\in M$ there is $y_x\in\cc$ such that $\dist(x,y_x)=\dist(x,\cc)$.
Thus, we get $|\varphi(x)|$ $=$ $|\varphi(x) - \varphi(y_x)|$ $\le$ $k_0\dist(x,y_x)^{k_1}$ $=$ $k_0\dist(x,\cc)^{k_1}$.
That is,
\begin{equation}\label{EquationJHGJKJH7}
  \log|\varphi(x)|\le\log k_0+k_1\log\dist(x,\cc)\;\,\forall\,x\in M.
\end{equation}

Let $m=dimension(M)$ and note that $\|A^{-1}\|^{-m}\le|\det A|\le\|A\|^{m}$ for every $A\in GL(m,\RR)$.
That is, $$m\log\big(\|A^{-1}\|^{-1}\big)\le\log|\det A|\;\;\text{ and }\;\;\log\|A\|\ge-\frac{1}{m}\log\bigg|\frac{1}{\det A}\bigg|.$$
Thus, 
if $\int\log|\varphi|\,d\mu=-\infty$, it follows from Birkhoff that either $$\limsup_{n\to\infty}\frac{1}{n}\log\|(Df^{n}(x))^{-1}\|^{-1}\le\frac{1}{m}\limsup_{n\to\infty}\frac{1}{n}\log|\det Df^n(x)|=$$
$$=\frac{1}{m}\,\lim\frac{1}{n}\sum_{j=0}^{n-1}\log|\det Df(f^j(x))|=\frac{1}{m}\,\lim\frac{1}{n}\sum_{j=0}^{n-1}\log|\varphi\circ f^j(x)|=-\infty$$
for $\mu$-almost every $x$ (when $\cc$ is purely critical) or, when $\cc$ is purely singular, 
$$\limsup_{n\to\infty}\frac{1}{n}\log\|(Df^{n}(x))\|\ge-\frac{1}{m}\liminf_{n\to\infty}\frac{1}{n}\log\bigg|\frac{1}{\det Df^n(x)}\bigg|=$$
$$=-\frac{1}{m}\,\lim\frac{1}{n}\sum_{j=0}^{n-1}\log\bigg|\frac{1}{\det Df(f^j(x))}\bigg|=-\frac{1}{m}\,\lim\frac{1}{n}\sum_{j=0}^{n-1}\log|\varphi\circ f^j(x)|=+\infty$$
for $\mu$-almost every $x$. In any case, we have a contradiction to our hypothesis.
So, $\int\log|\varphi|d\mu>-\infty$ and, by (\ref{EquationJHGJKJH7}), we get that  $$-\infty<\int\log|\varphi|\,d\mu-\log k_0\le k_1\int_{x\in M}\log\dist(x,\cc)\,d\mu\le k_1\log\diameter(M),$$  proving that the logarithm of the distance to the critical set is integrable. As a consequence, $$\int\log\dist_{e^{-n}}(x,\cc)\,d\mu(x)=
 \int_{\{x\,;\,\log\dist(x,\cc)<-n\}}\log\dist(x,\cc)\,d\mu\rightarrow0$$
 when $n\rightarrow\infty$.
and this implies (by Birkhoff)  the slow approximation condition. Furthermore, as $\mu(\cc)=0$, it follows from the condition (C1) on the definition of a non-degenerated critical/singular set that $$-\log B+\beta\log\dist(x,\cc)\le-\log\|(Df(x))^{-1}\|\le\log B-\beta\log\dist(x,\cc).$$
Thus, the integrability of $\log\|Df^{-1}\|$ follows from the integrability of $x\mapsto\log\dist(x,\cc)$.
\end{proof}

\begin{Definition}
	We say that an ergodic $f$ invariant probability $\mu$ is a {\bf\em synchronized expanding measure} if there exists $\ell\ge1$ such that  
\begin{equation}\label{EquationExpanding}
\limsup_{n\to\infty}\frac{1}{n}\sum_{i=0}^{n-1}
\log \|(Df^{\ell}(f^{i\ell}(x)))^{-1}\|^{-1}\ge\lambda,
\end{equation}
holds for $\mu$ almost every $x\in M$.
If a synchronized expanding measure $\mu$ satisfies the slow approximation condition, then $\mu$ is called a  {\bf\em geometric expanding measure}.
\end{Definition}

The main property of a geometric expanding measure $\mu$ is the existence, for $\mu$ almost every $x$, 
of a positive frequency of geometric hyperbolic times (see Definition~\ref{DefinitionGeomHyTimes}), that are useful in many applications, in particular, in the construction of induced Markov maps (see Lemma~2.7 in \cite{ABV} and Lemma~2.1 in \cite{ALP}).

\begin{Proposition}\label{Propositionmytfgyu}
Let $M$ be a Riemannian manifold and $f:M\to M$ a $C^{1+}$ non-flat map with critical/singular set  $\cc\subset {M}$.
Suppose that $\cc$ is either purely critical or purely singular.  If $\mu$ is a $f$-invariant
ergodic probability having only positive and finite Lyapunov exponent, i.e., $$0<\lim_{n}\frac{1}{n}\log\|(Df^{n}(x))^{-1}\|^{-1}\le\lim_{n}\frac{1}{n}\log\|Df^{n}(x)\|<+\infty$$ for $\mu$-almost every $x$,
then $\mu$ is a geometric expanding measure.
\end{Proposition}
\begin{proof} As $\cc$ is either purely critical or purely singular, we get that either $$\sup\log\|(Df)^{-1}\|^{-1}\le\sup\log\|Df\|<+\infty$$
or
$$-\infty<\inf\log\|(Df)^{-1}\|^{-1}\le\inf\log\|Df\|.$$ 
Thus, it follows from Furstenberg-Kesten Theorem \cite{FK}, together with the ergodicity of $\mu$ and the hypothesis of the exponents being positive and finite, there exists $0<\lambda<+\infty$ such that $$\lim_{n\to+\infty}\frac{1}{n}\log\|(Df^{n}(x))^{-1}\|^{-1}=\lambda$$ for $\mu$ almost every $x$.
So, as $\int|\log\|(Df(x))^{-1}\||d\mu<+\infty$ (Lemma~\ref{LemmaNonFlatToSlowRecurrence}), we can apply Lemma~\ref{Lemmahdhhg620} and get that there exists $\ell\ge1$ such that
$$\int\log\|(Df^{\ell})^{-1}\|^{-1} d\mu\ge\frac{\lambda}{4}\ell.$$
As $\mu$ is $f^{\ell}$-invariant and $\mu$ has at most $\ell$ ergodic components, it follows from Birkhoff's Theorem that there exists $\ch_0$ with $\mu(\ch_0)=1$ and $f^{-\ell}(\ch_0)=\ch_0\mo\mu$ such that every $x\in\ch_0$ satisfies (\ref{EquationExpanding}).
As $\mu(\cc)=0$ and $\mu(\bigcup_{j=0}^{\ell-1}f^j(\ch_0))=1$, we get that there exists $\ch_1$, with $\mu(\ch_1)=1$, such that (\ref{EquationExpanding}) is true for every $x\in\ch_1$.
From Lemma~\ref{LemmaNonFlatToSlowRecurrence}, we get also that there exists $\ch_2$, with $\mu(\ch_2)=1$, such that $\ch_2$ satisfies the slow approximation condition.
Hence, $\ch=\ch_1\cap\ch_2$ is an expanding set with $\mu(\ch)=1$, proving that $\mu$ is an expanding measure.
\end{proof}

All induced maps constructed in \cite{Pi11} are orbit-coherent. In particular, in the theorem about lift in \cite{Pi11} (Theorem~1), it was assumed explicitly the hypothesis of orbit-coherence and this result is used to lift the invariant probabilities in all induced maps there.
Below we give examples of results that can be obtained mixing the results in \cite{Pi11} with the those in the present paper.

\begin{T}\label{Theoremkjhiuyt567}
Let $M$ be a Riemannian manifold and $f:M\to M$ a $C^{1+}$ non-flat map with critical/singular set  $\cc\subset {M}$.
Suppose that $\cc$ is either purely critical or purely singular. If $\mu$ is a $f$-invariant
ergodic probability having only positive and finite Lyapunov exponent, then 
there are open sets $B_j\subset B\subset M$, $j\in\NN$, and an induced map $F:A:=\bigcup_{j}B_j\to B$ such that
\begin{enumerate}
	\item $B_j\cap B_k=\emptyset$ whenever $j\ne k$;
	\item $R$ is constant on each $B_j$, where $R$ is the induced time of $F$;
	\item for every $j\ge1$, $F|_{B_j}$ is a diffeomorphism between $B_j$ and $B$;
	\item there is $\lambda>1$ such that $\|(DF(x))^{-1}\|^{-1}\ge\lambda$ for every $x\in\bigcup_j B_j$;
	\item $\mu(\bigcup_j B_j)=\mu(B)>0$;
	\item $F$ is orbit-coherent;
	\item there exists one and only one $F$-lift $\nu$ of $\mu$;
	\item $\nu$ is $F$-ergodic and $\nu\le C\mu$ for some $C>0$.
\end{enumerate}
\end{T}
\begin{proof}
Using Proposition~\ref{Propositionmytfgyu}, we get that $\mu$ is an expanding measure.
Thus, by Theorem~B in \cite{Pi11}, we get that there exists an induced map $F$ satisfying all the first six items of the Theorem~\ref{Theoremkjhiuyt567}.
Furthermore, $\mu$ is $F$-liftable and $F$ orbit coherent, items (7) and (8) follows from Theorem~\ref{TheoremLift}.
\end{proof}

Given a $C^1$ diffeomorphism $f:M\to M$ defined on a compact Riemannian manifold $M$ and $\ell\ge1$, define the equivalent Riemannian metric $\langle .,.\rangle_\ell$ by
$$\langle u,v\rangle_\ell=\langle  Df^{\ell-1}(x)u, Df^{\ell-1}(x)v\rangle$$ for every $u,v\in T_xM$. Given a vector bundle morphism $A:TM\to TM$ and $x\in M$,  define $\|A(x)\|_{\ell}=\max\{|A(x)v|_{\ell}\,;\,v\in T_xM\text{ and }|v|_{\ell}=1\}$. In Theorem~\ref{TheoremHipBlock} below, given an ergodic invariant probability without zero Lyapunov exponents, we use the coherent blocks to produce the {\em Hyperbolic Blocks} of Pesin theory with $\mu$ positive measure. Nevertheless, here, we use the metric $\langle .\,,.\rangle_{\ell}$ instead of the {\em induced Finsler metric} (see \cite{PuS}). Because of that, we do not need the $C^{1+\alpha}$ regularity, it suffices $f$ to be $C^1$.

\begin{T}[Hyperbolic Blocks]\label{TheoremHipBlock}
Let $f:M\to M$ be a $C^{1}$ diffeomorphism. If $\mu$ is a $f$-invariant
ergodic probability without zero Lyapunov exponents,
then there are integers $\ell\ge1$, $m\ge0$, measurable sets $B^{u}$ and $B^{s}$ with $\mu(B^{u}),\mu(B^{s})>0$, $0<\sigma<1$, $C>0$ and a measurable $f$ invariant splitting $TM=E^{s}\oplus E^{u}$ such that
\begin{enumerate}
	\item $\|Df^{-n}|_{E^{u}(x)}\|_{\ell}\le\sigma^n$ for every $x\in B^{u}$ and $n\ge1$;
	\item $\|Df^n|_{E^{s}(x)}\|_{\ell}\le\sigma^{n}$ for every $x\in B^{s}$ and $n\ge1$;
	\item $\ch(\mu):=B^{u}\cap f^{-m}(B^{s})$ has $\mu$ positive measure and
	$$\|Df^{-n}|_{\EE^{u}(x)}\|\le\sigma^n\;\;\text{ and }\;\;\|Df^n|_{\EE^{s}(x))}\|_{\ell}\le C\sigma^n$$ for every $n\ge1$ and $x\in\ch(\mu)$.
\end{enumerate}
\end{T}
\begin{proof}
As $\mu$ is ergodic, it follows from Oseledets Theorem that there exist $U\subset M$, $\lambda>0$ and a measurable splitting $\EE^{u}\oplus\EE^{s}$ such that $\mu(U)=1$, $$\lim\frac{1}{n}\log\|(Df^{n}|_{\EE^{u}(x)})^{-1}\|^{-1}\ge\lambda$$
and 
$$\lim\frac{1}{n}\log\|Df^n|_{\EE^{s}(x)}\|\le-\lambda$$
for every $x\in U$.
Let $\varphi,\psi:\NN\times M\to R$ given by
$$\varphi(n,x)=
\begin{cases}
	\log\|(Df^{n}|_{\EE^{u}(x)})^{-1}\|^{-1} & \text{ if }x\in U\\
	0 & \text{ if }x\notin U
\end{cases}
$$
and
$$\psi(n,x)=
\begin{cases}
	\log\|(Df^{-n}|_{\EE^{s}(x)})^{-1}\|^{-1} & \text{ if }x\in U\\
	0 & \text{ if }x\notin U
\end{cases}
$$
As $f$ is a diffeomorphism and $M$ is compact, we get that $\int_{x\in M}|\varphi(1,x)|d\mu$ and $\int_{x\in M}|\psi(1,x)|d\mu$ $<$ $+\infty$.
As $\varphi$ is a sup-additive cocycle for $f$, $\psi$ is a sup-additive cocycle for $f^{-1}$ and $\mu$ is invariant for both $f$ and $f^{-1}$, it follows from Lemma~\ref{Lemmahdhhg620} that there exists $\ell\ge1$ such that
$$\int_{x\in M}\varphi(\ell,x)d\mu\;\;\text{ and }\int_{x\in M}\psi(\ell,x)d\mu\ge \frac{\lambda}{4}\ell.$$
By Birkhoff, we can take $\ell\ge1$ so that $$\lim_n\frac{1}{n}\sum_{j=0}^{n-1}\varphi(\ell,f^j(x))\;\;\text{ and }\;\;\lim_n\frac{1}{n}\sum_{j=0}^{n-1}\psi(\ell,f^{-j}(x))\ge2\lambda$$
for $\mu$ almost every $x\in M$.

Let $\Phi,\Psi:\NN\times M\to \RR$ be given by $$\Phi(n,x)=\sum_{j=0}^{n-1}\varphi(\ell,f^j(x))\;\;\text{ and }\;\;\Psi(n,x)=\sum_{j=1}^{n}\psi(\ell,f^{\ell-j}(x)).$$
Hence, $\Phi$ is a $f$-additive cocycle and $\Psi$ is a $f^{-1}$-additive cocycle.

Given $x\in M$, let $\cu^u$ be the set of all $(\lambda,\Phi)$-Pliss time for $x$ with respect to $f$ and $\cu^s$ be the set of all $(\lambda,\Psi)$-Pliss time for $x$ with respect to $f^{-1}$.
It follows from Lemma~\ref{Lemma67ygtgh} that $\cu^u$ is a $f$-coherent schedule of events and 
$\cu^s$ is a $f^{-1}$-coherent schedule of events. By Lemma~\ref{Lemma67ygtghXXXX}, both $\cu^u$ and $\cu^s$ have positive upper density for $\mu$ almost every $x$. Indeed, $\dd_{\NN}^+(\cu^{u}(x))$ and $\dd_{\NN}^+(\cu^{s}(x))\ge\frac{1}{2\gamma}$ for $\mu$ almost every $x$, where $\gamma$ $=$ $\ell$ $\max\{|\log\|Df\||,$ $|\log\|Df^{-1}\||\}$.

Let $B^u=B_{\cu^u}$ be the $f$-coherent block for $\cu^u$ and $B^s=B_{\cu^s}$ be the $f^{-1}$-coherent block for $\cu^s$. It follows from Theorem~\ref{TheoremPlissBlock} that $\mu(B^s)$ and $\mu(B^u)\ge\frac{1}{2\gamma}>0$.
By the definition of $\cu^u$-block, given $x\in M$, $1\le n\in\cu^u(x)$ whenever $f^n(x)\in B^u$. Thus, if $f^n(x)\in B^u$ and $v\in\EE^u(x)$ with $|v|_{\ell}=1$, then 
$$
\log|Df^{n}(x)v|_{\ell}=
\sum_{j=0}^{n-1}\log|Df(f^j(x))v_j|_{\ell}=
\sum_{j=0}^{n-1}\log|Df^{\ell-1}(f^{j+1}(x))\,Df(f^j(x))v_j|=
$$
$$=\sum_{j=0}^{n-1}\log|Df^{\ell}(f^{j}(x))v_j|\ge\sum_{j=0}^{n-1}\log\|(Df^{\ell}|_{\EE^{u}(f^j(x))})^{-1}\|^{-1}=\sum_{j=0}^{n-1}\varphi(\ell,f^j(x))=\Phi(n,x)\ge n\lambda,
$$
where $v_j=\frac{Df^j(x)v}{|Df^j(x)v|_{\ell}}$. 

On the other hand, if $f^{-n}(x)\in B^s$ and $v\in\EE^s(f^{-n}(x))$ with $|v|_{\ell}=1$, we get that
$$
\log|Df^{n}(f^{-n}(x))v|_{\ell}=
\sum_{j=0}^{n-1}\log|Df(f^{j-n}(x))v_{j-n}|_{\ell}=$$
$$=
\sum_{j=0}^{n-1}\log|Df^{\ell-1}(f^{j-n+1}(x))\,Df(f^{j-n}(x))v_{j-n}|=
$$

$$=\sum_{j=0}^{n-1}\log|Df^{\ell}(f^{j-n}(x))v_{j-n}|\le\sum_{j=0}^{n-1}\log\|Df^{\ell}|_{\EE^{s}(f^{j-n}(x))}\|=\sum_{j=0}^{n-1}\log\|(Df^{-\ell}|_{\EE^{s}(f^{j-n+\ell}(x))})^{-1}\|=$$
$$=\sum_{j=1}^{n}\log\|(Df^{-\ell}|_{\EE^{s}((f^{\ell-j}(x))})^{-1}\|=-\sum_{j=1}^{n}\psi(\ell,f^{\ell-j}(x))=-\Psi(n,x)\le-n\lambda,
$$
where $v_{j-n}=\frac{Df^j(f^{-n}(x))v}{|Df^j(f^{-n}(x))v|_{\ell}}$.
Thus, taking $\sigma=e^{-\lambda}$, we get $0<\sigma<1$ and 
$$\|Df^{-n}|_{\EE^{u}(p)}\|\;\;\text{ and }\;\;\|Df^n|_{\EE^{s}(q))}\|_{\ell}\le \sigma^n$$
for every $n\ge1$, $p\in B^u$ and $q\in B^s$.
As $\mu$ is ergodic, there is $m\ge0$ such that $\mu(B^u\cap f^{-m}(B^s))>0$.
Finally, using that $\langle.,.\rangle_{\ell}$ is equivalent to the Riemannian metric and $f$ is $C^1$, there is $C>0$ such that $\|Df^n|_{\EE^s(p)}\|_{\ell}\le C\sigma^n$ for every $p\in\ch(\mu)=B^{u}\cap f^{-m}(B^{s})$.
\end{proof}

\section{Auxiliary results and proofs}
\label{AppendixAuxResAndPro}

\begin{Lemma}\label{LemmaExacRtoRR}
Let $f$ be a measure preserving automorphism on a probability space $(\XX,\mathfrak{A},\mu)$ and $F:A\subset\XX\to\XX$ a measurable induced map with induced time $R:A\to\NN$.
Suppose that $\mu$ is $f$ ergodic and that $\nu$ is a $F$-lift of $\mu$.
If $R$ is exact and $\mu(A)=1$ then 
$$\frac{1}{2}\int Rd\nu \int R d\mu\le\int (R)^2 d\nu \le 2\int Rd\nu\int R d\mu.$$
\end{Lemma}
\begin{proof} As $\mu(A)=1$, we get that $\mu(V)=1$, where $V=\bigcap_{n\ge0}f^{-n}(A)$. Write $r_n(x)=\sum_{j=0}^{n-1}R\circ F^j(x)$, for every $x\in V$ and $n\ge1$.
As $R$ is exact, $$\sum_{j=0}^{R(x)-1}R\circ f^j(x)=\sum_{j=0}^{R(x)-1}(R(x)-j)=\sum_{j=1}^{R(x)}j=R(x)(R(x)+1)/2\ge \frac{1}{2}(R(x))^2,$$
	for every $x\in V$. Hence,
$$\frac{1}{2}\sum_{j=0}^{n-1}(R\circ F^j(x))^2\le \sum_{j=0}^{r_n(x)-1}R\circ f^j(x)=\frac{1}{2}\sum_{j=0}^{n-1}R\circ F^j(x)\big(R\circ F^j(x)+1)\big)\le2\sum_{j=0}^{n-1}(R\circ F^j(x))^2.$$
or
\begin{equation}\label{Equationbh7r5f9v}
  \frac{1}{2}\frac{r_n(x)}{n}\frac{1}{r_n(x)}\sum_{j=0}^{r_n(x)-1}R\circ f^j(x)\le \frac{1}{n}\sum_{j=0}^{n-1}(R\circ F^j(x))^2\le2\frac{r_n(x)}{n}\frac{1}{r_n(x)}\sum_{j=0}^{r_n(x)-1}R\circ f^j(x),
\end{equation}
for every $x\in V$. As $\nu\ll\mu$ and $\nu(A_0)=1$, we get that $\mu(A_0)>0$, where $A_0=\bigcap_{n\ge0}F^{-n}(V)$. 
As $\mu$ is ergodic and $f$-inaviant, it follows from Birkhoff that exist $U_0\subset A_0$ with $\mu(U_0)=\mu(A_0$ and such that $\lim_n\frac{1}{r_n(x)}\sum_{j=0}^{r_n(x)-1}R\circ f^j(x)=\int R d\mu$ for every $x\in U_0$. As $R$ is coherent (because $R$ is  exact), it follows from Lemma~\ref{LemmaOrbCoh1} that $F$ is orbit coherent. Thus, $\nu$ is also ergodic (see Theorem~\ref{TheoremLift}).
Thus, using Birkhoff again, there is a measurable set $U\subset U_0$ with $\nu(U)=1$, such that $\lim_n\frac{r_n(x)}{n}=\int R d\nu$ and $\lim_n\frac{1}{n}\sum_{j=0}^{n-1}(R)^2\circ F^j(x)=\int (R)^2 d\nu$ for every $x\in U$.
Thus, taking any $x\in U_0$ and applying the limit on (\ref{Equationbh7r5f9v}) we conclude the proof. \end{proof}

\begin{proof}[\bf Proof of Lemma~\ref{Lemma222jhf76h3}]
Let $F:\Sigma_2^+\setminus\{\overline{0}\}\to\Sigma_2^+$ and $R:\Sigma_2^+\setminus\{\overline{0}\}\to\NN$ be the $\sigma$-induced map and time of Example~\ref{Examplejg6f7tv}.
Let $\cu:\Sigma_2^+\to2^{\NN}$ be the schedule of events given by $\cu(x)=\{j\ge1\,;\,x(j)=1\}$ for every $x=(x(1),x(2),\cdots)\in\Sigma_2^+$.
One can see that $R=R_{\cu}$ (i.e., $R(x)=\min\cu(x)$) and that $j\in\cu(x)$ $\iff$ $\sigma^{j-1}(x)\in C_+(1)$.
Thus, $$\dd_{\NN}^+(\cu(x))=\dd_{\NN}^+(\{j\in\NN\,;\,\sigma^j(x)\in C_+(1)\}).$$
As $\co_{\sigma}^+(x)\cap C_+(1)=\emptyset$ $\implies$
$x\in W^s(\overline{0})=\{y\in\Sigma_2^+\,;\,\lim_{n\to+\infty} d(\sigma^n(y),\overline{0})=0\}$, we get that $$\dd_{\NN}^+(\cu(x))=\mu(C_+(1))>0$$ for $\mu$ almost every $x\in\Sigma_2^+$ and every ergodic $\sigma$ invariant probability $\mu\ne\delta_{\overline{0}}$.

On the other hand, it is immediate that $R$ is an exact induced time and so, a coherent one.
Hence, $F(x)=\sigma^{R}(x)$ is orbit-coherent (Lemma~\ref{LemmaOrbCoh1}) and $\cu:\Sigma_2^+\to2^{\NN}$ is a coherent schedule of events by Lemma~\ref{LemmaCoerenteToCoerente} (we may assume that $\mu(\per(\sigma))=0$, since the case when $\mu=\frac{1}{n}\sum_{j=0}^{n-1}\delta_{\sigma^j(p)}$ is trivial for $\sigma^n(p)=p$).
 So, Lemma~\ref{Lemma535dcv} below follows straightforward form     Corollary~\ref{CorollaryPlissBlock}.

\begin{Lemma}\label{Lemma535dcv}  Let $F:\Sigma_2^+\setminus\{\overline{0}\}\to\Sigma_2^+$ be the induced map $F(x)=\sigma^{R(x)}(x)$, where $R$ is given by (\ref{EquatioInduchgf}).
	Every ergodic $\sigma$ invariant probability $\mu\ne\delta_{\overline{0}}$ is $F$-liftable.
\end{Lemma}

Let us recall the definition of full Markov map and mass distribution. Let $\XX$ be a metric space and $\{U_n\}_{n\in\NN}$ a countable collection  of two by two disjoint open sets.
A map $f:\bigcup_{n}U_n\to\XX$ is called a {\bf\em full Markov map} if
\begin{enumerate}
	\item $f|_{U_n}$ is a homeomorphism between $U_n$ and $\XX$ for every $n\ge1$;
	\item $\lim_n\diam(\cp_n(x))=0$ for every $x\in\bigcap_{j}f^{-j}(\XX)$,
\end{enumerate}
where $\cp=\{U_n\}$, $\cp_n=\bigvee_{j=0}^{n-1}f^{-j}\cp$ and $\cp_n(x)$ is the element of $\cp_n$ containing $x$.
A {\bf\em mass distribution} on $\{U_n\}_{n}$ is a map $m:\{U_n\}_{n}\to[0,1]$ such that $\sum_{n\in\NN}m(U_n)=1$.
The {\bf\em $f$-invariant probability $\mu$ generated by the mass distribution} $m$ is the ergodic $f$ invariant probability $\mu$ defined by $$\mu(U_{k_1}\cap f^{-1}(U_{k_2})\cap\cdots\cap f^{n-1}(U_{k_n}))=m(U_{k_1}) m(U_{k_2})\cdots m(U_{k_n}),$$ where $U_{k_1}\cap f^{-1}(U_{k_2})\cap\cdots\cap f^{n-1}(U_{k_n})\in\bigvee_{j=0}^{n-1}f^{-j}(\cp)$.

Proceeding with the proof of Lemma~\ref{Lemma222jhf76h3}, let  $\zeta(s)=\sum_{n=1}^{\infty}n^{-s}$ be the Riemann zeta function, $\alpha\in(0,1)$ and, for any given $\ell\in\NN$, consider the mass distribution $$m_{\ell}(P_j)=\begin{cases}
	2^{-j} & \text{ if }j\le\ell\\
	\frac{2^{-\ell}/\zeta(2+\alpha)}{(j-\ell)^{2+\alpha}} & \text{ if }j>\ell\end{cases},$$
where $P_1=C_+(1)$ and, for $j\ge2$, $P_j=C_+((\underbrace{0,\cdots,0}_{j-1\,times},1))$.
Note that $\sum_{j=1}^{+\infty}m_{\ell}(P_j)=1$, 
$$\sum_{j>\ell}j\,m_{\ell}(P_j)=\frac{1}{2^{\ell}\zeta(2+\alpha)}\sum_{j>\ell}\frac{j}{(j-\ell)^{2+\alpha}}=\frac{1}{2^{\ell}\zeta(2+\alpha)}\sum_{j=1}\frac{j+\ell}{j^{2+\alpha}}\le$$
$$\le\frac{\ell+1}{2^{\ell}\zeta(2+\alpha)}\sum_{j=1}\frac{1}{j^{1+\alpha}}=\frac{\ell+1}{2^{\ell}}\;\frac{\zeta(1+\alpha)}{\zeta(2+\alpha)}$$
and
$$\sum_{j>\ell}m_{\ell}(P_j)\log(1/m_{\ell}(P_j))=\sum_{j>\ell}\frac{\log(2^{\ell}\zeta(2+\alpha)(j-\ell)^{2+\alpha})}{2^{\ell}\zeta(2+\alpha)(j-\ell)^{2+\alpha}}=$$
$$=\frac{1}{2^{\ell}\zeta(2+\alpha)}\sum_{j=1}^{+\infty}\frac{\log(2^{\ell}\zeta(2+\alpha)j^{2+\alpha})}{j^{2+\alpha}}=$$
$$=\frac{\log(2^{\ell}\zeta(2+\alpha))}{2^{\ell}\zeta(2+\alpha)}\sum_{j=1}^{+\infty}\frac{1}{j^{2+\alpha}}+\frac{1}{2^{\ell}\zeta(2+\alpha)}\sum_{j=1}^{+\infty}\frac{(2+\alpha)\log(j)}{j^{2+\alpha}}=$$
$$=\frac{\log(2^{\ell}\zeta(2+\alpha))}{2^{\ell}}+\frac{2+\alpha}{2^{\ell}\zeta(2+\alpha)}\sum_{j=1}^{+\infty}\frac{\log(j)}{j^{2+\alpha}}\le\frac{\log(2^{\ell}\zeta(2+\alpha))}{2^{\ell}}+\frac{(2+\alpha)\zeta(1+\alpha)}{2^{\ell}\zeta(2+\alpha)}.$$
Hence, there is $C=C(\alpha)>0$ such that $$\sum_{j>\ell}j\, m_{\ell}(P_j)\;\;\text{ and }\;\;\sum_{j>\ell}m_{\ell}(P_j)\log(1/m_{\ell}(P_j))\;\le\; C\frac{\ell}{2^{\ell}}$$
In particular, $\sum_{j=1}^{+\infty}j m_{\ell}(P_j)\le C+2$ and $\sum_{j=1}^{+\infty}m_{\ell}(P_j)\log\frac{1}{m_{\ell}(P_j)}$ $\le$ $C$ $+$ $(\log2)\sum_{j=1}^{+\infty}\frac{j}{2^j}$ $=$ $C+2\log2$, $\forall\,\ell\ge 1$ (recall that $\sum_{j=1}^{\infty}j\,2^{-j}=2$).
Moreover, $$\lim_{\ell\to+\infty}\frac{\sum_{j=1}^{+\infty}m_{\ell}(P_j)\log(1/m_{\ell}(P_j))}{\sum_{j=1}^{+\infty}j\,m_{\ell}(P_j)}=\frac{\lim_{\ell}\sum_{j=1}^{+\infty}m_{\ell}(P_j)\log(1/m_{\ell}(P_j))}{\lim_{\ell}\sum_{j=1}^{+\infty}j\,m_{\ell}(P_j)}=\frac{2\log2}{2}=\log2.$$
Given $\varepsilon>0$, let $\ell\ge1$ be big enough so that $$\frac{\sum_{j=1}^{+\infty}m_{\ell}(P_j)\log(1/m_{\ell}(P_j))}{\sum_{j=1}^{+\infty}j\,m_{\ell}(P_j)}>\log2-\varepsilon.$$
Thus, taking $\nu_{\alpha,\ell}$ as the ergodic $F$-invariant probability generated by the mass distribution $m_{\ell}$, we get that $h_{\nu_{\alpha,\ell}}(F)=\sum_{j=1}^{+\infty}m_{\ell}(P_j)\log(1/m_{\ell}(P_j))\le C+2\log2$ and $\int R d\nu_{\alpha,\ell}=\sum_{j=1}^{+\infty}j m_{\ell}(P_j)\le C+2$ and so, 
$$\frac{h_{\nu_{\alpha,\ell}}(F)}{\int R d\nu_{\alpha,\ell}}>\log2-\varepsilon.$$
It follows from Lemma~\ref{FolkloreResultB} that $$\mu_{\alpha,\ell}=\frac{1}{\int R d\nu_{\alpha,\ell}}\sum_{j\ge0}\sigma^j_*(\nu_{\alpha,\ell}|_{\{R>j\}})$$ is an ergodic $\sigma$ invariant probability.
As $\nu_{\alpha,\ell}(P_j)>0$ for every $j\ge1$, we have that $\supp\nu_{\alpha,\ell}=\Sigma_2^+$ and, as a consequence, $\supp\mu_{\alpha,\ell}=\Sigma_2^+$.
As $\int R d\nu_{\alpha,\ell}<+\infty$, it follows from the ``generalized Abramov formula'' \cite{Zw} that $h_{\mu_{\alpha,\ell}}(\sigma)=\frac{h_{\nu_{\alpha,\ell}}(F)}{\int R d\nu_{\alpha,\ell}}>\log 2-\varepsilon$.
Finally, it follows from Lemma~\ref{LemmaExacRtoRR} that $$\int R d\mu_{\alpha,\ell}\ge\frac{1}{2(2+C)}\int (R)^2d\nu_{\alpha,\ell}>\frac{1}{2(2+C)}\sum_{j>\ell}j^2\,m_{\ell}(P_j)=$$
$$=\frac{1}{2^{\ell+1}(2+C)\zeta(2+\alpha)}\sum_{j>\ell}\frac{j^2}{(j-\ell)^{2+\alpha}}=$$ 
$$=\frac{1}{2^{\ell+1}(2+C)\zeta(2+\alpha)}\sum_{j=1}\frac{(j+\ell)^2}{j^{2+\alpha}}\ge\frac{1}{2^{\ell+1}(2+C)\zeta(2+\alpha)}\sum_{j=1}\frac{1}{j^{\alpha}}=+\infty.$$
Note that, if $0<\alpha_0<\alpha_1<1$, then $\nu_{\alpha_0,\ell}(P_j)\ne\nu_{\alpha_1,\ell}(P)j)$ for any $j>\ell$.
In particular, $\nu_{\alpha_0,\ell}\ne\nu_{\alpha_1,\ell}$.
As $F$ is orbit-coherent, it follows from Theorem~\ref{TheoremLift} that $\nu_{\alpha_{_0},\ell}$ is the unique $F$-lift of $\mu_{\alpha_{_0},\ell}$. Therefore, $\mu_{\alpha_{_0},\ell}\ne\mu_{\alpha_{_1},\ell}$. This implies that $M=\{\mu_{\alpha,\ell}\,;\,\alpha\in(0,1)\text{ and }\ell\ge1\}$ is uncountable and it concludes the proof,  since we also have that $\sup\{h_{\mu}(\sigma)\,;\,\mu\in M\}=\log 2$.
\end{proof}

\begin{Lemma}\label{LemmaMedTeoOLD}
Let $(\XX,d)$ be a metric space and $\cm(\XX)$ the set of all finite Borel measures on $\XX$. If $\mu_n\in\cm(\XX)$ converges to $\mu\in\cm(\XX)$, then $\mu(S)\le\liminf_{n\to\infty}\mu_n(S)$ for every open set $S$.
	\end{Lemma}
	\begin{proof} As $\mu$ is a Borel finite measure on a  metric space, it follows from Proposition~\ref{PropositionAllBorelProbIsReg}\footnote{\begin{Proposition}[See Proposition~A.3.2. in \cite{OV}]\label{PropositionAllBorelProbIsReg}
If $(\XX,d)$ is a metric space then every Borel probability $\mu$ on $\XX$ is regular, i.e., 	given a Borel set $B$ and $\varepsilon>0$ there are a closed set $C\subset B$ and an open set $A\supset B$ such that $\mu(A\setminus C)<\varepsilon$. 
\end{Proposition}
} that $\mu$ is a regular measure.
	Thus, given an open set $S$ and $\varepsilon>0$, there is a closed set $K_{\varepsilon}\subset S$ such that $\mu(K_{\varepsilon})>\mu(S)-\varepsilon$.
By Urysohn' Lemma\footnote{\begin{Lemma}[Urysohn]
Suppose that $(\XX,d)$ is a metric space. If $A$ and $B$ are closed sets with $A\cap B=\emptyset$ then there is a continuous function $\phi:\XX\to[0,1]$ such that $\phi(A)=0$ and $\phi(B)=1$ and $\phi(\XX\setminus(A\cup B))\subset(0,1)$.
\end{Lemma}}, there is a continuous function $\varphi_{\varepsilon}:\XX\to[0,1]$ such that $\varphi_{\varepsilon}(K_{\varepsilon})=1$ and 
	$\varphi_{\varepsilon}(\XX\setminus S)=0$.
	As $\mu(K_{\varepsilon})\le\int \varphi_{\varepsilon}d\mu=\lim_{n}\int \varphi_{\varepsilon}d\mu_n\le\liminf_n\mu_n(S)$, we can take $\varepsilon\to0$ and so $\mu(S)\le\liminf_n\mu_n(S)$, proving the lemma.
\end{proof}

\begin{Lemma}[Continuity at the empty set]\label{LemmadeMEDIDA} Let $(\XX,\mathfrak{A},\mu)$ be a probability space.
	If $A_1\supset A_2\supset A_3\supset\cdots$ is a sequence of measurable sets, then $\mu(\bigcap_{n\ge1}A_n)=\lim_n\mu(A_n)$.
\end{Lemma}
\begin{proof} Set $r=\lim_n\mu(A_n)$,  $A_0=\emptyset$ and $L=\bigcap_{n\ge1}A_n$. Let $\triangle_n=A_n\setminus A_{n-1}$ and $a_n=\mu(\triangle_n)$ for $n\ge1$.
As $\triangle_i\cap\triangle_j=\emptyset$ for $i\ne j$, we get that 
 $\sum_{j=0}^{\infty}a_j\le1$ is a convergent series of nonnegative numbers. So, $\sum_{j=n}^{\infty}a_j\to0$ when $n\to\infty$.
 As $r\le \mu(L)+\sum_{j=n}^{\infty}a_j=\mu(A_n)\to r$, we get that $\mu(L)=r$.
\end{proof}

\begin{Lemma}
\label{LemmaTeoErg876y7ui} 
Let $f$ be a measure-preserving automorphism of a probability space $(\XX,\mathfrak{A},\mu)$. If $f$ is bimeasurable then $\mu\big(\bigcap_{n\ge0}f^j\big(\bigcap_{j\ge0}f^{-j}(\XX)\big)\big)=1$.
\end{Lemma}
\begin{proof}
Let $A_0=\bigcap_{n\ge0}f^{-n}(\XX)$. As $\mu(f^{-n}(\XX))=1$ for every $n$, we get $\mu(A_0)=1$.
Moreover, if $U\in\mathfrak{A}$ has $\mu(U)=1$ then $f(U)\in\mathfrak{A}$ and so  
$0=\mu(\XX\setminus U)\ge\mu(\XX\setminus f^{-1}(f(U)))=\mu(f^{-1}(\XX\setminus f(U)))=\mu(\XX\setminus f(U))$. Thus, $\mu(f(U))=1$.
Therefore, as $\mu(A_0)=1$ and $A_0\in\mathfrak{A}$, we get that $\mu(f(A_0))=1$ and inductively that $\mu(f^j(A_0))=1$ for every $j\ge0$. 
So, $\mu(\bigcap_{j\ge0}f^j(A_0))=1$.
\end{proof}

\begin{proof}[\em\bf Proof of Lemma~\ref{FolkloreResultA}] Given $x\in U$, let $\cu(x)=\{j\ge1\,;\,f^j(x)\in U\}$ and set $R:\XX\to\NN\cup\{+\infty\}$ by
$$R(x)=
\begin{cases}
\min\cu(x) & \text{ if }\cu(x)\ne\emptyset\\
+\infty & \text{ if }\cu(x)=\emptyset
\end{cases}.$$
Note that $R$ is the first return time to $U$ and $R|_{A}$ is the induced time of $F$.

Given $V\subset U$, with $0<\mu(V)<\mu(U)$, let $V_0=V$ and $V_n:=f^{-1}(V_{n-1})\setminus U$ $\forall n\ge1$.
Note that $f(V_n)\subset V_{n-1}$ $\forall n\ge1$. Thus $f^{n-j}(V_n)\subset V_{n-j}$ $\forall0\le j\le n$.
Moreover, as $V_{j}\cap U=\emptyset$ $\forall\,j\ge1$ (because $V_j=f^{-1}(V_{j-1})\setminus U$), we get $n=\min\{j\ge0\,;\,f^j(x)\in U\}$ for every $x\in V_n$.
This implies that $$f^{-1}(V_{n-1})\cap U=f^{-n}(V)\cap\{R=n\}\,\,\forall n\ge1.$$
As $\mu$ is ergodic, we get that $\mu(\bigcap_{j=0}^{n-1}f^{-j}(\XX\setminus U))=1$ and so, $$\mu(V_n)\le\mu\bigg(\bigcap_{j=0}^{n-1}f^{-j}(\XX\setminus U)\bigg)\to0.$$

 On the other hand, $$\mu(V)=\mu(f^{-1}(V))=\mu(f^{-1}(V)\cap U)+\mu(f^{-1}(V)\setminus U)=$$ $$=\mu(f^{-1}(V)\cap\{R=1\})+\mu(V_1)=\mu(f^{-1}(V)\cap\{R=1\})+\mu(f^{-1}(V_1))=$$ $$=\mu(f^{-1}(V)\cap\{R=1\})+\mu(f^{-1}(V_1)\cap U)+\mu(f^{-1}(V_1)\setminus U)=$$ $$=\mu(f^{-1}(V)\cap\{R=1\})+\mu(f^{-2}(V)\cap\{R=2\})+\mu(V_2)=$$
 $$\cdots=\bigg(\sum_{j=1}^{n}\mu(f^{-j}(V)\cap\{R=j\})\bigg)+\mu(V_n).$$
 
As $f^{-j}(V)\cap\{R=j\}=\big(f^{j}|_{\{R=j\}}\big)^{-1}(V)=F|_{\{R=j\}}^{-1}(V)$ $\forall\,j$, $F^{-1}(V)=\sum_{n\ge1}F|_{\{R=n\}}^{-1}(V)$ and $\mu(\{R=+\infty\})=0$ (where $\sum_{j}U_j$ means a union of disjoints sets), we get
$$\mu(F^{-1}(V))=\mu\bigg(\sum_{j\ge1}F|_{\{R=j\}}^{-1}(V)\bigg)=\sum_{j=1}^{\infty}\mu\big(f^{-j}(V)\cap\{R=j\}\big)=$$
$$=\lim_{n\to\infty}\bigg(\bigg(\sum_{j=1}^{n}\mu\big(f^{-j}(V)\cap\{R=j\}\big)\bigg)+\mu(V_n)\bigg)=\mu(V),$$ proving that $\mu|_U$ is $F$ invariant.

As $R$ is a first return time, it is an exact induced time and so, a coherent one. Thus, by Lemma~\ref{LemmaOrbCoh1}, $F$ is orbit-coherent.
Finally, as $\mu$ is $f$-ergodic and $F$ is orbit coherent, it follows from Proposition~\ref{PropositionErGOS} that $\mu$ is $F$-ergodic and, 
by Lemma~\ref{LemmaSingSing}, we can conclude that $\mu|_{U}$, as well as $\frac{1}{\mu(U)}\mu|_{U}$, is $F$-ergodic.
\end{proof}

\begin{proof}[\em\bf Proof of Proposition~\ref{PropRokhlin}]
If $A\subset\XX_f$ then $f^{-1}(\pi_n(A))\supset\pi_{n+1}(A)$, $\forall\,n\ge0$, and so, $0\le\mu(\pi_n(A))\le1$ is a monotonous non-increasing sequence for every measurable $A$.
	Thus, $\overline{\mu}(A):=\lim_{n\to\infty}\mu(\pi_n(A))$ is well defined whenever $\pi_n(A)$ is a measurable subset of $\XX$ for every $n\ge0$.
	
	Let $A=[A_0,\cdots,A_n]$ and $B=[B_0,\cdots,B_{\ell}]$ be two cylinders such that $A\cap B=\emptyset$. 
In this case, $\pi_j(A)=f^{-(j-n)}(A_n)$ $\forall\,j\ge n$ and $\pi_j(B)=f^{-(j-\ell)}(B_\ell)$ $\forall\,j\ge\ell$. Thus, $\overline{\mu}(A)=\mu(A_n)$ and $\overline{\mu}(B)=\mu(B_\ell)$.
Without loss of generality, we may assume that $\ell\le n$.
As $A\cap B=\emptyset$, we get that $A_n\cap f^{-(n-\ell)}(B_\ell)=\emptyset$ and so, $$\lim_{j}\mu(\pi_j(A\cup B))=\lim_j\mu\bigg(f^{-(j-n)}(\pi_n(A))\cup f^{-(j-\ell)}(\pi_\ell(B))\bigg)=$$
	$$=\lim_j\mu\bigg(f^{-(j-n)}\bigg(\;\pi_n(A)\cup f^{-(n-\ell)}(\pi_\ell(B))\bigg)\bigg)=\mu\big(\pi_n(A)\cup f^{-(n-\ell)}(\pi_\ell(B))\big)=$$
	$$=\mu(\pi_n(A))+\mu(\pi_\ell(B))=\overline{\mu}(A)+\overline{\mu}(B).$$
As the set of all cylinders generates the $\sigma$-algebra of $\XX_f$, we conclude that $\mu$ is a well defined probability on $\XX_f$. Similarly one can show that $\overline{\mu}$ is $\sigma$-additive.
 
Note that, if $A=[A_0,\cdots,A_n]$ is a cylinder, then \begin{equation}\label{EqPreNatExt}
  \overline{f}^{-1}(A)=[f^{-1}(A_0)\cap A_1,A_2,\cdots,A_{n}].
\end{equation}
Indeed, $\overline{x}\in \overline{f}^{-1}([A_0,\cdots,A_n])$ $\iff$ $\overline{f}(\overline{x})=(f(\overline{x}(0)),\overline{x}(0),\overline{x}(1),\overline{x}(2),\cdots)$ $\in$ $[A_0,\cdots,A_n]$ $\iff$ $(f(\overline{x}(0)),\overline{x}(0),\overline{x}(1),\cdots,\overline{x}(n-1))$ $\in$ $ A_0\times A_1\times A_2\times \cdots\times A_{n}$ $\iff \overline{x}(0)\in f^{-1}(A_0)\cap A_1,\,\overline{x}(1)\in A_2,\,\overline{x}(2)\in A_3,\,\cdots,\,\overline{x}(n-1)\in A_n$ $\iff$
$\overline{x}\in[f^{-1}(A_0)\cap A_1,A_2,\cdots, A_n]$.

As a consequence of (\ref{Eqesfr4}), (\ref{EqPreNatExt}) and the invariance of $\mu$, we get that $\overline{\mu}\big(\overline{f}^{-1}(A)\big)=\mu(A_n)=\overline{\mu}(A)$ for every cylinder $A=[A_0,\cdots,A_n]$, proving that $\overline{\mu}$ is $\overline{f}$ invariant. Furthermore, it is easy to check that $\mu=\pi_*\overline{\mu}=\overline{\mu}\circ\pi^{-1}$.

Applying (\ref{EqPreNatExt}) recursively, we get that 
$$\overline{f}^{-n}([A_0,\cdots, A_n])=[f^{-n}(A_0)\cap f^{-(n-1)}(A_1)\cap\cdots\cap A_n]$$
and this implies the unicity of the lift. Indeed, if $\mu=\pi_*\overline{\nu}$ for a $\overline{f}$-invariant probability $\overline{\nu}$ then $$\overline{\nu}([A_0,\cdots,A_n])=\overline{\nu}(\overline{f}^{-n}([A_0,\cdots,A_n])=\overline{\nu}([f^{-n}(A_0)\cap\cdots\cap A_n])=$$
$$=\nu(\pi^{-1}(f^{-n}(A_0)\cap\cdots\cap A_n))=\mu(f^{-n}(A_0)\cap\cdots\cap A_n)=\overline{\mu}(\pi^{-1}(f^{-n}(A_0)\cap\cdots\cap A_n))=$$
$$=\overline{\mu}([f^{-n}(A_0)\cap\cdots\cap A_n])=\overline{\mu}(\overline{f}^{-n}([A_0,\cdots,A_n]))=\overline{\mu}([A_0,\cdots,A_n]).$$

Now we will verify that $\overline{\mu}$ is ergodic with respect to $\overline{f}$. 
Suppose that $A$ is a $\overline{f}$-invariant measurable set with $\overline{\mu}(A)>0$.
As $\overline{f}^{-1}(A)=A$, we get that $\pi_n(\overline{f}^{-1}(A))=\pi_n(A)$ $\forall\,n\ge0$. So, $\pi_n(A)=\pi_{n+1}(A)$ for every $n\ge0$,
as $$\overline{f}^{-1}(A)=\{\overline{x}\in\XX_f\,;\,\overline{x}(n)\in\pi_{n+1}(A)\;\forall\,n\ge0\}.$$
Thus, $\overline{\mu}(A)=\mu(\pi(A))=\mu(\pi_n(A))$ $\forall\,n\ge0$. Moreover, using that $\pi_0(A)=f(\pi_1(A))$, we have that $\pi(A)=f(\pi(A))$.
This implies that $U:=\cup_{j\ge0}f^{-j}(\pi(A))$ is an invariant set, i.e., $f^{-1}(U)=(U)$.
As $\mu(U)\ge\mu(\pi(A))=\overline{\mu}(A)>0$, it follows from the ergodicity of $\mu$ that $\mu(U)=1$.
On the other hand, as $\pi(A)\subset f^{-1}(\pi(A))\subset f^{-2}(\pi(A))\subset\cdots\subset f^{-j}(\pi(A))\nearrow U$ and $\mu(f^{-j}(\pi(A)))=\mu(\pi(A))$, we conclude that $\mu(U)=\mu(\pi(A))$.
Thus, $\overline{\mu}(A)=\mu(\pi(A))=1$, proving the ergodicity of $\overline{\mu}$. 
\end{proof}

\begin{proof}[\em\bf Proof of Lemma~\ref{PlissLemma}]
	The proof below follows the proof of Lemma~11.8 in \cite{ManeLivro} and  Lemma~3.1 in \cite{ABV}.
	
Let, for $1\le j\le n$, $s(j)=a_j-j c_0$ and set $s(0)=0$.
Let $1< n_1< n_2< \cdots n_{\ell}\le n$ be the maximal sequence such that $s(j)\le s(n_i)$ for every $0\le j<n_i$ and $1\le i\le \ell$.
As $s(n)=c_1 n>0=s(0)$, we get that $\ell\ge1$.
By definition, $s(n_i)\ge s(j)$ for $0\le j<n_i$ and so, $a_{n_i}-n_i c_0\ge a_j-j c_0$.
That is, $a_{n_i}-a_j\ge (n_i-j)c_0$ 
for every $0\le j<n_i$ and $1\le i\le \ell$. As $a_j$ is subadditive, we have that $a_{n_i}=a_{n_i-j+j}\le a_{n_i-j}+a_j$ and so, $a_{n_i}-a_j\le a_{n_i-j}$. Hence, $$ a_{n_i-j}\ge (n_i-j)c_0,$$
for every $0\le j<n_i$ and $1\le i\le \ell$.

Therefore, we only need to show that $\ell\ge\theta n$. It follows from the maximality of the sequence $n_1,\cdots,n_{\ell}$ that $s(n_i-1)<s(n_{i-1})$ for every $1\le i\le \ell$. Thus, $s(n_i)<s(n_{i-1})+(C-c_0)$ for every $1\le i\le \ell$. Indeed, $s(n_i)=a_{n_i}-n_i c_0\le a_{n_i-1}+a_1-n_i c_0=a_{n_i-1}-(n_i-1)c_0+(a_1-c_0)=s(n_i-1)+(a_1-c_0)<s(n_{i-1})+(a_1-c_0)\le s(n_{i-1})+(C-c_0)$.

By the maximality of $n_1,\cdots,n_{\ell}$, we get that $s(n_{\ell})\ge s(n)=a_n-n c_0\ge n c_1-n c_0=(c_1-c_0)n$. Hence, $n(c_1-c_0)\le s(n_{\ell})=(s(n_{\ell})-s(n_{\ell-1}))+(s(n_{\ell-1})-s(n_{\ell-2}))+\cdots((s(n_{2})-s(n_{1}))+s(n_1)$ $\le$ $(C-c_0)(\ell-1)+s(n_1)$. Noting that $s(n_1)=a_{n_1}-c_0
\begin{cases}
	\le C-c_0 & \text{ if }n_1=1\\
	<C-c_0 & \text{ if }n_1>1
\end{cases}
$,
we get that $$n(c_1-c_0)\le s(n_{\ell})\le(C-c_0)\ell.$$
That is, $\ell\ge\frac{c_1-c_0}{C-c_0}n$, which concludes the proof.
\end{proof}

\begin{Lemma}[Corollary of Pliss Lemma, see a proof in Appendix]\label{CorollaryOfPlissLemma}
	Given $-C\le c_0<c_1\le C$, let $0<\theta=\frac{c_1-c_0}{C-c_0}<1$.
	Let $a_1,\cdots,a_{n}\in[-C, C]$ be a subadditive sequence of real numbers.
	If $\frac{1}{n}a_n\ge c_1$, then there is $\ell\ge\theta n$ and a sequence $1<n_1<n_2<\cdots<a_{\ell}$ such that $\frac{1}{n_j-k}\,a_{n_j-k}\ge c_0$ for every $1\le j\le\ell$ and $0\le k<n_j$.
\end{Lemma}
\begin{proof}
	Taking $b_j=a_j+2Cj$, we get that $\frac{1}{n}b_n=\frac{1}{n}a_n+2C\ge c_1+2C$.
	As $0<C\le c_0+2C<c_1+2C<3C$, it follows from Lemma~\ref{PlissLemma} that, taking $\theta=\frac{(c_1+2C)-(c_0+2C)}{3C-(c_0+2C)}=\frac{c_1-c_0}{C-c_0}$, there is $\ell\ge\theta n$ and a sequence $1<n_1<n_2<\cdots<a_{\ell}$ such that $\frac{1}{n_j-k}b_{n_j-k}\ge c_0+2C$ for every $1\le j\le\ell$ and $0\le k<n_j$.
	As $\frac{1}{n_j-k}b_{n_j-k}=2C+\frac{1}{n_j-k}a_{n_j-k}$, we get that $\frac{1}{n_j-k}a_{n_j-k}\ge c_0$ for every $1\le j\le\ell$ and $0\le k<n_j$.
\end{proof}



\begin{thebibliography}{10}

\bibitem[Aa]{Aa} Aaronson, J.. {\em An introduction to infinite ergodic theory}. Math.
Surv. Monographs 50, AMS, Providence R.I. US. (1997).

\bibitem[ABV]{ABV} Alves, J.F., Bonatti, C., Viana, M.: {\em SRB measures for partially hyperbolic systems whose central direction is mostly expanding}. Invent. Math. 140, 351-398 (2000).


\bibitem[ALP04]{ALP1} J. F. Alves, S. Luzzatto, V. Pinheiro, {\em Lyapunov exponents and rates of mixing for one-dimensional maps}, Ergodic Theory And Dynamical Systems, {\bf 24} , 1-22. (2004).

\bibitem[ALP05]{ALP} 	Alves, J.F., Luzzatto, S., Pinheiro, V.. {\em Markov structures and decay of correlations for non-uniformly expanding dynamical systems}. Annales de l'Institut Henri Poincaré. Analyse Non Linéaire. 22, 817-839 (2005).

\bibitem[BDV]{BDV} Bonatti, C.; Diaz, L.; Viana, M..{\em Dynamics Beyond Uniform Hyperbolicity: A Global Geometric and Probabilistic Perspective}. EMS, volume 102. Springer. ISBN
978-3-540-22066-4 (2005).

\bibitem[Ca]{Cas} Castro, A.: {\em Backward inducing and exponential decay of correlations for partially hyperbolic attractors}. Israel Journal of Mathematics. 130, 29-75 (2002).

\bibitem[COP]{COP} Castro, A., Oliveira, K., Pinheiro, V.: Shadowing by non-uniformly hyperbolic periodic points and uniform hyperbolicity. Nonlinearity. 20, 75-85 (2006).



\bibitem[CP]{CP} Climenhaga V.; Pesin, Y.. {\em Building Thermodynamics for Non-uniformly Hyperbolic Maps.} Arnold Math J. 3:37-82 DOI 10.1007/s40598-016-0052-8. (2017).

\bibitem[CLP]{CLP} Climenhaga, V; Luzzatto, S.; Pesin. Y..{\em SRB measures and Young towers for surface diffeomorphisms}. Preprint. (2019).

\bibitem[Do]{TD} Downarowicz, T.. {\em Entropy in Dynamical Systems} (New Mathematical Monographs, Vol. 18). Cambridge University Press. (2011).



\bibitem[FK]{FK} Furstenberg, H.; Kesten, H.. {\em  Products of Random Matrices}. The Annals of Mathematical Statistics. 31, 457-469 (1960).

\bibitem[Le]{Le1979} Ledrappier, F.. {\em A variational principle for topological entropy}. Ergodic theory, Proc., Oberwolfach. Lec. Notes Math. 729, Springer. pp. 78-88. (1978).


\bibitem[Man]{ManeLivro}Mañe, R.. Ergodic theory and differentiable dynamics. Springer-Verlag. (1987).



\bibitem[Mau]{M81} Mauldin, R.. {\em Bimeasurable Functions}. Proceedinngs of the American Mathematical Society. 83, 369-370 (1981).



\bibitem[MS]{dMvS} de Melo, W.; van Strien, S.. {\em One Dimensional Dynamics}, Springer-Verlag, (1993).


\bibitem[Ol]{KO} Oliveira, K.. {\em Every Expanding Measure has the nonuniform specification property}. Proceedings of the American Mathematical Society. 140, 1309-1320 (2012).

\bibitem[OV]{OV} Oliveira, K.; Viana, M.. {\em Foundations of Ergodic Theory}. Cambridge University Press, (2016).

\bibitem[PSZ]{PSZ} Pesin, Y., Senti, S., Zhang, K.: Thermodynamics of towers of hyperbolic type. Trans. Amer. Math. Soc. Trans. Amer. Math. Soc. 368, 8519-8552. (2016). 

\bibitem[PS]{PS} Pesin, Y., Senti, S.: Equilibrium measures for maps with inducing schemes. J. Mod. Dyn. 2(3), 397-430. (2008).

\bibitem[Pe]{Pe} Petersen, K.. Ergodic Theory. Cambridge Studies in Advanced Mathematics 2. Cambridge University Press. (1990).

\bibitem[Pl]{Pl} Pliss, V.. {\em On a conjecture due to Smale}. Diff. Uravnenija, 8:262-268, (1972).


\bibitem[Pi06]{Pi06}	Pinheiro, V.. {\em Sinai-Ruelle-Bowen measures for weakly expanding maps}. Nonlinearity. 19, 1185-1200 (2006).

\bibitem[Pi11]{Pi11} Pinheiro, V.. {\em Expanding Measures}. Annales de l Institut Henri Poincaré. Analyse non Linéaire, v. 28, p. 889-939, (2011).


\bibitem[PuSh]{PuS} Pugh, C.; Shub, M.. {\em Ergodic Attractors}. Transactions of the American Mathematical Society. 312, 1-54 (1989).

\bibitem[Pur]{Pu} Purves, R.. {\em  Bimeasurable functions}. Fundamenta Mathematicae. 149-157 (1966).



\bibitem[Ro]{Ro} Rokhlin, V.. {\em Exact endomorphisms of a Lebesgue space}. Am. Math. Soc.,Transl., II. Ser., 39, 1-36. (1964).

\bibitem[Sa]{SS} Sagitov, S.. {\em Weak Convergence of Probability Measures}, (2015).

\bibitem[vSV]{vSV} van Strien, S.; Vargas, E.. {\em Real bounds, ergodicity and negative Schwarzian for multimodal maps}, J. Am. Math. Soc. {\bf 17}, 749-782. (2004).

\bibitem[Vi]{Vi98} Viana, M. {\em  Dynamics: a probabilistic and geometric perspective}. Proc. Int. Congress of Mathematicians
(Berlin, Germany) vol 1, pp r. (1998).


\bibitem[Wa]{Wa} Walters, P.. An Introduction to Ergodic Theory.  Springer-Verlag. (1982).


\bibitem[Zw]{Zw} Zweimüller, R.. {\em Invariant measures for general(ized) induced transformations}. Proceedinngs of the American Mathematical Society. 133, 2283-2295. (2005).


\end{thebibliography}
\end{document}